\documentclass[11pt]{amsart}

\usepackage{amssymb,amsthm,amsmath}
\usepackage[numbers,sort&compress]{natbib}
\usepackage{color}
\usepackage{graphicx}
\usepackage{tikz}
\usepackage[colorlinks,linkcolor=blue,anchorcolor=blue,citecolor=blue]{hyperref}
\usepackage{comment}

\date{}

\title[Global rigidity results]{Global rigidity
for ultra-differentiable quasiperiodic  cocycles and its spectral applications}

\author{Hongyu Cheng}
\address{
Chern Institute of Mathematics and LPMC, Nankai University, Tianjin 300071, China} \email{hychengmath@126.com}

\author{Lingrui Ge}
\address{Department of Mathematics, University of California Irvine, CA, 92697-3875, USA}
\email{lingruig@uci.edu}

\author {Jiangong You}
\address{
Chern Institute of Mathematics and LPMC, Nankai University, Tianjin 300071, China} \email{jyou@nankai.edu.cn}

\author{Qi Zhou}
\address{
Chern Institute of Mathematics and LPMC, Nankai University, Tianjin 300071, China
}

 \email{qizhou@nankai.edu.cn}

\newtheorem{theorem}{Theorem}[section]
\newtheorem{corollary}{Corollary}
\newtheorem{proposition}{Proposition}
\newtheorem{lemma}{Lemma}[section]
\theoremstyle{definition}
\newtheorem{remark}{Remark}[section]
\newtheorem{definition}{Definition}
\newtheorem{claim}{Claim}

\newcommand{\me}{\mathrm{e}}   
\newcommand{\mi}{\mathrm{i}}     
\newcommand{\dif}{\mathrm{d}}   
\newcommand{\ZZ}{\mathbb{Z}}

\newcommand{\RR}{\mathbb{R}}

\newcommand{\TT}{\mathbb{T}}
\newcommand{\mc}{\mathcal}

\numberwithin{equation}{section}

\newcommand{\N}{{\mathbb N}}
\newcommand{\Q}{{\mathbb Q}}
\newcommand{\R}{{\mathbb R}}
\newcommand{\T}{{\mathbb T}}

\newcommand{\Z}{{\mathbb Z}}

\begin{document}

\maketitle

\begin{abstract}
  For quasiperiodic  Schr\"odinger operators with one-frequency analytic potentials, from dynamical systems side, it has been proved that the corresponding quasiperiodic Schr\"odinger cocycle is either rotations reducible or has positive Lyapunov exponent for all irrational frequency  and almost every energy  \cite{AvilaF11}. From spectral theory side, the ``Schr\"odinger conjecture"  \cite{AvilaF11} and the ``Last's intersection spectrum conjecture" have been verified \cite{JitomirskayaM12}.  The proofs of above results crucially depend on the analyticity of the potentials. People are curious about if the analyticity is essential for those problems, see open problems by Fayad-Krikorian \cite{Krikorian09,Krikorian11} and Jitomirskaya-Marx \cite{JitomirskayaM12,Jitomirskayam17}. In this paper, we prove the above mentioned  results for
 ultra-differentiable potentials.
\end{abstract}

\section{Introduction and main results}
We consider smooth quasiperiodic $SL(2,\R)$ cocycles
\begin{equation*}
(\alpha,A):\mathbb{T}\times\mathbb{R}^{2}\rightarrow\mathbb{T}\times\mathbb{R}^{2},\ \ (\theta,w)\mapsto(\theta+\alpha,A(\theta)w),
\end{equation*}
where $\alpha\in \R\setminus\Q$, $A\in C^{\infty}(\T, SL(2,\R))$ and $\mathbb{T}:=\mathbb{R}/\mathbb{Z}$. Typical
examples are  Schr\"odinger cocycles where
\begin{equation*}
\begin{split}
A(\theta)=S_{E}^V(\theta)=\left(
\begin{array}{ll}
E-V(\theta) & -1\\
  1 & 0
\end{array}
\right),
\end{split}
\end{equation*}
which is equivalent to the eigenvalue equations of  the one-dimensional quasiperiodic Schr\"odinger operator
$H_{V,\alpha,\theta}$ defined by
\begin{equation}\label{20200824operator}
\begin{split}
(H_{V,\alpha,\theta}u)_{n}=u_{n-1}+u_{n+1}+V(\theta+n\alpha)u_{n}.
\end{split}
\end{equation}
Quasiperiodic Schr\"odinger operators describe the conductivity of electrons in a two-dimensional crystal layer subject to an external magnetic field of flux acting perpendicular to the lattice plane. Due to the rich backgrounds in quantum physics, quasiperiodic Schr\"odinger operators  have been extensively studied \cite{Last2005}.

It has been proved that the (almost) reducibility of the above Schr\"odinger cocycles is a powerful tool in the study of the spectral theory of quasiperiodic Schr\"odinger operators \cite{YouICM}.  Recall that $(\alpha,A)$
is $C^r$($r$ could be $\infty$ or $\omega$) reducible, if there  exist $B\in C^r(\T,PSL(2,\R))$ and $C\in SL(2,\R)$ such that
$$B(\cdot+\alpha)A(\cdot)B(\cdot)^{-1}=C.$$ We remark that the reducibility  is too restrictive since even an $\R$-valued cocycle is in general not reducible if the frequency is very Liouvillean. The appropriate notion is  $C^r$ rotations reducibility,
 which means, there  exist $B\in C^r(\T,PSL(2,\R))$ and $C\in C^r(\T,SO(2,\R))$ such that
$$B(\cdot+\alpha)A(\cdot)B(\cdot)^{-1}=C(\cdot).$$

By Kotani's theory, for Lebesgue almost every $E \in \mathbb{R}$,  the Schr\"odinger cocycles
$(\alpha, S_{E}^{V})$ is  either $L^2$ rotations reducible or  has positive Lyapunov exponent \footnote{ We refer  to Section \ref{qpcocycle} for the definitions and basic results.}.     In many circumstances,  especially for its dynamical and spectral applications, what's important is the rigidity, i.e.,  whether $L^2$ conjugacy  for analytic (resp. smooth) cocycles implies analytic (resp. smooth) conjugacy under some additional assumptions.
In this paper, we are interested in the  global rigidity results for smooth quasiperiodic Schr\"odinger cocycles and its spectral applications.

\subsection{Global rigidity results for smooth quasiperiodic cocycles}
 Based on the powerful method of renormalization, Avila-Krikorian
\cite{Krikorian06} proved that if $\alpha$ is recurrent Diophantine, $V\in C^{\omega}(\T,\R)$, then for
Lebesgue almost every $E,$ the Schr\"odinger cocycle
$(\alpha, S_{E}^{V})$ is either nonuniformly hyperbolic or $C^{\omega}$ reducible. Later, Fayad-Krikorian \cite{Krikorian09} proved that for all Diophantine $\alpha$ \footnote{Here $\alpha$ is {\it Diophantine} (denote $\alpha \in {\rm DC}(v,\tau)$),  if there exist $v>0$ and $\tau>1$ such that
\begin{equation*}
 \|n \alpha\|_{\Z}:=\inf_{j \in \Z}\left|  n \alpha  - j \right|
> \frac{v}{|n|^{\tau}},\quad \forall \  n\in\Z\backslash\{0\}.
\end{equation*}
We also denote ${\rm DC}:=\bigcup_{v>0,
\, \tau>1} {\rm DC}(v,\tau)$ the union.}, and $V\in C^{\infty}(\T,\R)$, then for
Lebesgue almost every $E,$ the Schr\"odinger cocycle
$(\alpha, S_{E}^{V})$ is either nonuniformly hyperbolic or $C^{\infty}$-reducible. Indeed,  it was pointed out by Fayad-Krikorian \cite{Krikorian09},
 to extend the results of \cite{Krikorian09} to any irrational number is an interesting and important problem.
The problem was later settled by Avila-Fayad-Krikorian \cite{AvilaF11} in the analytic case. More precisely, for all irrational $\alpha$ and any $V\in C^{\omega}(\T,\R)$,  Avila-Fayad-Krikorian \cite{AvilaF11} proved that, the Schr\"odinger cocycle  $(\alpha, S_{E}^{V})$   is  either $C^{\omega}$ rotations reducible or has positive Lyapunov exponent for
Lebesgue almost every $E$.
 Around 2011's, R. Krikorian \cite{Krikorian11} asked the fourth author whether the global dichotomy results of \cite{AvilaF11} hold in the $C^{\infty}$ topology.

In the present paper, we address  R. Krikorian's question for  a large class of $C^\infty$ cocycles.   We first introduce the definition of $M$-ultra-differentiable functions. It is known that the derivatives of a $C^\infty$ function $f$ may grow as fast as you like and its regularity  is characterized by the growth of $D^{s}f$.  For a given sequence of positive real numbers $M=(M_{s})_{s\in\mathbb{N}},$  we say $f\in C^{\infty}(\T,\R)$ is
 $M$-ultra-differentiable if  there exists $r>0$ such that
   $$\|D^{s}f\|_{C^0} \leq  r^{-s}M_{s},$$  here $r$ is also called the ``width".
The real-analytic and $\nu$-Gevrey functions are two special cases corresponding to $M_{s}=s!$
and $M_{s}=(s!)^{\nu^{-1}}, 0<\nu<1$ respectively. Then our main result is the following:

\begin{theorem}\label{mainresult}
Let $\alpha\in\R\setminus\mathbb{Q}$, $ V:\mathbb{T} \rightarrow  \R$ be an M-ultra-differentiable function with  $M=(M_{s})_{s\in\mathbb{N}}$ satisfying\\

$ \mathbf{(H1)}$:  \text{Log-convex:}   \quad \quad \quad  $  M_{\ell}^{s-k}< M_{k}^{s-\ell}M_{s}^{\ell-k}, \quad  s>\ell>k,$ \\

$\mathbf{(H2)}$: \text{Sub-exponential growth: } \quad    $\lim_{s\rightarrow\infty}s^{-1}\ln(M_{s+1}/M_s)=0.$\\

 Then for Lebesgue almost every $E \in \mathbb{R},$ either the Schr\"odinger cocycle
$(\alpha, S_{E}^{V})$ is  $C^{\infty}$ rotations reducible or it has positive Lyapunov exponent.
\end{theorem}

\begin{remark}
Theorem \ref{mainresult} proved the  $C^{\infty}$ rigidity  for a large class of $C^{\infty}$ quasiperiodic cocycles.  The {\it almost rigidity}  in $C^{\infty}$ topology was already proved by Fayad-Krikorian \cite{Krikorian09}. More precisely, they proved the cocycle  either  has positive Lyapunov exponent, or the cocycle is  $C^{\infty}$ almost rotations reducible,  which means the cocycle  can be approximated by rotations reducible  cocycles in the $C^{\infty}$ topology.
\end{remark}

We remark that  the
assumptions $\mathbf{(H1)}$ and $\mathbf{(H2)}$ are not  restrictive.
It is obvious that both analytic and Gevrey class functions satisfy  $\mathbf{(H1)}$ and $\mathbf{(H2)}$. Indeed,
the log-convexity condition $\mathbf{(H1)}$  is a very classical assumption in the literature, which guarantees  the space of M-ultra-differentiable functions form a Banach algebra.
The sub-exponential condition $\mathbf{(H2)}$ was first introduced by Bounemoura-Fejoz \cite{Bounemoura20} to guarantee the ultra-differentiable functions have an analogue of the Cauchy estimates for analytic functions, which is one of the main ingredients in KAM theory. We remark that the commonly used condition in the   literature is called  \textit{moderate growth condition:}
$$ \sup_{s,\ell\in\N} \left( \frac{M_{s+\ell}}{M_s M_{\ell}}\right)^{\frac{1}{s+\ell}}<\infty,$$
which is  stronger than $\mathbf{(H2)},$ see \cite{Bounemoura20} for details.

Attached to the sequence $(M_s)_{s\in\N}$, one can define $\Lambda: [0,\ \infty)\rightarrow[0,\ \infty)$ by
\begin{equation*}
\Lambda(y):=\ln \big(\sup_{s\in \mathbb{N}}y^{s}M_{s}^{-1}\big)=\sup_{s\in \mathbb{N}}
(s \ln y-\ln M_{s}),
\end{equation*}
which in fact describes the decay rate of the Fourier coefficients for periodic functions. For  $C^\infty$ smooth periodic functions, the growth of   $\Lambda(y)$ is faster than $\ln (y^{s})$  for any $s\in\N$ as $y$ goes to infinity.  Consequently,  $C^\infty$ means
\begin{equation*}
\lim_{y\rightarrow\infty}\Lambda(y)/\ln (y)=\infty.
\end{equation*}
On the other hand, one can easily check that $M_{s}=\exp\{s^{\delta(\delta-1)^{-1}}\}$ satisfies  $\mathbf{(H1)}$ and $\mathbf{(H2)}$  if and only if $\delta>2.$
 Attached to this $M_s$, $\Lambda(y)=(\ln y)^{\delta}. $ Notice that $\Lambda(y)=y^\nu, 0<\nu<1$ for Gevrey functions. Thus the space of M-ultra-differentiable functions with $\mathbf{(H1)}$ and $\mathbf{(H2)}$ is much bigger than the space of Gevrey functions, and  quite close to the whole space of  $C^\infty$ functions.
However,  those $C^\infty$ functions with $\Lambda(y)\le (\ln y)^2$ are not included. We don't know it is essential or  due to the  shortage of our method.

The proof of Theorem~\ref{mainresult} is based on  renormalization technique and  local KAM result.
For  M-ultra-differentiable functions, to describe the smallness of perturbation,  we define the $\|\cdot\|_{M,r}$-norm by
\begin{equation*}
\begin{aligned}
\|f\|_{M,r}=c\sup_{s\in\mathbb{N}}\big((1+s)^{2}r^{s}\|D_{\theta}^{s}f(\theta)\|_{C^{0}}M_{s}^{-1}\big)<\infty,\
c=4\pi^{2}/3,
\end{aligned}
\end{equation*}
and denote by $U^{M}_{r}(\mathbb{T}, *)$  the
set of all these $*$-valued functions ($*$ will usually denote $\R$, $sl(2,\R)$
$SL(2,\R)$).

Then our precise KAM-type result is the following:

\begin{theorem}\label{mainresults}
Let $r>0$, $\alpha\in \R \setminus\mathbb{Q}$ and $A\in U^{M}_{r}(\mathbb{T}, SL(2,\mathbb{R}))$
with $M$ satisfying $\mathbf{(H1)}$ and $\mathbf{(H2)}.$ Then for every $\tau>1$ and $\gamma>0,$
there exists $\varepsilon_{*}=\varepsilon_{*}(\gamma,\tau,r,M)>0,$ such that if
$\|A-R\|_{M,r}\leq \varepsilon_{*}$ for some $R\in SO(2,\R),$
and  $\rho(\alpha,A)=:\rho_{f}\in DC_{\alpha}(\gamma,\tau),\ i.e.,$
 \begin{equation*}
\|k\alpha\pm2\rho_{f}\|_{\mathbb{Z}}\geq\gamma\langle k\rangle^{-\tau},\ \forall k\in\mathbb{Z},
\langle k\rangle=\max\{1, |k|\},
\end{equation*}
then  $(\alpha,A)$ is $C^{\infty}$ rotations reducible.
\end{theorem}
We point out that, Theorem \ref{mainresults} is a
semi-local result in the terminology of \cite{FayadK18}, i.e., the smallness of the perturbation $\varepsilon_{*}$  does not depend on the frequency $\alpha$. One should not expect that $\varepsilon_{*}$ is independent of $\rho_{f}$ (in terms of $\gamma,\tau$) as this is not true in the $C^{\infty}$ topology (or even Gevrey class) \cite{AvilaK16}. To this end,  we mention another open problem of Fayad-Krikorian \cite{FayadK18}: Is the semi-local version of the almost reducibility conjecture true for cocycles in quasi-analytic classes? In the analytic topology, it has been established in \cite{Hou12,YouZ13}.

The technical reason why we introduce  $\mathbf{(H1)}$ and $\mathbf{(H2)}$ is the following: The proof of  Theorem \ref{mainresults} is based on a non-standard KAM scheme developed in \cite{Hou12,KWYZ18}. The key idea is to prove  that
the homological equations
\begin{equation}\label{h}
\me^{2\mi(2\pi \rho_f+\widetilde{g}(\cdot))}f(\cdot+\alpha)-f+h=0,
\end{equation}
has a smooth approximating solution,  consult section~\ref{Sketchoftheproof} for more discussions. Here $\widetilde{g}(\cdot)$ comes from the perturbation, in order to ensure that \eqref{h} has a smooth approximating solution, we do need some kind of control for all derivatives $\|D^{s} \widetilde{g}(\cdot)\|_{C^{0}}, s\in\N$ which is guaranteed by $\mathbf{(H2)}$.

Next we give a short review of local reducibility results. The pioneering result of local reducibility was due to Dinaburg-Sinai
\cite{Dinaburg75}, who proved that if $\alpha\in DC$, and $V$ is  analytically small, then  $(\alpha, S_{E}^{V})$ is reducible for majority of $E$. Eliasson \cite{Eliasson92} further proved  for Lebesgue almost surely $E$, $(\alpha, S_{E}^{V})$ is reducible. Note these two results are perturbative, i.e. the smallness of $V$ depends on Diophantine constants of $\alpha$. For reducibility results in other topology, one can consult \cite{caiyouzhou19,Claire13,Bounemoura21} and the references therein.

If $\alpha$ is Liouvillean, based on ``algebraic conjugacy trick"  developed in \cite{Krikorian09}, Avila-Fayad-Krikorian \cite{AvilaF11} proved that in the local regime,   $(\alpha, S_{E}^{V})$ is reducible for majority of $E$,  thus gives a generalization of Dinaburg-Sinai's Theorem \cite{Dinaburg75}  to arbitrary one-frequency.  The result was also proved for analytic quasiperiodic linear systems by Hou-You in \cite{Hou12}.  Later, Zhou-Wang  \cite{ZhouW12} generalized $SL(2, \R)$ cocycles result \cite{AvilaF11}
to $GL(d, \R)$ cocycles  by different method. Theorem~\ref{mainresult} and Theorem~\ref{mainresults} can be seen as a  generalization of \cite{AvilaF11} from analytic functions to  ultra-differentiable functions.

\subsection{The spectral applications}

We point out that  global rigidity results in the analytic topology \cite{Krikorian06,AvilaF11} have many important applications in the spectral theory of quasiperiodic Schr\"odigner operators. To name a few, it was used to verify
the Schr\"odinger Conjecture \cite{MMG13} in the Liouvillean context \cite{AvilaF11}, it also  plays an essential role  in solving
``Last's intersection spectrum conjecture" \cite{JitomirskayaM12},
 Aubry-Andre-Jitomirskaya's conjecture \cite{AvilaYZ17}. With Theorem \ref{mainresults}, one can prove the first two conjectures also hold for quasiperiodic operators with M-ultra-differentiable potentials satisfying $\mathbf{(H1)}$ and $\mathbf{(H2)}$.

\subsubsection{Schr\"odinger conjecture} The  Schr\"odinger conjecture \cite{MMG13} says, for general discrete Schr\"odinger operators over uniquely ergodic base dynamics,  all eigenfunctions are bounded for almost every energy in the support of  the absolutely continuous  part of the spectral measure. This conjecture has recently been disproved by Avila \cite{AvilaA15}. However,  it is still interesting to know, to what extend the conjecture is true. For example,
the KAM scheme of \cite{AvilaF11} implies that the Schr\"odinger conjecture is true in the quasiperiodic case with analytic potentials, and this was the first time it was verified in a Liouvillean context. Indeed, as pointed by Jitomirskaya and Marx in \cite{Jitomirskayam17} (page 2363 of \cite{Jitomirskayam17}): addressing the Schr\"odinger conjecture for quasiperiodic operators with lower regularities of the potentials still remains an open problem.

With Theorem \ref{mainresult}, we can prove the Schr\"odinger conjecture  with M-ultra-differentiable quasiperiodic potentials.
\begin{corollary}
Let $\alpha\in\R\setminus\mathbb{Q}$, $ V:\mathbb{T} \rightarrow  \R$ be a M-ultra-differentiable function satisfying $\mathbf{(H1)}$ and $\mathbf{(H2)}$.   Then the Schr\"odinger conjecture is true.
\end{corollary}

\subsubsection{Last's intersection spectrum conjecture}
Denote
\begin{equation*}
\begin{aligned}
S_{-}(\beta)=\cap_{\theta\in\T}\Sigma_{ac}(\beta,\theta),
\end{aligned}
\end{equation*}
where $\Sigma_{ac}(\beta,\theta)$ is the absolutely continuous spectrum of (quasi)periodic Schr\"odinger operator $H_{V,\beta,\theta}$ defined by \eqref{20200824operator}.
For any $\alpha\in \R\setminus\Q,$ it can be approximated by a sequence of rational numbers $(p_{n}/q_{n}).$
It is well known that the rational frequency approximation is indispensable for numeric analysis, thus the existence of the
limits $S_{-}(p_{n}/q_{n})$ as $p_{n}/q_{n}\rightarrow \alpha$ is crucial.
A conjecture of Y. Last says,
up to a set of zero Lebesgue measure, the absolutely continuous spectrum can be
obtained asymptotically from $S_{-}(p_{n}/q_{n})$, the spectrum of  periodic operators associated
with the continued fraction expansion of $\alpha.$

 Jitomirskaya-Marx \cite{JitomirskayaM12} settled the ``Last's intersection spectrum conjecture" for analytic quasiperiodic Schr\"odinger operators. They also pointed out, in \cite{JitomirskayaM12}, that
the analyticity of the potential $V$ is essential for the proof of their result,
and, whether or not the
analyticity can be relaxed without reducing the range of frequencies
for which the statement holds is an interesting open problem (page 5 of \cite{JitomirskayaM12}). In this work,
we will give a positive answer to this problem
for $\nu$-Gevrey potentials with $1/2<\nu\leq 1$. In the following, we say two sets $A\doteq B$ if $ \chi_{A}=\chi_{B}\
$ Lebesgue almost everywhere. Moreover, we say $\lim_{n\rightarrow\infty}B_{n}\doteq B$
 if
$
 \lim_{n\rightarrow\infty} \chi_{B_{n}}=\chi_{B}\
$ Lebesgue almost everywhere.

\begin{theorem}\label{lastintersection}
Let $\alpha\in\R\setminus\mathbb{Q}$ and $V:\mathbb{T} \rightarrow  \R$ be a $\nu$-Gevrey function with $1/2< \nu\leq 1$, there is a sequence $p_{n}/q_{n}\rightarrow \alpha$ such that
\begin{equation*}
\begin{aligned}
\lim_{n\rightarrow\infty}S_{-}(p_{n}/q_{n})\doteq
S_{-}(\alpha)=\Sigma_{ac}(\alpha).
\end{aligned}
\end{equation*}
\end{theorem}
%
\begin{remark}
As we will see, in fact we will prove
\begin{equation}\label{l1es}
\Sigma_{ac}(\alpha)\subset\liminf_{n\rightarrow\infty}S_{-}(p_{n}/q_{n})
\end{equation}
for all  M-ultra-differentiable potential satisfying  $\mathbf{(H1)}$ and $\mathbf{(H2)}$ (Theorem \ref{addproposition}). Gevrey property only plays a role in proving
\begin{equation}\label{l2es}\limsup_{n\rightarrow\infty}S_{-}(p_{n}/q_{n})\subset\Sigma_{ac}(\alpha)\end{equation}
for Diophantine frequency  (Theorem \ref{addpropositions}).
\end{remark}

We briefly explain why  analyticity is crucial for the proof of  \cite{JitomirskayaM12}.
On the one hand, the key  of \eqref{l1es}
is  to prove  that $E\in\Sigma_{ac}(\alpha)$ implies exponentially
small variation (in $q_{n}$) of the approximating discriminants (``generalized Chambers' formula"). For the analytic potential, Jitomirskaya-Marx \cite{JitomirskayaM12} got this estimate as a corollary of  Avila's quantization of acceleration \cite{Avila15}, which can be defined only for analytic cocycles.
On the other hand, the proof of \eqref{l2es}
 was first obtained by Shamis \cite{Shamis11} as a corollary of the  continuity of Lyapunov exponent: i.e. the Lyapunov exponent  $L(\beta+\cdot,\cdot)$: $\T\times C^{\omega}(\T, SL(2,\mathbb{C}))$ is jointly continuous for any irrational $\beta$ \cite{Bourgainj02,JitomirskayaKS09,JitomirskayaCAMAR12}.  However,  the Lyapunov  exponent  $L(\beta+\cdot,\cdot)$: $\T\times C^{\infty}(\T, SL(2,\mathbb{C}))$ is not continuous \cite{Wangy13}.

 In fact, it was also pointed out by Jitomirskaya and Marx in \cite{JitomirskayaM12} that analyticity should not be essential for their results, while one needs new methods in the non-analytic case. To generalize the result in \cite{JitomirskayaM12} to ultra-differentiable potential, we have to
overcome the difficulty caused by the non-analyticity of potential. One key issue is to prove the ``generalized Chambers' formula" in
ultra-differentiable case. Instead of using Avila's quantization of acceleration \cite{Avila15},  we will use perturbative argument
 which avoids the analyticity, showing that if the cocycle is smoothly rotations reducible, then $q$-step transfer matrices grows sub-exponentially in $q$ \footnote{As pointed out in footnote 5 of   \cite{JitomirskayaM12}, this ideas was first pointed out by the fourth author after first preprint of  \cite{JitomirskayaM12}.}.
 To do this, we will use inverse renormalization and quantitative KAM result, to show if the cocycle is almost  reducible in ultra-differentiable topology, then we have a good control of the  growth of $q$-step transfer matrices. As for the proof of second inclusion, the key is to prove that Lyapunov exponent can be still continuous with respect to the rational approximation of the frequency for $\nu$-Gevrey potential $V$ with $1/2< \nu<1,$ if $\alpha$ is Diophantine, which is a generalization of the results in \cite{Bourgainj02}. See Theorem~\ref{continuity}  for details. Recently, Ge-Wang-You-Zhao \cite{GWYZ} further constructed counter-examples for $\nu$-Gevrey potential $V$ with $0< \nu<1/2$, which shows  Theorem~\ref{continuity} is optimal.

Finally, we  review some related results. For general ergodic discrete Schr\"odinger operators, the relation between the absolutely continuous spectrum and the spectrum of certain periodic approximates has been studied by Last in \cite{Last92,LastY93}, more precisely, \cite{LastY93} essentially proved that for $V\in C^1(\T)$ and a.e. $\alpha$, $\limsup_{n\rightarrow\infty}S_-(p_n/q_n)\subset \Sigma_{ac}(\alpha)$ up to sets of zero Lebesgue measure. The conjecture is known for the almost Mathieu operator where $V(\theta)=2\lambda\cos\theta$ (\cite{AvronS90,Last94} for a.e. $\alpha$, $\lambda$ and \cite{Krikorian06,Jitomirskayak02,LastY93,JitomirskayaM12} extending to all $\alpha$). More recently, the conjecture was settled for a.e. $\alpha$ and sufficiently smooth potential by Zhao \cite{Zhaoxin19}.

\subsection{The structure of this paper}
The paper is arranged as follows. In Section~\ref{Definitionsandpreliminaries} we give some definitions and preliminaries.
Before  giving the proof of Theorem \ref{mainresults}, we first derive condition $\mathbf{(A)}$ on Fourier coefficients from assumptions
$\mathbf{(H1)}$ and $\mathbf{(H2)}$ on Taylor coefficients (Lemma \ref{finallemma}) in Section~\ref{Ultradifferentiable}. Then we prove Theorem  \ref{mainresults}
in Section~\ref{Inthissection} and prove Theorem \ref{mainresult} in Section~\ref{Theoremmainresults}. The proof of  Theorem \ref{lastintersection}
 is given in Section~\ref{Lastsintersection}, which was based on Theorem \ref{addproposition} and Theorem \ref{addpropositions}.
 In Section~\ref{chambers} we give the proof of Generalized Chambers' formula (Proposition \ref{differenceestimate}),  and in  Section~\ref{gerc} we give the
 proof  of the joint continuity of Lyapunov exponent (Theorem \ref{continuity}), these two results are bases of the proof of Theorem \ref{addproposition} and Theorem \ref{addpropositions} respectively.

\section{Definitions and preliminaries}\label{Definitionsandpreliminaries}

\subsection{Quasiperiodic cocycles}\label{qpcocycle}

Given $A \in C^0(\T,SL(2,\R))$ and $\alpha\in \R\setminus\Q$, the iterates of $(\alpha,A)$ are of the form $(\alpha,A)^n=(n\alpha,  A_n)$, where
$$
A_n(\cdot):=
\left\{\begin{array}{l l}
A(\cdot+(n-1)\alpha) \cdots A(\cdot+\alpha) A(\cdot),  & n\geq 0\\[1mm]
A^{-1}(\cdot+n\alpha) A^{-1}(\cdot+(n+1)\alpha) \cdots A^{-1}(\cdot-\alpha), & n <0
\end{array}\right.    .
$$
Define the finite Lyapunov exponent as
$$ L_n(\alpha,A)= \frac{1}{n}\int_{\mathbb{T}}\ln\|A_{n}(\theta)\|d\theta,$$
then by Kingman's subadditive ergodic theorem,
the {\it Lyapunov exponent}  of $(\alpha, A)$ is defined as
\begin{equation*}
\begin{split}
L(\alpha,A)=\lim_{n\rightarrow\infty}L_n(\alpha,A)= \inf_{n>0} L_n(\alpha,A)\geq0.
\end{split}
\end{equation*}

The cocycle $(\alpha, A)$ is called uniformly hyperbolic if there exists a continuous splitting
$E_{s}(\theta)\oplus E_{u}(\theta)=\mathbb{R}^{2},$ and $C>0, 0<\lambda<1,$ such that for every $n\geq1$ we have
\begin{equation*}
\begin{split}
\|A_{n}(\theta)w\|\leq C\lambda^{n}\|w\|,\ \forall w\in E_{s}(\theta),\\
\|A_{-n}(\theta)w\|\leq C\lambda^{n}\|w\|,\ \forall w\in E_{u}(\theta).
\end{split}
\end{equation*}

Assume now  $A\in C^0(\T, SL(2,\R))$ is homotopic to the
identity, then there exist $\psi:\T \times \T \to \R$ and $u:\T
\times\T \to \R^+$ such that
\begin{equation*}
\begin{split}
A(x) \cdot \left (\begin{matrix} \cos 2 \pi y \\
\sin 2 \pi y \end{matrix} \right )=u(x,y) \left (\begin{matrix} \cos 2 \pi (y+\psi(x,y))
\\ \sin 2 \pi (y+\psi(x,y)) \end{matrix} \right ).
\end{split}
\end{equation*}
The function $\psi$ is
called a {\it lift} of $A$.  Let $\mu$ be any probability measure on
$\T \times \T$ which is invariant by the continuous map $T:(x,y)
\mapsto (x+\alpha,y+\psi(x,y))$, projecting over Lebesgue measure on
the first coordinate (for instance, take $\mu$ as any accumulation
point of $\frac {1} {n} \sum_{k=0}^{n-1} T_*^k \nu$ where $\nu$ is
Lebesgue measure on $\T \times \T$). Then the number
\begin{equation*}
\begin{split}
\rho(\alpha,A)=\int \psi d\mu \mod \Z
\end{split}
\end{equation*}
does not depend on the choices of $\psi$ and $\mu$ and is called the {\it fibered rotation number} of $(\alpha,A)$, see \cite {Johnsonm82} and \cite {Herman83}.
It is immediate from the definition that
\begin{equation}\label{rotationnumberper}
|\rho(\alpha, A)-\rho|\leq\|A-R_{\rho}\|_{C^{0}}.
\end{equation}

%
%
%

\subsection{Continued fraction expansion}

Let $\alpha \in (0,1)$ be irrational. Define $ a_0=0,
\alpha_{0}=\alpha,$ and inductively for $k\geq 1$,
$$a_k=[\alpha_{k-1}^{-1}],\qquad \alpha_k=\alpha_{k-1}^{-1}-a_k=G(\alpha_{k-1})=\{\alpha_{k-1}^{-1}\},$$
where $G(\cdot)$ is the Gauss map.
Let $p_0=0,  p_1=1,  q_0=1,  q_1=a_1,$ then we define inductively
$p_k=a_kp_{k-1}+p_{k-2}$, $q_k=a_kq_{k-1}+q_{k-2}.$
The sequence $(q_n)$  is the  denominators of best rational
approximations of $\alpha$ since we have
\begin{equation*}
\begin{split}
\|k\alpha\|_{\mathbb{Z}}\geq\|q_{n-1}\alpha\|_{\mathbb{Z}},\quad \forall\,\, 1\leq k<q_{n},
\end{split}
\end{equation*}
and
\begin{equation*}
\begin{split}
(q_{n}+q_{n+1})^{-1}<\|q_{n}\alpha\|_{\mathbb{Z}}\leq q_{n+1}^{-1}.
\end{split}
\end{equation*}

For sequence $(q_{n})$, we will fix a particular subsequence $(q_{n_{k}})$ of the denominators of the best rational
approximations for $\alpha,$ which for simplicity will be denoted by $(Q_{k})$. Denote the sequences
 $(q_{n_{k}+1})$ and $(p_{n_{k}})$ by $(\overline{Q}_{k})$ and $(P_{k}),$ respectively. Next, we introduce the
 concept of CD bridge which was introduced in \cite{AvilaF11}.
\begin{definition}[CD bridge,\cite{AvilaF11}]\label{CDbridge}
Let $0<\mathbb{A}\leq\mathbb{B}\leq\mathbb{C}$. We say that the pair of denominators
$(q_{m},q_{n})$ forms a $CD(\mathbb{A},\mathbb{B},\mathbb{C})$ bridge if
\begin{equation*}
\begin{split}
&\bullet\,\,  q_{i+1}\leq q_{i}^{\mathbb{A}},\quad i=m,\cdots,n-1,\\
&\bullet \,\, q_{m}^{\mathbb{C}}\geq q_{n}\geq q_{m}^{\mathbb{B}}.
\end{split}
\end{equation*}
\end{definition}
\begin{lemma}\label{bridgeestimate}\emph{\cite{AvilaF11}}
For any $\mathbb{A}\geq1$, there exists a subsequence $(Q_{k})$ of $(q_{n})$ such that $Q_{0}=1$
and for each $k\geq0,$ $Q_{k+1}\leq\overline{Q}_{k}^{\mathbb{A}^{4}}$, either
$\overline{Q}_{k}\geq Q_{k}^{\mathbb{A}}$, or the pairs $(\overline{Q}_{k-1},Q_{k})$ and
$(Q_{k},Q_{k+1})$ are both $CD(\mathbb{A},\mathbb{A},\mathbb{A}^{3})$ bridges.
\end{lemma}
Set $\tau>1$ and $\mathbb{A}>\tau+23>24$, then for $\{\overline{Q}_{n}\}_{n\geq0},$ the selected subsequence in
 Lemma~\ref{bridgeestimate}, we have the following lemma.
\begin{lemma}\label{knandknrelation}
For  $\{\overline{Q}_{n}\}_{n\geq0},$ we have
\begin{equation*}
\begin{split}
\overline{Q}_{n+1}\geq\overline{Q}_{n}^{\mathbb{A}},\ \forall n\geq0.
\end{split}
\end{equation*}
\end{lemma}

\begin{proof}
$\mathbf{Case\ one:}$ $\overline{Q}_{n+1}\geq Q_{n+1}^{\mathbb{A}}.$ Obviously
$\overline{Q}_{n+1}\geq Q_{n+1}^{\mathbb{A}}\geq\overline{Q}_{n}^{\mathbb{A}}.$

   $\mathbf{Case \ two:}$ $\overline{Q}_{n+1}< Q_{n+1}^{\mathbb{A}}.$ In this case
we know that $(\overline{Q}_{n},Q_{n+1})$ forms a  CD $(\mathbb{A},\mathbb{A},\mathbb{A}^{3})$
bridge. Thus $Q_{n+1}\geq\overline{Q}_{n}^{\mathbb{A}},$ which implies
$\overline{Q}_{n+1}\geq\overline{Q}_{n}^{\mathbb{A}}.$

\end{proof}

\subsection{Renormalization}\label{addrenormalization}
In this subsection we give the notations and definitions about the renormalization
which are given in \cite{Krikorian06,Krikorian09,Krikorian15}.

\subsubsection{$\Z^{2}-$actions}
Consider the cocycle $(\alpha, A)\in (0,1) \setminus\mathbb{Q}\times U_r^{M}(\mathbb{T}, SL(2,\mathbb{R}))$ and
set $\beta_{n}=\Pi_{l=0}^{n}\alpha_{l}=(-1)^{n}(q_{n}\alpha-p_{n})
=(q_{n+1}+\alpha_{n+1}q_{n})^{-1},$ where $\alpha_n=G^n(\alpha)$.
 Let $\Omega^{r}=\R\times U_{r}^{M}(\R,SL(2,\R))$ be the subgroup of
Diff$(\R\times U_{r}^{M}(\R,SL(2,\R)))$ made of skew-product diffeomorphisms
$(\alpha,A)\in\R\times U_{r}^{M}(\R,SL(2,\R)).$

A $U_{r}^{M}$ fibered $\Z^{2}-$action is a homomorphism $\Phi:\Z^{2}\rightarrow\Omega^{r}.$
We denote by $\Lambda^{r}$ the
space of such actions, and denote $\Phi=(\Phi(1,0),\Phi(0,1))$ for short.
Let $\Pi_{1}:\R\times U_{r}^{M}(\R,SL(2,\R))\rightarrow \R,$
$\Pi_{2}:\R\times U_{r}^{M}(\R,SL(2,\R))\rightarrow
U_{r}^{M}(\R,SL(2,\R))$
be the coordinate projections.
Let also $\gamma_{n,m}^{\Phi}=\Pi_{1}\circ\Phi(n,m)$
 and $A_{n,m}^{\Phi}=\Pi_{2}\circ\Phi(n,m).$

Two fibered $\Z^2$ actions $\Phi$, $\Phi'$ are said to be conjugate if there exists a smooth map
 $B: \R\rightarrow SL(2,\mathbb{R})$ such that
 $$\Phi'(n,m)=(0,B) \circ \Phi(n,m) \circ (0,B)^{-1}, \qquad \forall (n,m)\in\Z^2 .$$
 That is
 $$  A_{n,m}^{\Phi'}(\cdot) =B(\cdot+\gamma_{n,m}^{\Phi})A_{n,m}^{\Phi}(\cdot) B(\cdot)^{-1} , \qquad  \gamma_{n,m}^{\Phi'}= \gamma_{n,m}^{\Phi}. $$
 We denote $\Phi'= \mathrm{Conj_{B}}(\Phi)$ for short.     We  say that an action is normalized if $\Phi(1,0)=(1,Id),$ and in that case, if
 $\Phi(0,1)=(\alpha,A),$ the map $A\in U_{r}^{M}(\R,SL(2,\R))$
is clearly $\Z-$periodic.

For  any  $M$-ultra-differentiable function $f: \R \rightarrow \R$ (not necessary periodic), one can also define
\begin{equation*}
\begin{aligned}
\|f\|_{r,T}=c\sup_{s\in\mathbb{N}}\big((1+s)^{2}r^{s}\|D_{\theta}^{s}f(\theta)\|_{C^{0}([0,T])}M_{s}^{-1}\big)\
, c=4\pi^{2}/3.
\end{aligned}
\end{equation*}
If $f: \T \rightarrow \R$ is periodic, we also denote $\|f\|_{r,1}=\|f\|_{M,r}$.

\begin{lemma}\emph{(Lemma~2 of \cite{Krikorian09})}\label{normalizinglemma}
 If $\Phi\in\Lambda^{r}$ with $\gamma_{1,0}^{\Phi}=1,$ then there exists  $B\in U_{r}^{M}(\mathbb{R},SL(2,\mathbb{R}))$ and a normalized action $\widetilde{\Phi}$
such that $\widetilde{\Phi}=\mathrm{Conj_{B}}(\Phi)$. Moreover,  for any $T\in\R^{+}$, we have  estimate
\begin{eqnarray*}
\|B-Id\|_{rK_{*}^{-1},1} &\leq& \|\Phi(1,0)-Id\|_{r,1} ,\\
\|B\|_{r(K_{*}T)^{-1},T} &\leq& \|\Phi(1,0)\|_{r,T}^{T+1},  \quad  \forall T\in\R^{+},
\end{eqnarray*}
where $K_{*}$ is an absolute constant.
 \end{lemma}

\subsubsection{Renormalization of actions}
Following \cite{Krikorian06,Krikorian09,Krikorian15}, we introduce  the scheme of renormalization of $\Z^{2}$ actions.

Fixing  $\lambda\neq 0.$ Define $M_{\lambda}: \Lambda^{r}\rightarrow \Lambda^{r}$ by
\begin{equation*}
M_{\lambda}(\Phi)(n,m):=(\lambda^{-1}\gamma_{n,m}^{\Phi}, A_{n,m}^{\Phi}(\lambda\cdot)).
\end{equation*}
Let $\theta_{*}\in\R.$ Define $T_{\theta_{*}}: \Lambda^{r}\rightarrow \Lambda^{r}$ by
\begin{equation*}
T_{\theta_{*}}(\Phi)(n,m):=(\gamma_{n,m}^{\Phi}, A_{n,m}^{\Phi}(\cdot+\theta_{*})).
\end{equation*}
Let $U\in GL(2,\R).$ Define $N_{U}: \Lambda^{r}\rightarrow \Lambda^{r}$ by
\begin{equation*}
N_{U}(\Phi)(n,m):=\Phi(n',m'),\ \text{where} \ \Big(\begin{matrix}
n'\\
m'
\end{matrix}\Big)=U^{-1}
\Big(\begin{matrix}
n\\
m
\end{matrix}\Big).
\end{equation*}

Let $\widetilde{Q}_{n}=\Big(\begin{matrix}
q_{n},  & p_{n}\\
q_{n-1}, & p_{n-1}
\end{matrix}\Big),$
and define for $n\in\N$ and $\theta_{*}\in\R$ the renormalized actions
\begin{equation*}
\mathcal{R}^{n}(\Phi):=M_{\beta_{n-1}}\circ N_{\widetilde{Q}_{n}}(\Phi),\
\mathcal{R}_{\theta_{*}}^{n}(\Phi):=
T_{\theta_{*}}^{-1}\big[\mathcal{R}^{n}(T_{\theta_{*}}(\Phi))\big].
\end{equation*}

For any given cocycle $(\alpha, A)$ with $\alpha\in\mathbb{R}\setminus\mathbb{Q}$, we set $\Phi=((1,Id),(\alpha, A)).$  Then by the definitions of the operators above, we get
\begin{equation*}
\begin{split}
\mathcal{R}_{\theta_{*}}^{n}(\Phi)=((1,A^{(n,0)}),(\alpha_{n},A^{(n,1)})),
\end{split}
\end{equation*}
where
\begin{equation*}
\begin{split}
A^{(n,0)}(\theta)&=A_{(-1)^{n-1}q_{n-1}}(\theta_{*}+\beta_{n-1}(\theta-\theta_{*})),\\
A^{(n,1)}(\theta)&=A_{(-1)^{n}q_{n}}(\theta_{*}+\beta_{n-1}(\theta-\theta_{*})).
\end{split}
\end{equation*}
Thus $A^{(n,0)}$ and $A^{(n,1)}$ are
$\beta_{n-1}^{-1}-$periodic and can be regarded as cocycles over the dynamics on
$\mathbb{R}$ given by $\theta\mapsto \theta+1$ and $\theta\mapsto \theta+\alpha_{n}$. It is easy to see that
$A^{(n,1)}(\theta+1)A^{(n,0)}(\theta)=A^{(n,0)}(\theta+\alpha_{n})A^{(n,1)}(\theta),$ which expresses the commutation of
the cocycles. Based on this fact,  there exists $D_n$  which is a normalizing map such that
\begin{eqnarray*}
D_n(\theta+1)A^{(n,0)}(\theta)D_n(\theta)^{-1} & =& Id, \\
D_n(\theta+\alpha_{n})A^{(n,1)}(\theta)D_n(\theta)^{-1} &=& A^{(n)}(\theta),
\end{eqnarray*}
which  satisfies $A^{(n)}(\theta+1)=A^{(n)}(\theta).$ Thus $A^{(n)}$ can be seen as an element of
$C^{0}(\mathbb{T}, SL(2, \mathbb{R})),$
 and $(\alpha_{n}, A^{(n)})$ is called a representative of the $n$-th renormalization of $(\alpha, A).$

 \subsubsection{Convergence of the renormalized actions}

 The following result on convergence of  renormalized actions was essentially contained in \cite{Krikorian06,Krikorian09,Krikorian15}, which
deal with cocycles in $C^{\ell}$ setting with $\ell\in\mathbb{N}$ and $\ell=\infty, \omega.$ We will sketch the proof in the ultra-differentiable setting, just for completeness.

\begin{proposition} [\cite{Krikorian06,Krikorian09,Krikorian15}]\label{llambdarenormalizations}
Suppose that $(\alpha, A)\in (0,1) \setminus\mathbb{Q}\times U_r(\mathbb{T}, SL(2,\mathbb{R}))$. If $(\alpha, A)$ is $L^{2}$-conjugated to rotations and  homotopic to the identity,
then  for almost every $\theta_{*}\in\mathbb{R},$ there exists  $D_{n}\in U_{r/K_*^{2}}(\mathbb{R}, SL(2,\mathbb{R}))$ with
\begin{equation}\label{modelnormalizings}
\|D_{n}\|_{r/(K_{*}^2 T),T} \leq  C^{q_{n-1}(T+1)},
\end{equation}
such that
\begin{equation*}
\begin{aligned}
\mathrm{Conj_{D_{n}}}(\mathcal{R}_{\theta_{*}}^{n}(\Phi))= ((1,Id),(\alpha_{n},\ R_{\rho_{n}}\me^{F_{n}})),
\end{aligned}
\end{equation*}
with $\|F_{n}\|_{r/K_{*}^{2},1}\rightarrow 0,$  where $K_{*}$ is an absolute constant defined in  Lemma \ref
{normalizinglemma}.
\end{proposition}
\begin{proof}
We first prove $\{A^{(n,i)}(\theta)\}_{n\geq0}, i=1,2,$ are precompact
in $U_{r/K_{*}}^{M}.$ Indeed, Theorem 5.1 of  \cite{Krikorian06} shows that  for any $(\alpha, A)\in (0,1) \setminus\mathbb{Q}\times C^s(\mathbb{T}, SL(2,\mathbb{R}))$, if it is $L^{2}$-conjugated to rotation, then  for almost every $\theta_{*}\in\mathbb{R},$ there exists $K_{*}>0$ such
that for every $d>0$ and for every $n>n_{0}(d),$
\begin{equation*}
\begin{aligned}
\|\partial^{\ell} A^{(n,i)}(\theta)\|\leq K_{*}^{\ell+1}\|A(\theta)\|_{C^{s}},\ i=1,2,\ 0\leq \ell\leq s,\ |\theta-\theta_{*}|<d/n,
\end{aligned}
\end{equation*}
which implies that
\begin{equation*}
\begin{aligned}
\|A^{(n,i)}(\theta)\|_{C^{s}}\leq 2K_{*}^{s+1}\| A(\theta)\|_{C^{s}},\ i=1,2,
\end{aligned}
\end{equation*}
therefore, by the definition of norms of ultra-differentiable functions, we have
\begin{equation*}
\begin{aligned}
\|A^{(n,i)}(\theta)\|_{r/K_{*},1}\leq 2K_*\| A(\theta)\|_{r,1}<\infty,\ i=1,2.
\end{aligned}
\end{equation*}
That is the sequences $\{A^{(n,i)}(\theta)\}_{n\geq0}, i=1,2,$ are  uniformly bounded in $U_{r/K_{*}},$
which implies $\{A^{(n,i)}(\theta)\}_{n\geq0}, i=1,2,$ are precompact
in $U_{r/K_{*}}^{M}.$

Assume $B\in L^{2}(\T, SL(2,\R))$ is the conjugation such that
$$B(\theta+\alpha)A(\theta)B(\theta)^{-1}\in SO(2,\R),$$
consequently by Theorems~4.3, Theorem 4.4 in \cite{Krikorian15} we have
$$\mathcal{R}_{\theta_{*}}^{n}(\mathrm{Conj_{B(\theta_{*})}}(\Phi))
=((1,\widetilde{C}_{n}^{(1)}(\theta)),(\alpha_{n},\widetilde{C}^{(2)}_{n}(\theta))$$
with $\widetilde{C}_{n}^{(i)}=R_{\rho_{n}}\me^{U_{n}^{(i)}(\theta)}, i=1,2,$ and
\begin{equation*}
\begin{aligned}
\|U_{n}^{(i)}(\theta)\|_{r/K_{*},1}\rightarrow0,\ i=1,2,\text{if}\ n\rightarrow\infty, |\theta-\theta_{*}|\leq d/n, n\geq n_{0}(d).
\end{aligned}
\end{equation*}
Using Lemma~\ref{normalizinglemma}, there is a normalizing conjugation $\widetilde{D}_{n},$
which is closed to identity in $\|\cdot\|_{rK_{*}^{-2},1}-$topology such that $$\widetilde{D}_{n}(\theta+1)\widetilde{C}_{n}^{(1)}
(\theta)\widetilde{D}_{n}(\theta)^{-1}=Id.$$
Denote $D_{n}=\widetilde{D}_{n}B,$ the action $\mathrm{Conj_{D_{n}}}(\mathcal{R}_{\theta_{*}}^{n}(\Phi))$
is of form $((1,Id),(\alpha_{n},\ R_{\rho_{n}}\me^{F_{n}}))$ with
$\|F_{n}\|_{rK_{*}^{-2},1}\rightarrow 0.$
Moreover,  for any $T\in \R^+$, by Lemma~\ref{normalizinglemma} we get
\begin{equation*}
\begin{aligned}
\|\widetilde{D}_{n}\|_{r/(K_{*}^2 T),T}&\leq
\|\widetilde{C}_{n}^{(1)}\|_{ r/K_{*} ,T}^{T+1}
\leq  \|B(\theta_*)\|^{2(T+1)}\|A^{(n,0)}\|_{r,T}^{T+1}\\
&\leq \|B(\theta_*)\|^{2(T+1)} \|A\|_{r,1}^{q_{n-1}(T+1)}\leq C^{q_{n-1}(T+1)},
\end{aligned}
\end{equation*}
then \eqref{modelnormalizings} follows directly.
\end{proof}

\section{Ultra-differentiable functions}\label{Ultradifferentiable}

As we introduced, one way to define the modulus of ultra-differentiable functions is by the growth of $D^{s}f$.
For periodic function $f\in C^{\infty}(\T,\R),$ an alternative way is to define the modulus of ultra-differentiability
by the decay rate of its Fourier coefficient.
Attached to the sequence $(M_s)_{s\in\N}$, we can define $\Lambda: [0,\ \infty)\rightarrow[0,\ \infty)$ by
\begin{equation}\label{definelambda}
\Lambda(y):=\ln \big(\sup_{s\in \mathbb{N}}y^{s}M_{s}^{-1}\big)=\sup_{s\in \mathbb{N}}
(s \ln y-\ln M_{s}).
\end{equation}
This defines a function $\Lambda:[0,\infty)\rightarrow[0,\infty),$
which is continuous, constant equal to zero for $y\leq1$ and strictly increasing for $y\geq1$
(see \cite{Bounemoura20,Chaumat94} or \cite{Thilliez03}).

For any $f\in U_{r}^{M}(\mathbb{T}, \mathbb{R}),$ write it as $f(\theta)=\sum_{k\in\mathbb{Z}}\widehat{f}(k)\me^{2\pi\mi k\theta}$,
one easily checks that $\Lambda$ controls the decay of the Fourier coefficients in the sense that
\begin{equation}\label{fouriercoefficient}
|\widehat{f}(k)|\leq \|f\|_{M,r}\exp\{-\Lambda(|2\pi k|r)\},\ \forall k\in\mathbb{Z}.
\end{equation}

For periodic function, using $\Lambda$ is more natural and convenient. Now we derive some properties of $\Lambda$ for $f\in U_{r}^{M}(\mathbb{T}, \mathbb{R})$ from $\mathbf{(H1)}$ and $\mathbf{(H2)},$
which will be the bases of our whole proof of the KAM scheme.

\begin{lemma}[Proposition~10 of \cite{Bounemoura20}]\label{banachalgebra}
Let $f,g\in U_{r}^{M} (\mathbb{T}, \mathbb{R})$ with $M$ satisfying $\mathbf{(H1)}.$ Then $f\cdot g\in U_{r}^{M} (\mathbb{T}, \mathbb{R}),$ and we have
\begin{equation*}
\begin{aligned}
\|f\cdot g\|_{M,r}\leq \|f\|_{M,r}\|g\|_{M,r}.
\end{aligned}
\end{equation*}
\end{lemma}

\begin{remark}
As explained in  \cite{Bounemoura20},  the role of the normalizing constant $c>0$   in the definition of $ \|f\|_{M,r}$ is to ensure that $U_{r}^{M} (\mathbb{T}, \mathbb{R})$ forms a standard Banach algebra with respect to multiplication.
\end{remark}

\begin{lemma}[Proposition~8 of \cite{Bounemoura20}]\label{cauchyestimate}
Let $f\in U_{r}^{M}(\mathbb{T},\mathbb{R})$ with $M$ satisfying $\mathbf{(H2)}.$
Then $\partial f\in U_{r/2}^{M}(\mathbb{T},\mathbb{R})$ with
\begin{equation*}
\begin{aligned}
\|\partial f\|_{r/2}\leq C_{M}r^{-1}\|f\|_{r},
\end{aligned}
\end{equation*}
where
\begin{equation*}
\begin{aligned}
C_{M}:=\sup_{s\in\mathbb{N}}\{2^{-s}M_{s+1}M_{s}^{-1}\}<\infty.
\end{aligned}
\end{equation*}
\end{lemma}

\begin{lemma}\label{finallemma}
 $\mathbf{(H1)}$ and $\mathbf{(H2)}$  imply that there exists $\Gamma: [1,\ \infty) \rightarrow \mathbb{R}^{+}$
such that the following hold:
\begin{equation*}
\begin{aligned}
  \mathbf{(A):}
  \left\{
\begin{array}{l}
\mathrm{(I)}: \lim_{x\rightarrow\infty}\Gamma(x) = \infty,
  \\
\mathrm{(II)}: \Gamma(x)\ln x\ \mathrm{is} \ \mathrm{non-decreasing},
  \\
\mathrm{(III)}: \Lambda(y)-\Lambda(x) \geq (\ln y-\ln x)\Gamma(x)\ln x,\ \forall y>x\geq 1.
 \end{array}
 \right.
\end{aligned}
\end{equation*}
\end{lemma}

\begin{proof}
For any $x\geq 1,$ we select  $s(x)\in\mathbb{N}$ as the one such that
\begin{equation}\label{appedix20}
\begin{split}
\Lambda(x)=\sup_{s\in\mathbb{N}}\ln (x^{s}M_{s}^{-1})=\ln(x^{s(x)}M_{s(x)}^{-1}).
\end{split}
\end{equation}

\begin{claim}
The function $s(x)\in\mathbb{N}$ is well-defined, non-decreasing with
\begin{equation}\label{sx}
\lim_{x\rightarrow \infty} s(x)(\ln x)^{-1} =\infty.
\end{equation}
\end{claim}

\begin{proof}
It is quite standard $s(x)\in\mathbb{N}$ is well-defined, we first prove that $s(x)$ is non-decreasing.
By the definition of $s(x)$, we have
\begin{equation*}
\begin{split}
x^{s(x)}M_{s(x)}^{-1}\geq x^{s(x)+1}M_{s(x)+1}^{-1},\ x^{s(x)}M_{s(x)}^{-1}\geq x^{s(x)-1}M_{s(x)-1}^{-1},
\end{split}
\end{equation*}
which implies
\begin{equation}\label{appedix21}
\begin{split}
M_{s(x)}/M_{s(x)-1}\leq x\leq M_{s(x)+1}/M_{s(x)}.
\end{split}
\end{equation}

Assume that there exist  $y>x\geq 1$ such that $s(y)<s(x).$
The fact that $s(\cdot) \in \N$ implies  $s(y)+1\leq s(x).$ First, by \eqref{appedix21} we get
\begin{equation*}
\begin{split}
y\leq M_{s(y)+1}/M_{s(y)}, \quad x\geq M_{s(x)}/M_{s(x)-1}.
\end{split}
\end{equation*}
However, by $\mathbf{(H1)}$ we know that $\{M_{\ell+1}/M_{\ell}\}_{\ell\in\mathbb{N}}$
is increasing, which together with $s(y)+1\leq s(x),$ implies that
\begin{equation*}
\begin{split}
y\leq M_{s(y)+1}/M_{s(y)}\leq M_{s(x)}/M_{s(x)-1}\leq x,
\end{split}
\end{equation*}
this contradicts with the assumption $y>x.$ Thus $s(x)$ is non-decreasing.

By \eqref{appedix21}, we have
\begin{equation*}
\begin{split}
s^{-1}(x)\ln (M_{s(x)}/M_{s(x)-1})\leq s^{-1}(x)\ln x\leq s^{-1}(x)\ln (M_{s(x)+1}/M_{s(x)}),
\end{split}
\end{equation*}
then \eqref{sx} follows from the assumption $\mathbf{(H2)}$.
\end{proof}

Let $\Gamma(x)=s(x)(\ln x)^{-1}$. Then \eqref{sx} implies  $\mathrm{(I)}.$ Moreover,
note $\Gamma(x)\ln x=s(x),$ which together with the fact $s(x)\in\mathbb{N}$ is non-decreasing, implies $\mathrm{(II)}.$

Now we  prove $\mathrm{(III)}$.
For any  $y>x$, by the fact $s(x)\in\mathbb{N}$ is non-decreasing, we can distinguish the proof into two cases:

\noindent
$\textbf{Case 1}: s(y)=s(x).$ By the definitions of $\Lambda(x)$ and $s(x)$,  we get
\begin{equation*}
\begin{split}
\Lambda(y)-\Lambda(x)=(\ln y-\ln x)s(x)=(\ln y-\ln x)\Gamma(x)\ln x.
\end{split}
\end{equation*}

\noindent
$\mathbf{Case\ 2}:\ s(y)\geq s(x)+1.$ The inequality
on the left hand of \eqref{appedix21} and the fact that $\{M_{s+1}/M_{s}\}_{s\in\mathbb{N}}$ is increasing imply
\begin{equation*}
\begin{split}
y\geq M_{s(y)}/M_{s(y)-1}\geq  M_{s(x)+1}/M_{s(x)},
\end{split}
\end{equation*}
that is $\ln y \geq \ln M_{s(x)+1}-\ln M_{s(x)}.$
Together with the definitions of $\Lambda(x)$ and $s(x)$, it yields
\begin{equation*}
\begin{split}
\Lambda(y)-\Lambda(x)
&\geq \ln(y^{s(x)+1}M_{s(x)+1}^{-1})-\ln(x^{s(x)}M_{s(x)}^{-1})\\
&=(\ln y-\ln x)s(x)+\ln y-(\ln M_{s(x)+1}-\ln M_{s(x)})\\
&\geq (\ln y-\ln x)s(x)=(\ln y-\ln x)\Gamma(x)\ln x.
\end{split}
\end{equation*}
We thus finish the whole proof.
\end{proof}

\begin{remark}
To give a  heuristic   understanding of the function $\Gamma$, we can assume that
$\Lambda$ is differentiable, by \eqref{appedix20}, we can rewrite it as  $$\Gamma(x)=x\Lambda'(x)(\ln x)^{-1}.$$
Now if we fix $\Lambda(x)=(\ln x)^\delta,$
then
$\Gamma(x)=x\Lambda'(x)(\ln x)^{-1}=\delta(\ln x)^{\delta-2}$ and $\mathrm{(I)}$ and $\mathrm{(II)}$
are equivalent to
$$\delta(\ln x)^{\delta-2}\to +\infty$$
and $\delta(\ln x)^{\delta-1}$ is non-decreasing, which means $\delta>2.$ \end{remark}

For the function $f\in C^{\infty}(\mathbb{T}, \mathbb{R})$ and any $ K \geq 1 $, we define the truncation operator $\mathcal{T}_{K}$ and projection operator $\mathcal{R}_{K}$ as
\begin{equation*}
\mathcal{T}_{K}f(\theta)=\sum_{k\in\ZZ,|k|<  K}\widehat{f}(k)\me^{2\pi\mathrm{i}k\theta},
\quad
\mathcal{R}_{K}f(\theta)=\sum_{k\in\ZZ,|k|\geq K} \widehat{f}(k)\me^{2\pi\mathrm{i}k\theta}.
\end{equation*}
We denote the average of $f(\theta)$ on $ \TT $ by
$[f(\theta)]_{\theta}=\int_{\mathbb{T}}f(\theta)\dif \theta=\widehat{f}(0).$
The norm of $\mathcal{R}_{K}f(\theta)$ in a shrunken regime has the following estimate for $C^\infty$ functions satisfying $\mathbf{(H1)}$ and $\mathbf{(H2)}.$

\begin{lemma}\label{projectionestimate}
Under the assumptions  $\mathbf{(A)},$ there exists $T_1=T_1(M)$, such that  for any $f\in U_{r}^{M}(\mathbb{T},\mathbb{R})$, if
$Kr\geq T_1$, then
\begin{equation}\label{inequivalentnorm}
\begin{aligned}
\|\mathcal{R}_{K}f\|_{M,r/2}\leq C(Kr^2)^{-1}\|f\|_{M,r}
\exp\{-9^{-1}\Gamma(4 Kr)\ln (4Kr)\}.
\end{aligned}
\end{equation}
Particularly,
\begin{equation}\label{inequivalentnormss}
\begin{aligned}
\|\mathcal{R}_{K}f\|_{C^{0}}\leq (Kr^2)^{-1} \|f\|_{M,r}\exp\{-\Lambda(\pi K r)\}.
\end{aligned}
\end{equation}
\end{lemma}
\begin{proof} First by $\mathrm{(II)}$ and $\mathrm{(III)}$ in $\mathbf{(A)},$ for any  $|k|\geq K,$
we have
\begin{eqnarray*}
&&\Lambda(|2\pi k|r)-\Lambda(|2\pi k|(7/8)r)\nonumber \\
&\geq&
\{\ln (|2\pi k|r)-\ln (|2\pi k|(7/8)r)\}\Gamma(|2\pi k|(7/8)r)\ln (|2\pi k|(7/8)r) \nonumber\\
&\geq&\ln (8/7) \Gamma(|4k|r)\ln (|4k|r)>9^{-1}\Gamma(4Kr)\ln (4Kr).
\end{eqnarray*}
Moreover, by $\mathrm{(I)}$ in $\mathbf{(A)}$, we know there exists  $T_1=T_1(M)$,   such that  if $Kr \geq T_1$
then $\Gamma(4|k|r)>18, \forall |k|\geq K,$
which implies
\begin{equation*}
\begin{aligned}
\Lambda(|2\pi k|(7/8)r)-\Lambda(|2\pi k|(3/4)r)> 2 \ln (|4k|r).
\end{aligned}
\end{equation*}
Consequently, direct calculations show  that
\begin{eqnarray}\label{remaindersss}
 && \sum_{|k|\geq K}\exp\{-\Lambda(|2\pi k|r)\}|2\pi k(3r/4)|^{s}M_{s}^{-1}\nonumber\\
&\leq&\sum_{|k|\geq K}\exp\{-\Lambda(|2\pi k|r)\}
\exp\{\Lambda(|2\pi k|(3/4)r)\}\nonumber\\
&\leq&\sup_{|k|\geq K}\exp\{-\Lambda(|2\pi k|r)+\Lambda(|2\pi k|(7/8)r)\}\nonumber\\
&&\sum_{|k|\geq K}\exp\{-\Lambda(|2\pi k|(7/8)r)+\Lambda(|2\pi k|(3/4)r)\}\nonumber\\
&\leq&\exp\{-9^{-1}\Gamma(4 Kr)\ln (4Kr)\}\sum_{|k|\geq K}|4kr|^{-2}\nonumber\\
&\leq&(4Kr^2)^{-1}\exp\{-9^{-1}\Gamma(4 Kr)\ln (4Kr)\}.
\end{eqnarray}
Finally, by \eqref{fouriercoefficient}, we have
\begin{equation*}
\begin{aligned}
\|D_{\theta}^{s}\mathcal{R}_{K}f\|_{C^{0}}
\leq\|f\|_{M,r}\sum_{|k|\geq K}\exp\{-\Lambda(|2\pi k|r)\}|2\pi k|^{s},
\end{aligned}
\end{equation*}
then
\begin{equation*}
\begin{aligned}
\|\mathcal{R}_{K}f\|_{M,r/2}
&=3^{-1}4\pi^{2}\sup_{s\in\mathbb{N}}\big((r/2)^{s}(1+s)^{2}
\|D_{\theta}^{s}\mathcal{R}_{K}f\|_{C^{0}}M_{s}^{-1}\big)\\
&\leq\|f\|_{M,r}\sup_{s\in\mathbb{N}}3^{-1}4\pi^{2}
(1+s)^{2}(2/3)^{s}\\
&\sum_{|k|\geq K}\exp\{-\Lambda(|2\pi k|r)\}\{|2\pi k|(3r/4)\}^{s}M_{s}^{-1}\\
&\leq C(Kr^2)^{-1}\|f\|_{M,r}
\exp\{-9^{-1}\Gamma(4 Kr)\ln (4Kr)\},
\end{aligned}
\end{equation*}
where the  last inequality follows from \eqref{remaindersss}.

The conclusion  \eqref{inequivalentnormss} follows from similar computations, we thus omit the details.
\end{proof}

For the given function $f\in U_{r}^{M}(\mathbb{T},\mathbb{R}),$ we define the $\|\cdot\|_{\Lambda,r}$-
norm by
\begin{equation}\label{20201022}
\begin{aligned}
\|f\|_{\Lambda,r}
=\sum_{k\in\mathbb{Z}}|\widehat{f}(k)|\me^{\Lambda(|2\pi k|r)},
\end{aligned}
\end{equation}
with $\Lambda$ being the one defined by \eqref{definelambda}.
With the help of Lemma \ref{finallemma}, we can discuss the relationship between the
 spaces $\|\cdot\|_{M,r}$ and
$\|\cdot\|_{\Lambda,r}.$
\begin{lemma}\label{twonormrelationn}
Under the assumptions $\mathbf{(H1)}$ and $\mathbf{(H2)}$, we have
\begin{eqnarray}
\label{tttwonormrelations}  \|f\|_{M,r} &\leq& C\|f\|_{\Lambda,2r},\\
\label{tttwonormrelations2} \|f\|_{\Lambda,r/2} &\leq& (2\pi r)^{-1}(4+c_{M})\|f\|_{M,r},
\end{eqnarray}
where $c_{M}$ is a constant that only depends on the sequence $M$.
\end{lemma}

\begin{proof}
First, \eqref{definelambda} implies that for any $y>0$, any $s\in \N$,
\begin{equation*}
\exp\{-\Lambda(y)\}y^{s} \leq M_{s},
\end{equation*}
which yields
\begin{equation*}
\begin{aligned}
\exp\{-\Lambda(yr)\}y^{s}=\exp\{-\Lambda(yr)\}(yr)^{s}r^{-s} \leq M_{s}r^{-s}, \forall yr>0.
\end{aligned}
\end{equation*}
Then
\begin{equation*}
\begin{aligned}
\sup_{k\in\mathbb{Z}}\exp\{-\Lambda(|2\pi k|2r)\} |2\pi k|^{s}\leq M_{s}(2r)^{-s}.
\end{aligned}
\end{equation*}
Thus for any $f\in U_{r}^{M}(\mathbb{T},\mathbb{R})$, for any $s\in\mathbb{N},$
we get
\begin{equation*}
\begin{aligned}
\|D_{\theta}^{s}f\|_{C^{0}}&
\leq\sum_{k\in\mathbb{Z}}|\widehat{f}(k)||2\pi k|^{s}\\
&\leq\sup_{k\in\mathbb{Z}}\exp\{-\Lambda(|2\pi k|2r)\}|2\pi k|^{s}
\sum_{k\in\mathbb{Z}}|\widehat{f}(k)| \exp\{\Lambda(|2\pi k|2r)\}\\
&\leq\|f\|_{\Lambda,2r}M_{s}(2r)^{-s},
\end{aligned}
\end{equation*}
which implies that
\begin{equation*}
\begin{aligned}
\|f\|_{M,r}&=3^{-1}4\pi^{2}\sup_{s\in\mathbb{N}}\big(r^{s}(1+s)^{2}
\|D_{\theta}^{s}f\|_{C^{0}}M_{s}^{-1}\big)\\
&\leq\|f\|_{\Lambda,2r}\sup_{s\in\mathbb{N}}3^{-1}4\pi^{2}2^{-s}(1+s)^{2}<C\|f\|_{\Lambda,2r}.
\end{aligned}
\end{equation*}

Now we turn to the inequality in \eqref{tttwonormrelations2}.
Easily, for $y\geq1,$ we have
\begin{equation*}
\begin{split}
\exp\{-\Lambda(2y)+\Lambda{(y)}\}&=\inf_{s\in\mathbb{N}}\{(2y)^{-s}M_{s}\} y^{s(y)}M_{s(y)}^{-1}\\
&\leq (2y)^{-(s(y)+2)} M_{s(y)+2}y^{s(y)} M_{s(y)}^{-1}\leq c_{M}(2y)^{-2},
\end{split}
\end{equation*}
where (by $\mathbf{(H2)}$)
\begin{equation*}
\begin{aligned}
c_{M}:=\sup_{s\in\mathbb{N}}\{2^{-s}M_{s+2}M_{s}^{-1}\}<\infty.
\end{aligned}
\end{equation*}
Note $\Lambda(y)=0,$ if $y\leq1$. Consequently, we have
\begin{equation*}
\begin{split}
\|f\|_{\Lambda,r/2}&=\sum_{k\in\Z}|\widehat{f}(k)|\exp\{\Lambda{(|\pi k|r)}\}\\
&\leq\|f\|_{M,r}\sum_{k\in\Z}\exp\{-\Lambda(|2\pi k|r)+\Lambda{(|\pi k|r)}\}\\
&=\|f\|_{M,r}\{\sum_{|k|<(\pi r)^{-1}}+\sum_{|k|\geq (\pi r)^{-1}}\}\exp\{-\Lambda(|2\pi k|r)+\Lambda{(|\pi k|r)}\}\\
&\leq 2(\pi r)^{-1}\|f\|_{M,r}+c_{M}\|f\|_{M,r}\sum_{|k|\geq (\pi r)^{-1}}(|2\pi k|r)^{-2}\\
&\leq2(\pi r)^{-1}\|f\|_{M,r}+(2\pi r)^{-1}c_{M}\|f\|_{M,r}.
\end{split}
\end{equation*}
%

\end{proof}

We have stated the above lemma only for $\|f\|_{\Lambda,r/2}$ in \eqref{tttwonormrelations2} as this is the
only case we shall need; but clearly one could obtain an estimate for any $\|\partial ^{s}f\|_{\Lambda,r/2},\ s\in\N,$ by the
similar discussions above.

\section{The inductive step}\label{Inthissection}

\subsection{Sketch of the proof}\label{Sketchoftheproof}

The proof of Theorem~\ref{mainresults} is based on a non-standard KAM scheme which was first developed in \cite{KWYZ18}.
Now let us  briefly  introduce  the main idea of the proof.
We start from the cocycle  $(\alpha, R_{\rho_f}\me^{F_n})$ with $\|F_n\|$ of size  $\varepsilon_{n}$,  to conjugate it into $(\alpha, R_{\rho_{f}} \me^{F_{n+1}})$ with a smaller perturbation, a
crucial ingredient is to solve the  homological equations
\begin{eqnarray}
\label{ho-di}  f_{1}(\cdot+\alpha)-f_{1}=- (g_{1}-[g_1]_{\theta}),
\end{eqnarray}
$$ \me^{4\pi \mi\rho_f}f_{2}(\cdot+\alpha)-f_{2}+g_{2}=0.$$
However, if $\alpha$ is Liouvillean,  \eqref{ho-di} cann't be solved at all, even if in the analytic category.
This is  essentially
different from the classical KAM scheme.  Therefore, we have to leave $g_{1}(\theta)$ (at least the resonant terms of $g_{1}(\theta)$) into the normal form.  As a result, from the second step of iteration we need to consider the modified cocycle $(\alpha, R_{\rho_f+(2\pi)^{-1}g(\theta)}\me^{F_{n}(\theta)})$,  thus the second  equation  in (\ref{ho-di}) is of the form
\begin{equation*}
\me^{2\mi(2\pi \rho_f+g(\theta))}f(\cdot+\alpha)-f+g_2=0.
\end{equation*}

In order to get desired result, we distinguish the discussions into three steps. In the first step we eliminate the
lower order terms of $g(\theta)\in U_r^{M}(\T,\R)$ by solving the equation
$$
v(\theta+\alpha)-v(\theta)=-(\mathcal{T}_{\overline{Q}_{n}}g-[g]_{\theta}).
$$
Although  $\|g(\theta)\|$ is of size  $\varepsilon_{0}$, $\|\me^{\mi v}\|_{r}$ could  be  very large in Liouvillean frequency case. To control $\|\me^{\mi v}\|_{r}$,
the trick is to control $\|\mathrm{Im}v(\theta)\|$ at the cost of
reducing the analytic radius greatly, which was first developed in analytic case in \cite{YouZ14}.
The key point here is that $v(\theta)$ is in fact  a trigonometric polynomial,  one can analytic  continue $v(\theta)$ to become a real analytic function, and  the ``width" $r$ just plays the role of
analytic radius. Therefore, one can shrink $r$ greatly in order to control $\|\me^{\mi v}\|_{r}$ (Lemma~\ref{rotationlemma}). Consequently, the  ``width" will go to zero rapidly, and the convergence of the KAM iteration only works in the $C^{\infty}$ category.

The second step is to make the perturbation much smaller by solving the homological equation
$$
\me^{2\mi(2\pi \rho_f+\widetilde{g}(\cdot))}f(\cdot+\alpha)-f+h=0,
$$
where $\|\widetilde{g}\|=O(\|F_{n}\|).$ By the method of
diagonally dominant \cite{KWYZ18}, we can solve its approximation equation and then to make the perturbation  as small as we desire (Lemma~\ref{iterationlemma}).

By these two steps, we can already get $C^{\infty}$ almost reducible result (Corollary \ref{corollarylocalalmost}).
However,  to get $C^{\infty}$ rotations reducible result,  at the end of one
KAM step  we need to inverse the first step, such that the conjugation is
close to the identity (Lemma~\ref{newlemma}).

For simplicity, in the following parts we will shorten $U_{r}^{M}(\T,*)$ and $\|\cdot\|_{M,r}$ as
$U_{r}(\T,*)$ and $\|\cdot\|_{r}$, also the letter  $C$ denotes  suitable (possibly different) large constant that do not depend on the iteration step.

\subsection{Main iteration lemma}
For the functions $\Lambda(x)$ and $\Gamma(x)$ in Lemma~\ref{finallemma}, by $\mathrm{(I)}$ in $\mathbf{(A)}$ we know that there exists
$\widetilde{T}\geq T_1$, where $T_1$ is defined in Lemma \ref{projectionestimate},  such that for any $x\geq \widetilde{T}$
\begin{eqnarray}
\label{newaboutchi} \Gamma(x) &\geq& 64\mathbb{A}^{8}\tau^{4}, \\
\label{newaboutchi2} \Lambda(x) &\geq& \ln x.
\end{eqnarray}

 Denote
\begin{equation}\label{qnestimate}
\begin{split}
T=\max\{c_{M}^3, \widetilde{T}^{3},\ (2^{-1}r)^{-12}, (4\gamma^{-1})^{2\tau}\},
\end{split}
\end{equation}
where $c_{M}$ is the one in \eqref{tttwonormrelations2}.  Then for the $T$ defined above,
we claim that there exists $n_{0}$ such that $Q_{n_{0}+1}\leq T^{\mathbb{A}^{4}}$ and $\overline{Q}_{n_{0}+1}\geq T.$
Indeed, let $m_{0}$ be such that $Q_{m_{0}}\leq T\leq Q_{m_{0}+1}.$
If $\overline{Q}_{m_{0}}\geq T$, then we set $n_{0}=m_{0}-1.$ Otherwise, if
$\overline{Q}_{m_{0}}\leq T,$ by the definition
of $(Q_{k})$, it then holds
$Q_{m_{0}+1}\leq T^{\mathbb{A}^{4}}.$ By the selection,
$\overline{Q}_{m_{0}+1}\geq T,$ then $n_{0}=m_{0}$
satisfy our needs. In the following we will shorten $n_{0}$ as $0,$ that is $\overline{Q}_{n}$ stands for $\overline{Q}_{n+n_{0}}.$

Without loss of generality we assume $0<r_{0}:=2^{-1}r\leq1$. Set
\begin{equation}\label{varepsilonnestimate}
\begin{split}
  \widetilde{\varepsilon}_{0}=0, \qquad   \varepsilon_{0}= T^{-8\mathbb{A}^{4}\tau^{2}},
\end{split}
\end{equation}
then  $\varepsilon_{0}$ just depends on $\gamma,\tau,r,M,$
but not on $\alpha.$ Once we have this, we can define the
iterative parameters as following, for $n\geq1$
\begin{equation}\label{parameter}
\begin{split}
\overline{r}_{n}=2\overline{Q}_{n}^{-2}r_{0},  & \qquad
r_{n}=\overline{Q}_{n-1}^{-2}r_{0}.\\
\varepsilon_{n}=\varepsilon_{n-1}\overline{Q}_{n}^{-\Gamma^{\frac{1}{2}}(\overline{Q}_{n}^{\frac{1}{3}})},
&\qquad \widetilde{\varepsilon}_{n}=C\sum_{l=0}^{n-1}\varepsilon_{l}.
\end{split}
\end{equation}
To simplify the notations,  for any $g\in C^{0}(\mathbb{T},\mathbb{R}),$ we denote
\begin{equation*}
\begin{split}
R_{g}:=\left(
\begin{matrix}
\cos2\pi g\ \ -\sin2\pi g
\\
\sin2\pi g\  \ \cos2\pi g
\end{matrix}
\right)=\me^{-2\pi g J},\ J=\left(
\begin{matrix}
0\ \ \ \  1
\\
-1\ \ \ 0
\end{matrix}
\right),
\end{split}
\end{equation*}
and set
\begin{equation*}
\begin{split}
\mathcal{F}_{r}(\rho_{f},\eta, \widetilde{\eta}):=
\Big\{(\alpha,\ R_{\rho_f+(2\pi)^{-1}g(\theta)}\me^{F(\theta)}):  &\ \|g\|_{r}\leq \eta,\|F\|_{r}\leq\widetilde{\eta},\\  & \rho_{f}=\rho(\alpha,\ R_{\rho_f+(2\pi)^{-1}g(\theta)}
\me^{F(\theta)})
\Big\}.
\end{split}
\end{equation*}
Then the main inductive lemma is the following:

\begin{proposition}\label{rotationproposition}
 Assume that
$ \rho_f \in DC_{\alpha}(\gamma,\tau)$, then for $n\geq1,$ the cocycle
\begin{equation}\label{totalbeforerotation}
\begin{split}
(\alpha,\ R_{\rho_{f}+(2\pi)^{-1}g_{n}}\me^{F_{n}(\theta)})
\in\mathcal{F}_{r_{n}}(\rho_{f},\widetilde{\varepsilon}_{n}, \varepsilon_{n}),
\end{split}
\end{equation}
with $\mathcal{R}_{\overline{Q}_{n}}g_{n}=0$ can be conjugated to
\begin{equation}\label{totalafterrotation}
\begin{split}
(\alpha,\ R_{\rho_{f}+(2\pi)^{-1}g_{n+1}}\me^{F_{n+1}(\theta)})
\in\mathcal{F}_{r_{n+1}}(\rho_{f},\widetilde{\varepsilon}_{n+1}, \varepsilon_{n+1}),
\end{split}
\end{equation}
with $\mathcal{R}_{\overline{Q}_{n+1}}g_{n+1}=0$ by the conjugation $\Phi_{n}$ with the estimate
\begin{equation}\label{totalrotationestimate}
\begin{split}
\|\Phi_{n}-I\|_{r_{n+1}}\leq C\varepsilon_{n}^{\frac{1}{2}}.
\end{split}
\end{equation}
\end{proposition}

The construction of the conjugation in  Proposition \ref{rotationproposition} is divided into three steps given in Lemma~\ref{rotationlemma}, Lemma~\ref{iterationlemma} and  Lemma~\ref{newlemma} of the following.

\begin{lemma}\label{rotationlemma}
For $n\geq1,$ the cocycle
\begin{equation}\label{beforerotation}
\begin{split}
(\alpha,\ R_{\rho_{f}+(2\pi)^{-1}g_{n}}\me^{F_{n}(\theta)})
\in\mathcal{F}_{r_{n}}(\rho_{f},\widetilde{\varepsilon}_{n}, \varepsilon_{n}),
\end{split}
\end{equation}
with $\mathcal{R}_{\overline{Q}_{n}}g_{n}=0$ can be conjugated to the cocycle
\begin{equation}\label{afterrotation}
\begin{split}
(\alpha,\ R_{\rho_{f}}\me^{\widetilde{F}_{n}(\theta)})
\in\mathcal{F}_{\overline{r}_{n}}(\rho_{f},0,
C\varepsilon_{n}),
\end{split}
\end{equation}
via the conjugation $(0,\me^{-v_{n}J})$ with $\|\me^{-v_{n}J}\|_{\overline{r}_{n}}\leq C.$
\end{lemma}

Before giving the proof of Lemma~\ref{rotationlemma} we give an auxiliary lemma.
To this end, for $f(\theta)= \sum_{k\in\mathbb{Z}}
\widehat{f}(k)\me^{2\pi\mi k \theta }\in U_{r}(\mathbb{T},\mathbb{R})$, we
set $\vartheta=\theta+\mi\widetilde{\theta}, (\theta\in\mathbb{T},|\widetilde{\theta}|\leq r)$
\begin{equation*}
\begin{aligned}
\widetilde{f}(\vartheta )
=\sum_{k\in\mathbb{Z}}\widehat{f}(k)\me^{2\pi \mi k(\theta+\mi\widetilde{\theta})}.
\end{aligned}
\end{equation*}
Then we, formally,
define the analytic norm
\begin{equation*}
\begin{aligned}
\|\widetilde{f}\|_{r}^{*}=\sum_{k\in\mathbb{Z}}
|\widehat{f}(k)|\sup_{|\widetilde{\theta}|\leq r,\theta\in\mathbb{T}}\big|\me^{2\pi \mi k(\theta+\mi\widetilde{\theta})}\big|
=\sum_{k\in\mathbb{Z}}|\widehat{f}(k)|\me^{|2\pi k|r}.
\end{aligned}
\end{equation*}
If $\mathrm{Im} \vartheta=\widetilde{\theta}=0,$
then $\widetilde{f}(\vartheta)=f(\theta),$ and if $0<|\mathrm{Im} \vartheta|\leq r,$ one has
\begin{equation*}
\begin{aligned}
\|f\|_{\Lambda,r}=\sum_{k\in\mathbb{Z}}
|\widehat{f}(k)|\me^{\Lambda(|2\pi k|r)}\leq \|\widetilde{f}\|_{r}^{*}.
\end{aligned}
\end{equation*}
In general, $\|\widetilde{f}\|_{r}^{*}=\infty$, however,  if $f$ is a trigonometric polynomial, then
$\widetilde{f}$ really defines a real analytic function in the strip $|\mathrm{Im} \vartheta|\leq r,$ motivated by this we have the following:

\begin{lemma}\label{abstractderivatives}
Assume that $v$ is the solution of
\begin{equation}\label{B_homeq}
v(\theta+\alpha)-v(\theta)=-(\mathcal{T}_{\overline{Q}_{n}}g-[g]_{\theta}),
\end{equation}
where $g\in U_{r_{n}}(\mathbb{T},\mathbb{R})$ with $\|g(\theta)\|_{r_{n}} \leq C \varepsilon_{0}. $
Then
\begin{equation*}
\|\me^{\mathrm{i}v(\theta)}\|_{\overline{r}_{n}}\leq C.
\end{equation*}
\end{lemma}

\begin{proof}
By comparing the Fourier coefficients of  \eqref{B_homeq} we have
\begin{equation*}
v(\theta)=\sum_{0<|k|<\overline{Q}_{n}}\widehat{v}(k)
\me^{2\pi\mi k\theta}=  -\sum_{0<|k|<\overline{Q}_{n}} \widehat{g}(k)(\me^{2\pi\mi k\alpha}-1)^{-1}\me^{2\pi\mi k\theta}
\end{equation*}
with estimate
\begin{equation}\label{coefficient}
|\widehat{v}(k)|\leq \overline{Q}_{n}|\widehat{g}(k)|,\ 0<|k|<\overline{Q}_{n}.
\end{equation}

For $\theta\in\TT,$ by the fact  $g(\theta) \in\R$, one has  $v(\theta)\in\mathbb{R}.$ Thus for the function

\begin{equation*}
\begin{split}
\widetilde{v}(\vartheta)-v(\theta)=\sum_{0<|k|<\overline{Q}_{n}}
\widehat{v}(k)\me^{2\pi\mi k\theta} \big( \me^{-2\pi k\widetilde{\theta}}-1\big),
\end{split}
\end{equation*}
we have $
\mathrm{Im}\widetilde{v}(\vartheta)=\mathrm{Im}(\widetilde{v}(\vartheta)-v(\theta)).
$ Consequently, by \eqref{coefficient}, we have:
\begin{equation*}
\begin{split}
\|\mathrm{Im}\widetilde{v}(\vartheta)\|_{4\overline{Q}_{n}^{-2}}^{*}
&
\leq\|\widetilde{v}(\vartheta)-v(\theta)\|_{4\overline{Q}_{n}^{-2}}^{*}\\
&=\sum_{0<|k|<\overline{Q}_{n}}
|\widehat{v}(k)|\sup_{|\widetilde{\theta}|\leq 4\overline{Q}_{n}^{-2},\theta\in\mathbb{T}}
\big|\me^{2\pi\mi k\theta}\big(\me^{-2\pi k\widetilde{\theta}}-1\big)\big|\\
&\leq\sum_{0<|k|<\overline{Q}_{n}}
\overline{Q}_{n}^{-1}|\widehat{g}(k)|16\pi|k|\\
&\leq16\pi\overline{Q}_{n}^{-1}\widetilde{T}(\pi r_{n})^{-1}\sum_{|k|<\widetilde{T}(\pi r_{n})^{-1}}|\widehat{g}(k)|\\
&\ +
16(\overline{Q}_{n} r_{n})^{-1}\sum_{\widetilde{T}(\pi r_{n})^{-1}\leq |k|<\overline{Q}_{n}}
|\widehat{g}(k)|\exp\{\Lambda(|\pi k|r_{n})\}\\
&\leq32\widetilde{T}(\overline{Q}_{n}r_{n})^{-1}\|g\|_{\Lambda,r_{n}/2},
\end{split}
\end{equation*}
where the third inequality follows by \eqref{newaboutchi2}.

By \eqref{tttwonormrelations2} of  Lemma \ref{twonormrelationn}, one can further compute
\begin{equation*}
\|\mathrm{Im}\widetilde{v}(\vartheta)\|_{4\overline{Q}_{n}^{-2}}^{*} \leq
32\widetilde{T}(2\pi\overline{Q}_{n})^{-1}(4+c_M) \overline{Q}_{n-1}^{4}r_0^{2}\|g\|_{r_{n}}\leq \|g\|_{r_{n}},
\end{equation*}
the last inequality follows by $\overline{Q}_{n}\geq\max\{T,\overline{Q}_{n-1}^{\mathbb{A}}\}, n\geq1$
(by Lemma \ref{knandknrelation} and choice of $\overline{Q}_{n_{0}}$).
Therefore, by  \eqref{tttwonormrelations} of  Lemma  \ref{twonormrelationn}, we have
\begin{equation*}
\begin{split}
\|\me^{\mi v(\theta)}\|_{\overline{r}_{n}}&\leq C\|\me^{\mi v(\theta)}\|_{\Lambda,2\overline{r}_{n}}
\leq C\|\me^{\mi \widetilde{v}(\vartheta)}\|_{2\overline{r}_{n}}^{*}
\leq C\|\me^{\mi \widetilde{v}(\vartheta)}\|_{4\overline{Q}_{n}^{-2}}^{*}
\\
&\leq C\exp\{\|\mathrm{Im}\widetilde{v}(\vartheta)\|_{4\overline{Q}_{n}^{-2}}^{*}\}<C\exp\{\|g\|_{r_{n}}\}<C.
\end{split}
\end{equation*}
\end{proof}

\textbf{Proof of Lemma~\ref{rotationlemma}:}
Assume that $v_{n}$ is the solution of $$v_n(\theta+\alpha)-v_n(\theta)=-(g_n(\theta)- \widehat{g}_{n}(0)).$$ Note $\mathcal{R}_{\overline{Q}_{n}}g_{n}=0$, then by Lemma \ref{abstractderivatives} we have
\begin{equation*}
\begin{split}
\|\me^{v_{n}J}\|_{\overline{r}_{n}}\leq C.
\end{split}
\end{equation*}
Direct computation shows that $(0, \me^{-v_{n}J})$
conjugates the cocycle \eqref{beforerotation} into
$(\alpha, R_{\rho_{f}+(2\pi)^{-1}\widehat{g}(0)}\me^{ \overline{F}_{n} }),$ with
$\overline{F}_{n}=\me^{-v_{n}J}F_{n}(\theta)\me^{v_{n}J}.$
Thus by Lemma \ref{banachalgebra}, we have
\begin{equation}\label{evjfestimate}
\|\overline{F}_{n}\|_{\overline{r}_{n}}
\leq\|\me^{-v_{n}J}\|_{\overline{r}_{n}}\|F_{n}\|_{\overline{r}_{n}}
\|\me^{v_{n}J}\|_{\overline{r}_{n}}
\leq C\varepsilon_{n}.
\end{equation}

On the other hand, since $\me^{-v_{n}J}$ is homotopic to the identity,
the fibered rotation number remains unchanged,
then by \eqref{rotationnumberper}, we have
\begin{equation*}
|(2\pi)^{-1}\widehat{g}_{n}(0)|\leq\|\overline{F}_{n}\|_{\overline{r}_{n}}
\leq  C \varepsilon_{n},
\end{equation*}
which means
\begin{equation}\label{rhofgoestimate}
 |\widehat{g}_{n}(0)|\leq  C \varepsilon_{n}.
\end{equation}
Also note  if $B, D$ are small $sl(2, \mathbb{R})$ matrices,
then there exists $E\in sl(2, \mathbb{R})$ such that
\begin{equation*}
\me^{B}\me^{D}=\me^{B+D+E},
\end{equation*}
where $E$ is a sum of terms at least 2 orders in $B,D.$ Consequently, by  \eqref{evjfestimate}, \eqref{rhofgoestimate} and Lemma \ref{banachalgebra},  there exists  $\widetilde{F}_{n}\in  U_{\overline{r}_{n}}(\mathbb{T}, sl(2, \mathbb{R}))$
such that $R_{\rho_{f}+(2\pi)^{-1}\widehat{g}(0)}\me^{\overline{F}_{n}}
=R_{\rho_{f}}\me^{\widetilde{F}_{n}}$ with estimate $\|\widetilde{F}_{n}\|_{\overline{r}_{n}}\leq C \varepsilon_{n}.$  \qed\\

Once we get \eqref{afterrotation}, we will further conjugate it to
another cocycle with much smaller perturbation. We will give a lemma which can be
applied to more general cocycles rather than just \eqref{afterrotation}.

\begin{lemma}\label{iterationlemma}
Consider the cocycle $
(\alpha,\ R_{\rho_{f}}\me^{\widetilde{F}(\theta)})$
with
 $\rho_{f}=\rho(\alpha,\ R_{\rho_{f}}\me^{\widetilde{F}(\theta)})\in DC_{\alpha}(\gamma,\tau)$ and
\begin{equation}\label{itisreallylast}
\begin{split}
\|\widetilde{F}\|_{\overline{r}_{n}}\leq 8^{-2}\gamma^2 Q_{n+1}^{-2\tau^{2}},\ n\geq0.
\end{split}
\end{equation}
 Then there is a conjugation map
$\Psi_{n}\in U_{r_{n+1}}(\mathbb{T},SL(2,\mathbb{R}))$ with
\begin{equation*}
\begin{split}
\|\Psi_{n}-I\|_{r_{n+1}}\leq\|\widetilde{F}\|_{\overline{r}_{n}}^{\frac{1}{2}},\ n\geq0,
\end{split}
\end{equation*}
such that $\Psi_{n}$ conjugates the cocycle $
(\alpha,\ R_{\rho_{f}}\me^{\widetilde{F}(\theta)})$
into
\begin{equation}\label{n+thfunctionals}
\begin{split}
(\alpha,\ R_{\rho_{f}+(2\pi)^{-1}\widetilde{g}_{n}}\me^{G(\theta)})
\in\mathcal{F}_{r_{n+1}}(\rho_{f},2\|\widetilde{F}\|_{\overline{r}_{n}},
\epsilon),
\end{split}
\end{equation}
with $\mathcal{R}_{\overline{Q}_{n+1}}\widetilde{g}_{n}=0$ and $n\geq0,$ where $\epsilon=C^{-2}\|\widetilde{F}\|_{\overline{r}_{n}}
\overline{Q}_{n+1}^{-\Gamma^{\frac{1}{2}}(\overline{Q}_{n+1}^{\frac{1}{3}})}.$
\end{lemma}

Before give
the proof of Lemma~\ref{iterationlemma}
we give one important lemma, which is about the estimate of small divisors and serves
as the fundamental ingredients of the proof.
Although the proof is quite simple, it is the key observation that to obtaining semi-local results.

\begin{lemma}\label{bettersmalldivisor}
For any $0<\gamma<1,\, \tau>1,$ assume that $\overline{Q}_{n+1}\geq T$ and
$ \rho\in DC_{\alpha}(\gamma,\tau)$, then for any
 $|k|\leq \overline{Q}_{n+1}^{\frac{1}{2}},$ we have
\begin{equation}\label{diophantine}
\begin{split}
\big|\me^{2\pi\mathrm{i}(k\alpha\pm2\rho)}-1\big|\geq \gamma Q_{n+1}^{-\tau^{2}}.
\end{split}
\end{equation}
\end{lemma}
\begin{proof}
We just need to estimate $\big|\me^{2\pi\mathrm{i}(k\alpha+2\rho)}-1\big|$
since
\begin{equation*}
\begin{split}
 \big|\me^{2\pi\mathrm{i}(k\alpha-2\rho)}-1\big|
 =\big|\me^{2\pi\mathrm{i}(-k\alpha+2\rho)}-1\big|.
\end{split}
\end{equation*}

$\textbf{Case\ 1.}\ \overline{Q}_{n+1}\leq Q_{n+1}^{2\tau}.$ Then our assumptions imply
\begin{equation*}
\begin{split}
\big|\me^{2\pi\mathrm{i}(k\alpha+2\rho)}-1\big|&=2|\sin\pi(k\alpha+2\rho)|>
\big\|k\alpha+2\rho\big\|_{\mathbb{Z}}
\geq\gamma\overline{Q}_{n+1}^{-2^{-1}\tau}
\geq\gamma Q_{n+1}^{-\tau^{2}}.
\end{split}
\end{equation*}

$\textbf{Case\ 2.}\ \overline{Q}_{n+1}>Q_{n+1}^{2\tau}.$
Write $k$ as $
k=\widetilde{k}+mQ_{n+1},\ m\in\mathbb{Z} $
with  $|\widetilde{k}|< Q_{n+1}.$ Then we have
\begin{equation*}
\begin{split}
|m|\leq
|k|/Q_{n+1}<\overline{Q}_{n+1}^{\frac{1}{2}}/Q_{n+1}.
\end{split}
\end{equation*}
Consequently, by the assumption that $ \rho\in DC_{\alpha}(\gamma,\tau)$, one has
\begin{equation*}
\begin{split}
\big|\me^{2\pi\mathrm{i}(k\alpha+2\rho)}-1\big|&>
\big\| \widetilde{k}  \alpha+mQ_{n+1} \alpha+2\rho\big\|_{\mathbb{Z}} \\
&\geq
\big\|\widetilde{k}\alpha+2\rho\big\|_{\mathbb{Z}}- |m| \|Q_{n+1}\alpha\|_{\mathbb{Z}}
\\
&\geq\gamma Q_{n+1}^{-\tau}-|m|/\overline{Q}_{n+1}
\geq\gamma Q_{n+1}^{-\tau}-\overline{Q}_{n+1}^{-\frac{1}{2}}Q_{n+1}^{-1}
>2^{-1}\gamma Q_{n+1}^{-\tau},
\end{split}
\end{equation*}
where the last inequality is by
\begin{equation*}
\begin{split}
\overline{Q}_{n+1}^{\frac{1}{2}}=\overline{Q}_{n+1}^{\frac{1}{2\tau}}\overline{Q}_{n+1}^{\frac{\tau-1}{2\tau}}
\geq2\gamma^{-1}Q_{n+1}^{\tau-1},
\end{split}
\end{equation*}
which is guaranteed by
$\overline{Q}_{n+1}\geq (2\gamma^{-1})^{2\tau}$ and $\overline{Q}_{n+1}>Q_{n+1}^{2\tau}.$
\end{proof}

Set $su(1,1)$ be the space consisting of matrices of the
form $\left(
\begin{matrix}
\mathrm{i} t\ \ \ \ \ v
\\
\overline{v}\ \ -\mathrm{i}t
\end{matrix}
\right)$ $(t\in\mathbb{R},\ v\in\mathbb{C}),$ we simply denote such a matrix by $\{t,v\}.$
Recall that $sl(2,\mathbb{R})$ is isomorphic to $su(1,1)$ by the rule $A\mapsto MAM^{-1},$ where
$M=\frac{1}{1+\mathrm{i}}\left(
\begin{matrix}
1\ \ -\mathrm{i}
\\
1 \quad\ \ \mathrm{i}
\end{matrix}
\right),$
and a simple calculation yields
\begin{equation*}
\begin{split}
M\left(
\begin{array}{l}
x\ \qquad y+z
\\
y-z \quad\ -x
\end{array}
\right)M^{-1}=\left(
\begin{array}{l}
\mathrm{i}z\ \qquad x-\mathrm{i}y
\\
x+\mathrm{i} y \quad\ -\mathrm{i} z
\end{array}
\right), x,y,z\in\mathbb{R}.
\end{split}
\end{equation*}

Motivated by Lemma \ref{bettersmalldivisor}, we can define the following non-resonant and resonant spaces.
\begin{equation*}
\begin{split}
\mathcal{B}_{r}^{(nre)}&=
\Big\{\{0,\mathcal{T}_{\overline{Q}_{n+1}^{\frac{1}{2}}}g(\theta)
\}: g\in U_{r}(\mathbb{T},\mathbb{C})\Big\},\\
\mathcal{B}_{r}^{(re)}&=\Big\{\{f(\theta),\mathcal{R}_{\overline{Q}_{n+1}^{\frac{1}{2}}}g(\theta)\}: g\in U_{r}(\mathbb{T},\mathbb{C}),\ f\in U_{r}(\mathbb{T},\mathbb{R})\Big\}.
\end{split}
\end{equation*}
It follows that $U_{r}(\mathbb{T},su(1,1))=\mathcal{B}_{r}^{(nre)}\oplus \mathcal{B}_{r}^{(re)}.$
In order to prove Lemma~\ref{iterationlemma}, we will need the following lemma:
\begin{lemma}\label{caiyouzhou19}
Assume that $A=\mathrm{diag}\{\me^{-2\pi \mi \rho},\me^{2\pi \mi \rho}\}$ and $g\in U_{r}(\mathbb{T},su(1, 1))$.
If $\rho\in DC_{\alpha}(\gamma,\tau),$ and
$$\|g\|_{r} \leq  8^{-2} \gamma^2 Q_{n+1}^{-2\tau^{2}}, \ n\geq 0,$$
then  there exist $Y\in\mathcal{B}_{r}^{(nre)}$ and $g^{(re)}\in\mathcal{B}_{r}^{(re)}$ such that
\begin{equation*}
\begin{split}
\me^{Y(\cdot+\alpha)}A\me^{g(\cdot)}\me^{-Y(\cdot)}=A\me^{g^{(re)}(\cdot)}
\end{split}
\end{equation*}
with $\|Y\|_{r}\leq \|g\|_{r}^{1/2}$ and  $\|g^{(re)}\|_{r}\leq 2\|g\|_{r}.$

\end{lemma}
The proof of this lemma, which involves the  homotopy method, is postponed to Appendix.
Similar proofs appeared in \cite{YouZ14,Dias06}.\\

\textbf{Proof of Lemma~\ref{iterationlemma}}:
Since $SL(2,\R)$ is isomorphic to $SU(1,1)$, instead of $
(\alpha,\ R_{\rho_{f}}\me^{\widetilde{F}(\theta)})$, we just consider  $(\alpha, A\me^{W(\theta)}),$ where
$A=MR_{\rho_{f}}M^{-1}=\mathrm{diag}\{\me^{-2\pi \mi \rho_{f}},\me^{2\pi \mi \rho_{f}}\}\in SU(1,1),
 W=M\widetilde{F}M^{-1}\in su(1,1).$

Since $\rho_{f}\in D_{\alpha}(\gamma,\tau)$ and $\|\widetilde{F}\|_{\overline{r}_{n}}\leq 8^{-2}\gamma^2 Q_{n+1}^{-2\tau^{2}}$,
by Lemma~\ref{caiyouzhou19}, there exist $Y\in\mc{B}_{\overline{r}_{n}}^{(nre)}$
and $W^{(re)}\in\mc{B}_{\overline{r}_{n}}^{(re)}$ such that $\me^{Y}$ conjugates $(\alpha,\ A\me^{W})$ to $(\alpha,\ A\me^{W^{(re)}})$ with
\begin{equation}\label{widetildefestimate}
\begin{split}
\|Y\|_{\overline{r}_{n}}\leq\|\widetilde{F}\|_{\overline{r}_{n}}^{\frac{1}{2}},\ \|W^{(re)}\|_{\overline{r}_{n}}\leq
2\|\widetilde{F}\|_{\overline{r}_{n}}.
\end{split}
\end{equation}

Denote $W^{(re)}(\theta)=\{\widetilde{f}(\theta),\mathcal{R}_{\overline{Q}_{n+1}^{\frac{1}{2}}}\widetilde{g}(\theta)\}
\in\mathcal{B}_{\overline{r}_{n}}^{(re)}.$
Thus by \eqref{widetildefestimate}
\begin{equation*}
\begin{split}
\|\widetilde{f}(\theta)\|_{\overline{r}_{n}}
\leq\|W^{(re)}(\theta)\|_{\overline{r}_{n}}\leq
2\|\widetilde{F}(\theta)\|_{\overline{r}_{n}}.
\end{split}
\end{equation*}
Note $\overline{Q}_{n+1}\geq\overline{Q}_{n}^{24}$ (by Lemma~\ref{knandknrelation}) and
$\overline{Q}_{n+1}\geq T\geq r_{0}^{-12}$ (by \eqref{qnestimate}) we get
\begin{equation*}
\begin{split}
\overline{Q}_{n+1}^{\frac{1}{2}}> \overline{Q}_{n+1}^{\frac{1}{3}}>4^{-1}\overline{Q}_{n}^{4}r_{0}^{-2}=\overline{r}_{n}^{-2}\gg1,\ n\geq0,
\end{split}
\end{equation*}
which implies
\begin{equation}\label{kr2}
\begin{split}
\overline{Q}_{n+1}^{\frac{1}{2}}\overline{r}_{n}
>\overline{Q}_{n+1}^{\frac{1}{3}}\geq T^{\frac{1}{3}}\geq\widetilde{T}\geq T_{1}.
\end{split}
\end{equation}
Set
$P(\theta)=W^{(re)}(\theta)-\{\mathcal{T}_{\overline{Q}_{n+1}}\widetilde{f}(\theta),0\},$ thus $\mathcal{T}_{\overline{Q}_{n+1}^{\frac{1}{2}}}P(\theta)=0,$ then by \eqref{inequivalentnorm} in Lemma \ref{projectionestimate},
we get
\begin{equation*}
\begin{split}
\|P(\theta)\|_{\overline{r}_{n}/2}
&\leq C(\overline{Q}_{n+1}^{\frac{1}{2}}\overline{r}_{n}^{2})^{-1}
\|P(\theta)\|_{\overline{r}_{n}}
\exp\{-9^{-1}\Gamma(4 \overline{Q}_{n+1}^{\frac{1}{2}}\overline{r}_{n})
\ln (4\overline{Q}_{n+1}^{\frac{1}{2}}\overline{r}_{n})\}\\
&\leq12C\|\widetilde{F}(\theta)\|_{\overline{r}_{n}}
\exp\{-9^{-1}\Gamma(\overline{Q}_{n+1}^{\frac{1}{3}})\ln (\overline{Q}_{n+1}^{\frac{1}{3}})\}\\
&\leq(2C^{2})^{-1}\|\widetilde{F}(\theta)\|_{\overline{r}_{n}}
\overline{Q}_{n+1}^{-\Gamma^{\frac{1}{2}}(\overline{Q}_{n+1}^{\frac{1}{3}})},
\end{split}
\end{equation*}
where the second inequality is by \eqref{kr2} and the fact that $\Gamma(x)\ln x$ is non-decreasing, i.e., $\mathrm{(II)}$ in $\mathbf{(A)},$ and the last inequality is by \eqref{newaboutchi}, that is $\Gamma(\overline{Q}_{n+1}^{\frac{1}{3}})>64\mathbb{A}^{8}\tau^{4}.$

Note
\begin{equation*}
\begin{split}
A\me^{W^{(re)}}
=A \me^{\{\mathcal{T}_{\overline{Q}_{n+1}}\widetilde{f},0\}}E,\ E=\me^{-\{\mathcal{T}_{\overline{Q}_{n+1}}\widetilde{f},0\}}
\me^{W^{(re)}}.
\end{split}
\end{equation*}
Then by Lemma \ref{banachalgebra} we have
\begin{equation*}
\begin{split}
\|E-I\|_{\overline{r}_{n}/2}
&\leq\me^{\|\mathcal{T}_{\overline{Q}_{n+1}}\widetilde{f}\|_{\overline{r}_{n}/2}}
\|\me^{W^{(re)}}
-\me^{\{\mathcal{T}_{\overline{Q}_{n+1}}\widetilde{f},0\}}\|_{\overline{r}_{n}/2}\\
&=\me^{\|\mathcal{T}_{\overline{Q}_{n+1}}\widetilde{f}\|_{\overline{r}_{n}/2}}
\|\me^{\{\mathcal{T}_{\overline{Q}_{n+1}}\widetilde{f},0\}+P}
-\me^{\{\mathcal{T}_{\overline{Q}_{n+1}}\widetilde{f},0\}}\|_{\overline{r}_{n}/2}\\
&\leq\me^{2\|\mathcal{T}_{\overline{Q}_{n+1}}\widetilde{f}\|_{\overline{r}_{n}/2}}
\me^{\|P\|_{\overline{r}_{n}/2}}
\|P\|_{\overline{r}_{n}/2}\\
&\leq2\|P\|_{\overline{r}_{n}/2}
\leq C^{-2}\|\widetilde{F}\|_{\overline{r}_{n}}
\overline{Q}_{n+1}^{-\Gamma^{\frac{1}{2}}(\overline{Q}_{n+1}^{\frac{1}{3}})}.
\end{split}
\end{equation*}
Thus by implicit function theorem, there exists  $\widetilde{G}\in U_{\overline{r}_{n}/2}(\mathbb{T},su(1,1))$
such that $E=\me^{\widetilde{G}}$ with
\begin{equation*}
\begin{split}
\|\widetilde{G}\|_{\overline{r}_{n}/2}\leq \|E-I\|_{\overline{r}_{n}/2}
<C^{-2}\|\widetilde{F}(\theta)\|_{\overline{r}_{n}}
\overline{Q}_{n+1}^{-\Gamma^{\frac{1}{2}}(\overline{Q}_{n+1}^{\frac{1}{3}})}.
\end{split}
\end{equation*}

Now we go back to $SL(2,\mathbb{R}).$
Let $\Psi_{n}=\me^{M^{-1}YM}.$ Then $\|\Psi_{n}-I\|_{\overline{r}_{n}}\leq \|Y\|_{\overline{r}_{n}}.$ Moreover, $\Psi_{n}$ conjugates the
cocycle $(\alpha,\ R_{\rho_{f}}\me^{\widetilde{F}(\theta)})$ to \eqref{n+thfunctionals} with
$G=M^{-1}\widetilde{G}M, \widetilde{g}_{n}(\theta)=-\mathcal{T}_{\overline{Q}_{n+1}} \widetilde{f}(\theta).$ Obviously, $\mathcal{R}_{\overline{Q}_{n+1}}\widetilde{g}_{n}=0.$\qed

\bigskip
To ensure the composition of the conjugations is close to the identity, we do one more conjugation which is the inverse of transformation in
Lemma~\ref{rotationlemma}:

\begin{lemma}\label{newlemma}
Assume that $v_{n}$ is the one defined in Lemma~\ref{rotationlemma}.  Then for any $n\geq 1$,  $(0,\me^{v_{n}(\theta)J})$ further conjugates the cocycle
\begin{equation*}
\begin{split}
(\alpha,\ R_{\rho_{f}+(2\pi)^{-1}\widetilde{g}_{n}}\me^{G(\theta)})\in
\mathcal{F}_{r_{n+1}}(\rho_{f},C\varepsilon_{n}, C^{-1}\varepsilon_{n+1}),
\end{split}
\end{equation*}
with $\mathcal{R}_{\overline{Q}_{n+1}}\widetilde{g}_{n}=0$, to the cocycle
\begin{equation*}
\begin{split}
(\alpha, R_{\rho_{f}+(2\pi)^{-1}g_{n+1}}
\me^{F_{n+1}})\in\mathcal{F}_{r_{n+1}}(\rho_{f},\widetilde{\varepsilon}_{n+1},\ \varepsilon_{n+1})
\end{split}
\end{equation*}
with $\mathcal{R}_{\overline{Q}_{n+1}}g_{n+1}=0.$
\end{lemma}
\begin{proof}
Since $v_{n}$ is the solution of $v_n(\theta+\alpha)-v_n(\theta)=-g_n(\theta)+ \widehat{g}_{n}(0),$ then  $(0,\me^{v_{n}(\theta)J})$ conjugates the cocycle $(\alpha,\ R_{\rho_{f}+(2\pi)^{-1}\widetilde{g}_{n}}\me^{G(\theta)})$ to
\begin{equation*}
\begin{split}
(\alpha, R_{\rho_{f}+(2\pi)^{-1}(\widetilde{g}_{n}+g_{n}-\widehat{g}_{n}(0))}
\me^{F_{n+1}(\theta)}),
\end{split}
\end{equation*}
where $F_{n+1}=\me^{v_{n}(\theta)J}G\me^{-v_{n}(\theta)J}$.  Let $g_{n+1}=\widetilde{g}_{n}+g_{n}-\widehat{g}_{n}(0),$  then by Lemma \ref{banachalgebra} and Lemma \ref{rotationlemma}, we have estimates
\begin{equation*}
\begin{split}
\|g_{n+1}\|_{r_{n+1}}&\leq \|\widetilde{g}_{n}\|_{r_{n+1}}+\|g_{n}-\widehat{g}_{n}(0)\|_{r_{n}}
\leq C \varepsilon_{n}
+\widetilde{\varepsilon}_{n}= \widetilde{\varepsilon}_{n+1},
\\
\|F_{n+1}\|_{r_{n+1}}&\leq \|\me^{v_{n}J}\|_{r_{n+1}}^{2}\|G\|_{r_{n+1}}\leq  \varepsilon_{n+1}.
\end{split}
\end{equation*}
Obviously, $\mathcal{R}_{\overline{Q}_{n+1}}g_{n+1}=0.$
 Moreover, the fibered rotation number does not change since $\me^{v_{n}J}$ is homotopic to the identity.
\end{proof}

Now we are in the position to prove Proposition~\ref{rotationproposition}.
First by Lemma \ref{rotationlemma},  $(0,\me^{-v_{n}J})$ conjugates the cocycle \eqref{totalbeforerotation} to
\begin{equation*}
(\alpha,\ R_{\rho_{f}}\me^{\widetilde{F}_{n}(\theta)})
\in\mathcal{F}_{\overline{r}_{n}}(\rho_{f},0,
C\varepsilon_{n}).
\end{equation*}
Moreover, by our definition of $\varepsilon_{n}$, one can easily check that
\begin{equation*}
C\varepsilon_{n}=C\varepsilon_{n-1}\overline{Q}_{n}^{-\Gamma^{\frac{1}{2}}(\overline{Q}_{n}^{\frac{1}{3}})}
\leq C\overline{Q}_{n}^{-8\mathbb{A}^{4}\tau^{2}}\leq 8^{-2}\gamma^2 Q_{n+1}^{-2\tau^{2}},\ n\geq1,
\end{equation*}
the last inequality holds since, by Lemma \ref{bridgeestimate}, $\overline{Q}_{n}^{\mathbb{A}^{4}}\geq Q_{n+1}$
and $\overline{Q}_{n}\geq T\geq (4\gamma^{-1})^{2\tau}, n\geq1.$
That is \eqref{itisreallylast} holds with $C\varepsilon_{n}$ in place of
$\|\widetilde{F}_{n}\|_{\overline{r}_{n}}.$
Then by the assumption $ \rho_f \in DC_{\alpha}(\gamma,\tau)$,
one can apply Lemma~\ref{iterationlemma}, and there exists $\Psi_{n}\in U_{r_{n+1}}(\mathbb{T},SL(2,\mathbb{R}))$ with
\begin{equation*}
\|\Psi_{n}-I\|_{r_{n+1}}\leq C \varepsilon_{n}^{\frac{1}{2}},
\end{equation*}
which further conjugates the obtained cocycle into
\begin{equation*}
\begin{split}
(\alpha,\ R_{\rho_{f}+(2\pi)^{-1}\widetilde{g}_{n}}\me^{G(\theta)})
& \in\mathcal{F}_{r_{n+1}}(\rho_{f},2C\varepsilon_{n},
C^{-2} C\varepsilon_{n}
\overline{Q}_{n+1}^{-\Gamma^{\frac{1}{2}}(\overline{Q}_{n+1}^{\frac{1}{3}})})\\
&= \mathcal{F}_{r_{n+1}}(\rho_{f},C\varepsilon_{n}, C^{-1}\varepsilon_{n+1}).
\end{split}
\end{equation*}
Finally, by Lemma \ref{newlemma}, $(0,\me^{v_{n}(\theta)J})$ further conjugates the cocycle
above to \eqref{totalafterrotation} with desired estimates.
 Let $\Phi_{n}=\me^{v_{n}(\theta)J}\Psi_{n}\me^{-v_{n}(\theta)J},$ then  by Lemma \ref{banachalgebra} and Lemma \ref{rotationlemma},    we have
\begin{equation*}
\begin{split}
\|\Phi_{n}-I\|_{r_{n+1}}&=\|\me^{v_{n}(\theta)J}(\Psi_{n}-I)\me^{-v_{n}(\theta)J}\|_{r_{n+1}}\\
&\leq\|\me^{v_{n}(\theta)J}\|_{\overline{r}_{n}}^{2}\|\Psi_{n}-I\|_{r_{n+1}}
<C\varepsilon_{n}^{\frac{1}{2}},
\end{split}
\end{equation*}
which finishes the whole proof.\qed

\section{Proof of Theorem~\ref{mainresult} and  Theorem ~\ref{mainresults}}\label{Theoremmainresults}

\subsection{Proof of Theorem ~\ref{mainresults}}\label{Ttheoremmainresults}

Set $A=R_{\varrho}\me^{F}.$
By the assumption $\rho(\alpha,\ R_{\varrho}\me^{F})=\rho_{f}$ and  \eqref{rotationnumberper} one has $|\rho_{f}-\varrho|\leq 2\|F\|_{C^{0}},$ thus one can rewrite $(\alpha, R_{\varrho}\me^{F})$  as $(\alpha, R_{\rho_{f}}\me^{\widetilde{F}})$ with $\|\widetilde{F}\|_{r}\leq C \|F\|_{r}.$
Set $\varepsilon_{*}:=C^{-1}\varepsilon_{0},$ where
$\varepsilon_{0}$ is the one defined by \eqref{varepsilonnestimate}.

Set $\overline{r}_{0}=r,$ by the selection of $\varepsilon_{0}$ and $Q_{1}\leq T^{\mathbb{A}^{4}},\ T\geq (4\gamma^{-1})^{2\tau},$ we get
\begin{equation*}
\begin{split}
\|\widetilde{F}\|_{\overline{r}_{0}}\leq C \varepsilon_{*}=
 \varepsilon_{0}=T^{-8\mathbb{A}^{4}\tau^{2}}\leq 8^{-2} \gamma^2 Q_{1}^{-2\tau^{2}} .
\end{split}
\end{equation*}
Since we further assume $\rho_f \in DC_{\alpha}(\gamma,\tau)$, one can apply Lemma~\ref{iterationlemma}, then there exists  $\Psi_{0}\in U_{r_{1}}(\mathbb{T},SL(2,\mathbb{R}))$ with
\begin{eqnarray*}
\|\Psi_{0}-I\|_{r_{1}}\leq C\varepsilon_{0}^{\frac{1}{2}},
\end{eqnarray*}
which conjugates the cocycle $(\alpha, R_{\rho_{f}}\me^{\widetilde{F}})$ into
\begin{equation*}
\begin{split}
(\alpha,\  R_{\rho_{f}+(2\pi)^{-1}\widetilde{g}_{0}}\me^{G_{0}})
\in \mathcal{F}_{r_{1}}(\rho_{f},\widetilde{\varepsilon}_{1},
\varepsilon_{1}).
\end{split}
\end{equation*}
We emphasize that in the first iteration step, we only apply Lemma~\ref{iterationlemma}, without applying Proposition~\ref{rotationproposition}, which is quite different from the rest steps.

Now we set $\widetilde{g}_{0}=g_{1},\ G_{0}=F_{1},$ and $\Phi_{0}=\Psi_{0}.$ Then one can apply  Proposition~\ref{rotationproposition} inductively,
and get a sequence of transformations $\{\Phi_{n}\}_{n\geq0}$ with estimate
$\|\Phi_{n}-I\|_{r_{n+1}}\leq C\varepsilon_{n}^{\frac{1}{2}}.$ Furthermore, let
\begin{equation*}
\begin{split}
\Phi^{(n)}=\Phi_{n-1}\circ\Phi_{n-2}\circ\cdots\circ\Phi_{0},\ \Phi=\lim_{n\rightarrow\infty}\Phi^{(n)},
\end{split}
\end{equation*}
then $\Phi^{(n)}$ conjugates the original cocycle $(\alpha, \ R_{\rho_{f}}\me^{\widetilde{F}})$ to $(\alpha,\ R_{\rho_{f}+(2\pi)^{-1}g_{n}}\me^{F_{n}(\theta)})$.

Finally,  let's  show the convergence of $\Phi^{(n)}.$
Let $\Phi=\lim_{n\rightarrow\infty}\Phi^{(n)},$ we will show $\Phi \in C^{\infty}(\T, SL(2,\R))$.  Indeed, by the definition of $\|\cdot\|_{r}-$norm we have
\begin{equation}\label{20200613}
\begin{split}
\|D_{\theta}^{j}f\|_{C^{0}}\leq\|f\|_{r}r^{-j}M_{j},\ \forall f\in U_{r}(\mathbb{T}, SL(2,\mathbb{R})),
\end{split}
\end{equation}
and by $\mathrm{(I)}$ of $\mathbf{(A)}$, for any $j\in\N$,
there exists $n_{j}\in\mathbb{N},$
such that for any $n\geq n_{j}$, we have $CM_{j}\leq\overline{Q}_{n}^{j},$ and
$\Gamma^{\frac{1}{2}}(\overline{Q}_{n}^{\frac{1}{3}})\geq24j$.
By \eqref{totalrotationestimate} and standard computation, we get
$
\|\Phi^{(n+1)}-\Phi^{(n)}\|_{r_{n+1}}\leq C\varepsilon_{n}^{\frac{1}{2}},
$
then by \eqref{20200613} we can further compute
\begin{equation*}
\begin{split}
\big\|D^{j}(\Phi^{(n+1)}-\Phi^{(n)})\big\|_{C^{0}}&\leq\|\Phi^{(n+1)}-\Phi^{(n)}\|_{r_{n+1}}M_{j}r_{n+1}^{-j}
\leq C\varepsilon_{n}^{\frac{1}{2}}M_{j}r_{n+1}^{-j}\\
&=\overline{Q}_{n}^{-\frac{1}{2}\Gamma^{\frac{1}{2}}(\overline{Q}_{n}^{\frac{1}{3}})}
\varepsilon_{n-1}^{\frac{1}{2}}
CM_{j}\overline{Q}_{n}^{2j}r_{0}^{-j}\\
&<\overline{Q}_{n}^{-\frac{1}{2}\Gamma^{\frac{1}{2}}(\overline{Q}_{n}^{\frac{1}{3}})}
\varepsilon_{n-1}^{\frac{1}{2}}\overline{Q}_{n}^{4j}<
\overline{Q}_{n}^{-\frac{1}{3}\Gamma^{\frac{1}{2}}(\overline{Q}_{n}^{\frac{1}{3}})}\varepsilon_{n-1}^{\frac{1}{3}}
=\varepsilon_{n}^{\frac{1}{3}},
\end{split}
\end{equation*}
which means $\Phi \in C^{\infty}(\T, SL(2,\R))$.  Let $g_{\infty}=\lim_{n\rightarrow\infty} g_n$, then $\Phi$ conjugates the cocycle $(\alpha, R_{\rho_{f}}\me^{\widetilde{F}})$ to
$(\alpha, R_{\rho_{f}+(2\pi)^{-1}g_{\infty}})$, where $g_{\infty} \in C^{\infty}(\T, \R)$.
\qed

\bigskip

Note the proof of the proposition~\ref{rotationproposition} is separated into three steps and if we just
manipulate the first two steps, we get the local almost reducibility.

\begin{corollary}\label{corollarylocalalmost}
Under the assumptions of Theorem~\ref{mainresults},
there exists a sequence of $B_{\ell}\in U_{r_{\ell}}(\T, SL(2,\R))$ transforming $(\alpha, A)$ into $(\alpha, R_{\rho_{f}}\me^{F_{\ell}})$
with estimates
\begin{equation}\label{20201225s}
\|B_{\ell}\|_{r_{\ell}}\leq C,\ \|F_{\ell}\|_{r_{\ell}}\leq \varepsilon_{\ell}.
\end{equation}
\end{corollary}

\begin{proof}
Set $B_{\ell}=\Psi_{\ell-1}\me^{v_{\ell-1}J}\Phi^{(\ell-1)},$ where $\Psi_{\ell-1}, v_{\ell-1}$ are the ones
in Section~\ref{Inthissection} and $\Phi^{(\ell-1)}$ is the one defined above with $\ell-1$ in place of $n.$
Obviously, $B_{\ell}$ transforming $(\alpha, A)$ into $(\alpha, R_{\rho_{f}}\me^{F_{\ell}})$
and the estimates of $B_{\ell}$ and $F_{\ell}$ follow from the estimates of $\Psi_{\ell-1}, v_{\ell-1}$ and $\Phi^{(\ell-1)}.$
\end{proof}

\subsection{Proof of Theorem~\ref{mainresult}}\label{ProofofTheorem}
The proof of Theorem~\ref{mainresult} relies on the renormalization theory of one-frequency quasiperiodic $SL(2,\R)$ cocycles.
 Recall for any $0<\gamma<1$ and $\tau>1,$   $DC_{\alpha}(\gamma,\tau)$ denotes  the set
of all $\rho$ such that
\begin{equation*}
\begin{aligned}
\|k\alpha\pm2\rho\|_{\mathbb{Z}}
\geq \gamma \langle k\rangle^{-\tau},\ \langle k\rangle=\max\{1,|k|\},\ \forall k\in\mathbb{Z}.
\end{aligned}
\end{equation*}
Let $\mathcal{P}\subset[0,1/2)$ be the set of all $\rho$ such that there exist
$0<\gamma<1$ and $\tau>1$ with $\rho\beta_{n-1}^{-1}\in DC_{\alpha_{n}}(\gamma,\tau)$
for infinitely many $n.$ By Borel-Cantelli lemma,
$\mathcal{P}$ is full measure in $[0,1/2)$.
We will fix the sequence
$\{n_{j}\}_{j\in\mathbb{N}}$ such that $\beta_{n_{j}-1}^{-1}\rho_{f}\in DC_{\alpha_{n_{j}}}(\gamma,\tau).$

We also recall the following well-known Kotani's theory \cite{Kotani84}.
\begin{theorem} [\cite{Kotani84}]
Let $\widetilde{\mathcal{P}} \subset [0,1/2)$ be any full measure subset.  For every
$V \in C^{\infty}(\T,\R)$,
for almost every $E \in \R$, we have
\begin{itemize}
\item either $(\alpha,S_{E}^{V})$ has a positive Lyapunov exponent, or
\item $(\alpha,S_{E}^{V})$ is
$L^2$-conjugated to an $\mathrm{SO}(2,\R)$-valued cocycle and the fibered
rotation number of $(\alpha,S_{E}^{V})$ belongs to $\widetilde{\mathcal{P}}.$
\end{itemize}
\end{theorem}

We start from $(\alpha,S_{E}^{V})$  which can be
$L^2$-conjugated to an $\mathrm{SO}(2,\R)$-valued cocycle. By definition of $\mathcal{P}$, if $\rho(\alpha,S_{E}^{V})=\rho_f$
belongs to $\mathcal{P}$, we can find $0<\gamma<1$ and $\tau>1$, and arbitrary large $j>0$, such that $\rho_f\beta_{n_{j}-1}^{-1}\in DC_{\alpha_{n_{j}}}(\gamma,\tau)$.
Now Proposition \ref{llambdarenormalizations} ensures that
$
\|F_{n_{j}}\|_{rK_{*}^{-2}}\rightarrow0
$, then we can further choose $j$ large enough, such that
\begin{equation*}
\|F_{n_{j}}\|_{rK_{*}^{-2}}\leq \varepsilon_{*}(\gamma,\tau,rK_{*}^{-2},M),
\end{equation*}
where $\varepsilon_{*}=\varepsilon_{*}(\gamma,\tau,r,M)>0$ is the one in Theorem~\ref{mainresults}.
Since $(-1)^{n_{j}}\rho_f\beta_{n_{j}-1}^{-1}$ is just the rotation number of $(\alpha_{n_{j}}, R_{\rho_{n_{j}}}\me^{F_{n_{j}}})$,  by Theorem~\ref{mainresults} we know that $(\alpha_{n_{j}}, R_{\rho_{n_{j}}}\me^{F_{n_{j}}})$ is $C^{\infty}$ rotations reducible.
Note $(\alpha_{n_{j}}, R_{\rho_{n_{j}}}\me^{F_{n_{j}}})$ is rotations reducible (or reducible) implies $(\alpha,S_{E}^{V})$  is rotations reducible (or reducible) in the same regularity class (consult Proposition 4.2 of \cite{Krikorian15} for example), then Theorem~\ref{mainresult} follows directly.
 \qed

\section{Last's intersection spectrum conjecture}\label{Lastsintersection}
Consider the Schr\"odinger operator $H_{V,\beta,\theta}$ defined by \eqref{20200824operator} with ultra-differentiable potential $V\in U_{r}(\TT,\RR),$ frequency $\beta\in\TT$ and phase $\theta\in\TT.$
For fixed $\theta,$ denote by $\sigma(\beta,\theta)$ and $\sigma_{ac}(\beta,\theta)$ the spectrum of $H_{V,\beta,\theta}$
and its
absolutely continuous (ac)-component, respectively. It is well known that
in the case $\beta=p/q,$ $\sigma(p/q,\theta)$ is purely absolutely continuous and consists
of $q,$ possibly touching, bands. Moreover, in the case $\beta=\alpha$ is
irrational, the spectrum and ac spectrum do not depend on $\theta:$
\begin{equation*}
\begin{aligned}
\sigma(\alpha,\theta)=:\Sigma(\alpha), \ \sigma_{ac}(\alpha,\theta)=:\Sigma_{ac}(\alpha),\ \forall\theta\in\mathbb{T}.
\end{aligned}
\end{equation*}

In order to treat rational and irrational frequencies on the same footing, similar
to Avron et al. \cite{AvronS90}, given $\beta\in\mathbb{T},$ we introduce the sets
\begin{equation*}
\begin{aligned}
S_{+}(\beta):=\bigcup_{\theta\in\mathbb{T}}\sigma(\beta,\theta)=\Sigma(\beta),
\end{aligned}
\end{equation*}
and
\begin{equation*}
\begin{aligned}
S_{-}(\beta):=\bigcap_{\theta\in\mathbb{T}}\sigma_{ac}(\beta,\theta)=\Sigma_{ac}(\beta).
\end{aligned}
\end{equation*}
Note that it was proved in \cite{JitomirskayaM12} that 
$$
S_+(\alpha)=\Sigma(\alpha)=\lim\limits_{n\rightarrow\infty}S_+(p_n/q_n).
$$

Theorem~\ref{lastintersection} follows immediately from the following Theorem \ref{addproposition} and Theorem \ref{addpropositions}, while the key arguments are  ``generalized Chambers' formula" (Proposition \ref{differenceestimate}) and continuity of Lyapunov exponent  (Theorem \ref{continuity}).

\subsection{Generalized Chambers' formula}

\begin{theorem}\label{addproposition}
Let $\alpha\in \R \setminus\mathbb{Q}$,  $ V:\mathbb{T} \rightarrow  \R$ be an M-ultra-differentiable function satisfying  $\mathbf{(H1)}$ and $\mathbf{(H2)}$, then we have
\begin{equation*}
\begin{aligned}
S_{-}(\alpha)=\Sigma_{ac}(\alpha)\subset\liminf_{n\rightarrow\infty}S_{-}(p_{n}/q_{n}).
\end{aligned}
\end{equation*}
\end{theorem}

The proof of Theorem \ref{addproposition} depends on the following generalized Chambers' formula. To state this, recall that for each $\theta\in\mathbb{T},$ $H_{V,p/q,\theta}$ is a periodic operator whose
spectrum, $\sigma(p/q,\theta),$ is given in terms of the discriminant by
\begin{equation*}
\begin{aligned}
\sigma(p/q,\theta)=t_{p/q}(\cdot,\theta)^{-1}[-2,2],
\end{aligned}
\end{equation*}
where
\begin{equation}\label{20200603}
\begin{aligned}
t_{p/q}(E,\theta)=\mathrm{tr}\{\Pi_{s=q-1}^{0}
S^{V}_{E}(\theta+sp/q)\},
\end{aligned}
\end{equation}
which is called as the discriminant of $H_{V,p/q,\theta}$, here $``\mathrm{tr}"$ stands for the trace. In general, this discriminant is a polynomial of degree $q$ in $E$
and $q^{-1}-$periodic in $\theta$, whence one may write
\begin{equation}\label{20200606}
\begin{aligned}
t_{p/q}(E,\theta)=\sum_{k\in\mathbb{Z}}a_{q,k}(E)\me^{2\pi \mi qk\theta}.
\end{aligned}
\end{equation}

For the almost Mathieu operator, the potential $V=2 \lambda \cos 2\pi \theta$ is in fact a trigonometric polynomial of degree 1. Thus in the formula \eqref{20200606} only the Fourier coefficients with $k=0,\pm 1$ survive, resulting  the celebrated Chamber formula \cite{Chambers65,Bellissard82,Jitomirskayahan19}
\begin{equation*}
\begin{aligned}
t_{p/q}(E,\theta)=a_{q,0}(E)+2\lambda^{q}\cos(2\pi q\theta).
\end{aligned}
\end{equation*}
Note the classical Chamber's  formula holds for any $\lambda$.  In particular, it shows that phase variations of the discriminant for the subcritical almost Mathieu operator (thus has absolutely continuous spectrum) are exponentially small in q. Now for any $C^{\infty}$ potential $ V:\mathbb{T} \rightarrow  \R$  satisfying  $\mathbf{(H1)}$ and $\mathbf{(H2)}$, $E\in \Sigma_{ac}(\alpha)$, we will show that the difference between the determines $t_{p/q}(E,\theta)$ of rational approximates of $\alpha$, and its phase-average $a_{q,0}(E)$, is in fact sub-exponentially small in $q$:

\begin{proposition}\label{differenceestimate}
Let $\alpha\in \R \setminus\mathbb{Q}$,  $ V:\mathbb{T} \rightarrow  \R$ be an M-ultra-differentiable function satisfying  $\mathbf{(H1)}$ and $\mathbf{(H2)}$, then for almost every $E\in \Sigma_{ac}(\alpha)$,  there exist  $n_{*}=n(V,\alpha,E)\in \mathbb{N}$, $c=c(E)$ such that
\begin{equation}\label{20200607}
\begin{aligned}
\|t_{p_{n}/q_{n}}(E,\theta)-a_{q_{n},0}(E)\|_{C^{0}}\leq4
\exp\{-\Lambda( c q_{n})\}
\end{aligned}
\end{equation}
whenever  $n\geq n_{*}.$
\end{proposition}

\begin{remark}
Indeed, one can select $c(E)$, such that $c q_{n}>q_n^{\frac{3}{4}}.$
\end{remark}

If $V$ is analytic, Jitomirskaya-Marx (Proposition 3.1 in \cite{JitomirskayaM12}) proved that $\|t_{p_{n}/q_{n}}(E,\theta)-a_{q_{n},0}(E)\|_{C^{0}}$ is exponentially small in $q$. Their proof depends on Avila's quantization of acceleration \cite{Avila15}, key of his global theory. While our proof is a perturbation argument, completely different from theirs.

\bigskip

{\bf Proof of Theorem~\ref{addproposition}.}
We will first prove Theorem~\ref{addproposition} assuming Proposition \ref{differenceestimate} and postpone the proof of  Proposition \ref{differenceestimate} to Section \ref{chambers}. We  point out the ideas of the proof was essentially given by Avila and sketched in \cite{JitomirskayaM12}. We give the full proof here for completeness.
Let $\mathcal{K}\subset [0,1/2)$
be the set of all $\rho$ such that
\begin{equation*}
\inf_{ p\in\Z} |q_{n}\rho-p|\geq n^{-2}, \quad  \text{eventually.}
\end{equation*}
A simple Borel Cantelli argument shows $|\mathcal{K}|=1/2$. For any $\beta\in\T$, we denote by $N(\beta,E)$ the integrated density of states (IDS). Note the set $\mathcal{P}\subset [0,1/2)$ we defined in Section~\ref{ProofofTheorem} is also full of measure, $i.e.,$
$|\mathcal{P}|=1/2,$ thus $\mathcal{P}\doteq\mathcal{K}.$ Moreover, Theorem~\ref{mainresult} actually implies that for almost every $E\in \Sigma_{ac}(\alpha) \doteq \mathcal{P}$, the cocycle $(\alpha,S_{E}^V)$ is rotations reducible. Thus
by Theorem~6.1 in \cite{AvilaJ09},  if $E\in \Sigma_{ac}(\alpha) \doteq \mathcal{P} \doteq  \mathcal{K}$,   $N(\beta,E)$ is Lipschitz in  $\alpha,$ i.e., there exists some $\Gamma(E)$
\begin{equation}\label{e}
|N(\alpha,E)-N(p_n/q_n,E)|<q_n^{-2}\Gamma(E).
\end{equation}
Since $N(\alpha,E)\in \mathcal{K}$, then by \eqref{e}, for $n$ sufficiently large, we have
\begin{equation}\label{er}
p-1+\frac{1}{2q_n}<q_nN(p_n/q_n,E)<p-\frac{1}{2q_n},
\end{equation}
for some $1\leq p\leq q_n$. On the other hand, it was calculated in \cite{Avilad08} that if $E$ belongs to the $k$-th band of $S_+(p_n/q_n)$, we have
\begin{equation}\label{ids}
q_nN\big(p_n/q_n,E\big)=k-1+2(-1)^{q_n+k-1}\int_\T\rho
\big(p_n/q_n,E,\theta\big)d\theta+\frac{1-(-1)^{q_n-k+1}}{2},
\end{equation}
where
\begin{equation}\label{e4}
\rho\Big(p_n/q_n,E,\theta\Big)
=\left\{
\begin{aligned}
&0&t_{p_n/q_n}(E,\theta)>2,\\
&(2\pi)^{-1}\arccos(2^{-1}t_{p_n/q_n}(E,\theta))&|t_{p_n/q_n}(E,\theta)|\leq2,\\
&1/2&t_{p_n/q_n}(E,\theta)<-2.
 \end{aligned}
 \right.
\end{equation}
Then \eqref{er} and   \eqref{ids}  imply that
\begin{equation}\label{e2}
2\big|\cos\big(2\pi \int_\T\rho\big(p_n/q_n,E,\theta\big)d\theta\big)\big|<2- \frac{1}{q_n^{2}}.
\end{equation}
Since $\rho\big(p_n/q_n,E,\theta\big)$ is continuous in $\theta$,   \eqref{e2} and \eqref{e4}  imply that there exists
 $\tilde{\theta}\in \T$, such that
 \begin{equation*}
|t_{p_n/q_n}(E, \tilde{\theta})|= 2\big|\cos\big(2\pi \rho(p_n/q_n,E, \tilde{\theta}) \big)\big|<2- \frac{1}{q_n^{2}}.
\end{equation*}
 Then by \eqref{20200607} in Proposition \ref{differenceestimate} we have
%
%
%
%
%
\begin{equation*}
\begin{aligned}
\left|a_{q_n,0}(E)\right|\leq   2- \frac{1}{q_n^{2}}+   4 \exp\{-\Lambda(cq_{n})\} \leq 2- \frac{1}{2q_{n}^{2}}.
\end{aligned}
\end{equation*}
By \eqref{20200607} again, for any $\theta\in\T$, we have $
\left|t_{p_n/q_n}(E,\theta)\right|\leq 2,$
which means $E\in S_-(p_n/q_n)$.\qed

%
%

\subsection{Continuity of the Lyapunov exponent}\ \
Theorem \ref{addproposition} proves  $\Sigma_{ac}(\alpha)\subset\lim_{n\rightarrow\infty}S_{-}(p_{n}/q_{n})$  for any  M-ultra-differentiable potentials satisfying  $\mathbf{(H1)}$ and $\mathbf{(H2)},$  however, when we come to the inverse inclusion, we can only  prove the result for  $\nu$-Gevrey potentials with $1/2<\nu<1.$
It is interesting to extend the conclusion below to the cocycle with ultra-differential
potentials, even with $C^{\infty}$ potentials.
\begin{theorem}\label{addpropositions}
Let $V:\mathbb{T} \rightarrow  \R$ be a $\nu$-Gevrey function with $1/2< \nu<1$, and assume
that $\alpha\in \R\backslash\Q$.
Then there is a sequence $p_{n}/q_{n}\rightarrow \alpha$, such that
\begin{equation*}
\begin{aligned}
\limsup_{n\rightarrow\infty}S_{-}(p_{n}/q_{n})\subset
S_{-}(\alpha)=\Sigma_{ac}(\alpha).
\end{aligned}
\end{equation*}
\end{theorem}

\begin{remark}
The sequence $p_n/q_n$ will be the full sequence of continued fraction approximations
in the case $\alpha$ is Diophantine,
and an appropriate subsequence of it otherwise. For practical purposes of making conclusions
about $S_-(\alpha)$ based on the information on $S_-(p_n/q_n)$, it is sufficient to have convergence along a subsequence. However, in the latter case, the potential can be any stationary bounded ergodic one.
\end{remark}

The proof of Theorem \ref{addpropositions} depends on the continuity of the Lyapunov exponent for
the more general Gevrey cocycles.
For a Gevrey  (possibly matrix
valued) function $f$, we let \begin{equation*}
\begin{aligned}
\|f\|_{\nu,r}=\sum_{k\in\mathbb{Z}}|\widehat{f}(k)|\me^{|2\pi k|^{\nu}r}, \ 0<\nu<1.
\end{aligned}
\end{equation*} We denote by $G_{r}^{\nu}(\T,*)$ the set of all these
$*$-valued functions ($*$ will usually denote $\R$, $SL(2,\R)$.). If we set $r=\widetilde{r}^{\nu},$ we get
$
\|f\|_{\nu,r}=\|f\|_{\Lambda_{\nu},\widetilde{r}},$
where $\|\cdot\|_{\Lambda_{\nu},r}-$norm is the one defined by \eqref{20201022} with $\Lambda_{\nu}(x)=x^{\nu}.$
To simplify the notation, we introduce $\|\cdot\|_{\nu,r}.$ Note the function $\Lambda_{\nu}(x)=x^{\nu},0<\nu<1,$ satisfies
the subadditivity, thus $G_{r}^{\nu}(\T,*),0<\nu<1,$ is a Banach algebra.

%
%

\begin{theorem}\label{continuity}  Let $\rho>0, 2^{-1}<\nu<1.$ Consider the cocycle $(\alpha,A)\in (0,1)\setminus\Q\times G_{\rho}^{\nu}(\T,SL(2,\R))$ with $\alpha\in DC.$
Then $L(\alpha,A)$ is jointly continuous in the sense that
 \begin{equation*}
\lim\limits_{n\rightarrow\infty}L(p_n/q_n,A_n) = L(\alpha,A),
\end{equation*}
where $p_n/q_n$ is the continued fraction expansion of $\alpha$, and $A_{n}\in G_{\rho}^{\nu}(\T,SL(2,\R))$ with $A_{n}\rightarrow A$ under the
topology derived by $\|\cdot\|_{\nu,\rho}-$norm.
\end{theorem}

The full joint continuity was first proved by Bourgain-Jitomirskaya \cite{Bourgainj02} for analytic cocycles, which also plays a fundamental role in establishing the global theory of Sch\"odinger operator   \cite{Avila15}. However, due to lack of analyticity, it's very difficult to generalize the above result to all irrational $\alpha$.
The main reason is that in our large deviation theorem estimates, there is an upper bound for $N$ (see the assumptions in Proposition~\ref{8}),
thus we cannot deal with the extremely Liouvillean frequency. It's an interesting open question whether one can prove the continuity of the Lyapunov exponent for Liouvillean frequency and non-analytic potentials.

\bigskip
{\bf Proof of Theorem~\ref{addpropositions}.}
We will first prove Theorem~\ref{addpropositions} and left the proof of  Theorem \ref{continuity} to Section \ref{gerc}.
We separate the proof into two cases.

$\mathbf{Case~I:}$ $\alpha\in DC(v,10)$. We can assume that $L(\alpha,E)>0$,  then by Theorem \ref{continuity},
 for any $p_n/q_n$ sufficiently close to
$\alpha$,  we have $L(p_n/q_n,E)>0$. This implies that  $E\notin \sigma(p_n/q_n,
\theta)$ for some $\theta\in\T$, hence $E\notin S_-(p_n/q_n)$.

$\mathbf{Case~II:}$  $\alpha\notin DC(v,10)$, we define a sequence $\{V^\theta_m(n)\}_{m=1}^\infty$ periodic potentials by:
\begin{equation*}
\begin{aligned}
V_m^\theta(n)=V(\theta+n\alpha),\ \ n=1,2,\cdots,m,
\end{aligned}
\end{equation*}
\begin{equation*}
\begin{aligned}
V_m^\theta(n+m)=V_m^\theta(n),
\end{aligned}
\end{equation*}
such that $V_m^\theta$ is obtained from $V_\omega$ by ``cutting" a finite piece of length $m$, and then repeating it. We denote by $\sigma_m(\theta)$ the spectrum of the periodic Schr\"odinger operators
\begin{equation*}
\begin{aligned}
Hu=u_{n+1}+u_{n-1}+V_m^\theta(n)u_n.
\end{aligned}
\end{equation*}
By Theorem 1 in \cite{LastY93}, for a.e. $\theta\in \T$,
\begin{equation}\label{limsup}
\begin{aligned}
\limsup_{m\rightarrow\infty}\sigma_m(\theta)\subset \Sigma_{ac}(\alpha).
\end{aligned}
\end{equation}
We define
$$
A^\theta_{q_n}(\xi)=\begin{pmatrix}
V_{q_n}^\theta(1)&1&&& e^{-i\xi m}\\
1&V_{q_n}^\theta(2)&1&&\\
& 1&\ddots&\ddots&\\
&& \ddots&\ddots&1\\
e^{i\xi m}&&&1&V_{q_n}^\theta(q_n)
\end{pmatrix},
$$
$$
\widetilde{A}^\theta_{q_n}(\xi)=\begin{pmatrix}
V(\theta+p_n/q_n)&1&&& e^{-i\xi m}\\
1&V(\theta+2p_n/q_n)&1&&\\
& 1&\ddots&\ddots&\\
&& \ddots&\ddots&1\\
e^{i\xi m}&&&1&V(\theta+q_np_n/q_n)
\end{pmatrix}.
$$
It is standard that
$$
\sigma_{q_n}(\theta)=\bigcup_{\xi}\text{Spec}(A_{q_n}^\theta(\xi)),\  \ \sigma(p_n/q_n,\theta)=\bigcup_{\xi}\text{Spec}(\widetilde{A}_{q_n}^\theta(\xi)),
$$
where $\text{Spec}(A)$ denotes the sets of all eigenvalues of $A$. We need the following perturbation theory of matrices.
\begin{proposition}[Corollary 12.2 of \cite{Rajendra87}]\label{hauss}
Let $A$ and $B$ be normal with $\|A-B\|=\varepsilon$. Then within a distance of $\varepsilon$ of every eigenvalue of $A$ there is at least one eigenvalue of $B$ and vice versa.
\end{proposition}
Fix $\theta_{0}$ such that \eqref{limsup} holds with $\theta_{0}$ in place of $\theta.$ Notice for any $E_0\in \sigma_{q_n}(\theta_{0})$, there exists $\xi_{n}$ such that $E_0\in \text{Spec}(A_{q_n}^{\theta_{0}}(\xi_{n}))$. Applying Proposition \ref{hauss} to $A_{q_n}^{\theta_{0}}(\xi_{n})$ and $\widetilde{A}_{q_n}^{\theta_{0}}(\xi_{n})$, there exists $E'_0\in \text{Spec}(\widetilde{A}_{q_n}^{\theta_{0}}(\xi_{n}))$ such that
$$
|E_0-E_0'|\leq \big\|A_{q_n}^{\theta_{0}}(\xi_{n})-\widetilde{A}_{q_n}^{\theta_0}(\xi_n)\big\|\leq C(V)\sum\limits_{j=1}^{q_n}\big|j(\alpha-p_n/q_n)\big|\leq C(V)q_n^2\big|\alpha-p_n/q_n\big|.
$$
Since $E_0'\in \sigma(p_{n}/q_{n},\theta_0)$, it follows that
\begin{equation}\label{haus}
\big|\sigma_{q_n}(\theta_0)-\sigma\big(p_n/q_n,\theta_0\big)\big|_H\leq C(V)q_n^2\big|\alpha-p_n/q_n\big|,
\end{equation}
where $|A-B|_H$ denotes the Hausdorff distance of two sets. We denote $\sigma_{q_n}(\theta_0)=\bigcup_{i=1}^{q_n'}[a_{n,i},b_{n,i}]$, $q_n'\leq q_n$. \eqref{haus} implies that
$$
\sigma\big(p_n/q_n,\theta_0\big)\subset\bigcup_{i=1}^{q_n'}\Big[a_{n,i}-C(V)q_n^2\big|\alpha-p_n/q_n\big|,
b_{n,i}+C(V)q_n^2\big|\alpha-p_n/q_n\big|\Big].
$$
It follows that
\begin{equation}\label{ee}
\begin{aligned}
\big|\sigma(p_{n}/q_{n},\theta_0)\backslash\sigma_{q_n}(\theta_0)\big|\leq C(V)q_n^3|\alpha-p_{n}/q_{n}|.
\end{aligned}
\end{equation}
Since $\alpha\notin DC(v,10)$, by \eqref{ee}, there exists a subsequence $p_{n}/q_{n}$ such that
\begin{equation*}
\begin{aligned}
\limsup_{n\rightarrow\infty} \sigma(p_{n}/q_{n},\theta_{0})\subset \limsup_{n\rightarrow\infty}\sigma_{q_{n}}(\theta_0)\subset \Sigma_{ac}(\alpha).
\end{aligned}
\end{equation*}
Moreover, notice that
\begin{equation*}
\begin{aligned}
\limsup_{n\rightarrow\infty} S_{-}(p_{n}/q_{n})\subset\limsup_{n\rightarrow\infty} \sigma(p_{n}/q_{n},\theta_0),
\end{aligned}
\end{equation*}
hence
\begin{equation*}
\begin{aligned}
\limsup_{n\rightarrow\infty} S_{-}(p_{n}/q_{n})\subset \Sigma_{ac}(\alpha).
\end{aligned}
\end{equation*}

\section{Proof\ of \ Proposition~\ref{differenceestimate}}\label{chambers}
Suppose that $ V:\mathbb{T} \rightarrow  \R$ is an M-ultra-differentiable function satisfying  $\mathbf{(H1)}$ and $\mathbf{(H2)}$, then for almost every $E\in \Sigma_{ac}(\alpha)$,
by Theorem \ref{mainresult},  $(\alpha, S_{E}^{V})$ is  $C^{\infty}$ rotations reducible. However, this is not enough for us to conclude
\begin{equation*}
\begin{aligned}
\|t_{p_{n}/q_{n}}(E,\theta)-a_{q_{n},0}(E)\|_{C^{0}}\leq4
\exp\{-\Lambda(cq_{n})\}.
\end{aligned}
\end{equation*}
Since even we assume  $(\alpha, S_{E}^{V})=(\alpha,R_{\psi(\theta)})$ with $\psi(\theta)\in C^{\infty}(\mathbb{T},\mathbb{R}),$ which only gives
\begin{equation*}
\begin{aligned}
\|t_{p_{n}/q_{n}}(E,\theta)-a_{q_{n},0}(E)\|_{C^{0}}\leq c q_n^{-\infty}.
\end{aligned}
\end{equation*}

The idea is that our KAM scheme not only gives  $C^{\infty}$ rotations reducibility, but also almost reducibility in the ultra-differentiable topology (Corollary \ref{corollarylocalalmost}),
however, this only works for cocycles which are close to constant.  Coupled with the renormalization argument, we will show that if the cocycle is $L^{2}$-conjugated to rotations and  $\rho(\alpha, A)\in\mathcal{P},$ then $(\alpha, A)$ is  also almost reducibility in the ultra-differentiable topology (Lemma \ref{almostreducibility}). Consequently, Proposition~\ref{differenceestimate} follows from the perturbation arguments.


\subsection{Global almost reducibility}

To get desired quantitative estimates, the main method is the inverse renormalization which was first developed in \cite{Krikorian09}.
We introduce the notation $\Psi:=\mathcal{J}\mathcal{R}_{\theta_{*}}^{n}(\Psi')$
if $\Psi'=\mathcal{R}_{\theta_{*}}^{n}(\Psi).$ It is easy to check that
\begin{equation*}
\mathcal{J}\mathcal{R}_{\theta_{*}}^{n}(\Phi)
=T_{\theta_{*}}^{-1}\circ N_{\widetilde{Q}_{n}^{-1}}\circ M_{\beta_{n-1}^{-1}}\circ
T_{\theta_{*}}(\Phi).
\end{equation*}
In our setting, one of the bases is the following:

\begin{lemma}\label{renormalizationinverse}
Let $\Phi_{1}=((1,Id), (\alpha_{n}, R_{\rho_{n}}\me^{F(\theta)})),$
$\Phi_{2}=((1,Id), (\alpha_{n},R_{\rho_{n}}))$ where $F\in U_{r}(\T,sl(2,\R))$ with estimate
$\|F\|_{r,1}\leq q_{n-1}^{-2}.$ Then,
\begin{equation}\label{abdifference}
\begin{aligned}
\|\mathcal{J}\mathcal{R}_{\theta_{*}}^{n}(\Phi_{1})
-\mathcal{J}\mathcal{R}_{\theta_{*}}^{n}(\Phi_{2})\|_{\beta_{n-1}r,1}
< 2q_{n-1}\|F(\theta)\|_{r,1}.
\end{aligned}
\end{equation}
\end{lemma}

\begin{proof}
Direct computation shows that
\begin{equation*}
\begin{split}
M_{\beta_{n-1}^{-1}}(\Phi_{1})&=
((\beta_{n-1},Id),(\beta_{n-1}\alpha_{n},R_{\rho_{n}}\me^{F(\beta_{n-1}^{-1}\theta)})),
\\
M_{\beta_{n-1}^{-1}}(\Phi_{2})&=
((\beta_{n-1},Id),
(\beta_{n-1}\alpha_{n},R_{\rho_{n}})),
\end{split}
\end{equation*}
and consequently
\begin{equation*}
\begin{split}
N_{\widetilde{Q}_{n}^{-1}}\circ M_{\beta_{n-1}^{-1}}(\Phi_{1})&=((1,R_{\rho_{n}q_{n-1}}\me^{F_{1}(\theta)}),
(\alpha,R_{\rho_{n}p_{n-1}}\me^{F_{2}(\theta)})),
\\
N_{\widetilde{Q}_{n}^{-1}}\circ M_{\beta_{n-1}^{-1}}(\Phi_{2})&=
((1,R_{\rho_{n}q_{n-1}}),
(\alpha,R_{\rho_{n}p_{n-1}})).
\end{split}
\end{equation*}
Note for any $\lambda \neq0$,  $ \|D_{\theta}^{s} M_{\lambda}(\Phi)\|_{C^{0}([0,T])}= \lambda^s  \|D_{\theta}^{s}\Phi\|_{C^{0}([0,T])}$, which gives
\begin{equation}\label{ml}
\| M_{\lambda}(\Phi)\|_{\lambda^{-1}r,T}=\| \Phi\|_{r,\lambda T} .
\end{equation}
Thus the main task is to estimate of the norm under iteration of the cocycles.  To do this,  we need the following simple observation:

\begin{lemma}\emph{(Lemma~4.3 of \cite{ZhouW12})\ \label{zhouandwangjdde}}
Let $F_{i}\in U_{r}(\R,sl(2,\R)),
i=1,\cdots, j.$ Then it holds that
\begin{equation*}
R_{\rho}\me^{F_{j}(\theta)}
R_{\rho}\me^{F_{j-1}(\theta)}\cdots
R_{\rho}\me^{F_{1}(\theta)}=R_{j\rho}\me^{\widetilde{F}(\theta)},
\end{equation*}
with  estimate
\begin{equation*}
\|\widetilde{F}\|_{r,T}\leq \sum_{i=1}^{j}\|F_{i}\|_{r,T}.
\end{equation*}
\end{lemma}

By \eqref{ml} and Lemma~\ref{zhouandwangjdde}, we have
\begin{equation*}
\|F_{1}\|_{\beta_{n-1}r,1}\leq q_{n-1}\|F\|_{r,1}, \qquad \|F_{2}\|_{\beta_{n-1}r,1}\leq p_{n-1} \|F\|_{r,1}.
\end{equation*}
Then \eqref{abdifference} follows directly.

%
%
\end{proof}

\begin{lemma}\label{almostreducibility}
 Assume that  $(\alpha, A)$ is $L^{2}$-conjugated to rotations and homotopic to the identity.  If $\rho_{f}=\rho(\alpha, A)\in\mathcal{P}.$ Then there exist $B_{j,\ell},F_{j,\ell}\in U_{\widetilde{r}_{\ell}}(\mathbb{T}, SL(2,\mathbb{R}))$ such that
\begin{equation}\label{020200600}
\begin{aligned}
B_{j,\ell}(\theta+\alpha)A(\theta)B_{j,\ell}(\theta)^{-1}
=R_{\rho_{f}}\me^{F_{j,\ell}(\theta)}, \quad  \ell > \widehat{n} ,
\end{aligned}
\end{equation}
where $\widetilde{r}_{\ell}= r\beta_{n_j-1}/2K_{*}^{3}\overline{Q}_{\ell-1}^2,$ and $\widehat{n}\in\N$ is the smallest one such that $\overline{Q}_{\widehat{n}}\geq C^{q_{n_{j}}^{2}}.$ Moreover,  we have estimate
$$ \|F_{j,\ell}\|_{\widetilde{r}_{\ell}}\leq
\varepsilon_{\ell}^{\frac{1}{2}}, \qquad \|B_{j,\ell}\|_{\widetilde{r}_{\ell}}<4C^{3q_{n_{j}-1}q_{n_{j}}}. $$
\end{lemma}
\begin{proof}
Since $(\alpha, A)$ is $L^{2}$-conjugated to rotations and homotopic to the identity, by Proposition \ref{llambdarenormalizations}
 there exists  $D_{n}\in U_{rK_*^{-2}}(\mathbb{R}, SL(2,\mathbb{R}))$ with
\begin{equation}\label{20210102}
 \|D_n\|_{r(K_{*}^{2}T)^{-1},T} \leq C^{q_{n-1}(T+1)},
 \end{equation}
  such that
\begin{equation*}
\begin{aligned}
\mathrm{Conj_{D_{n}}}(\mathcal{R}_{\theta_{*}}^{n}(\Phi))= ((1,Id),(\alpha_{n},\ R_{\rho_{n}}\me^{F_{n}})),
\end{aligned}
\end{equation*}
with $\|F_{n}\|_{rK_{*}^{-2},1}\rightarrow 0.$ Since  $\rho_{f}=\rho(\alpha, A)\in\mathcal{P},$ which means  $\rho_f\beta_{n_{j}-1}^{-1}\in DC_{\alpha_{n_{j}}}(\gamma,\tau)$ for infinitely many $n_j$,
 then we can further choose $j$ large enough, such that
\begin{equation*}
\|F_{n_{j}}\|_{rK_{*}^{-2}}\leq \varepsilon_{*}(\gamma,\tau,rK_{*}^{-2},M),
\end{equation*}
where $\varepsilon_{*}=\varepsilon_{*}(\gamma,\tau,r,M)>0$ is the one in Theorem~\ref{mainresults}.
In the following, we will write $D_{j}$ for $D_{n_{j}}$ for short, and denote $T_n=\beta_{n-1}^{-1}=q_{n}+\alpha_{n}q_{n-1}$, $r_\ell= 2^{-1} rK_{*}^{-2}\overline{Q}_{\ell-1}^{-2}$, then $\widetilde{r}_{\ell}=(K_* T_{n_{j}})^{-1} r_{\ell}$.

Now we apply Corollary~\ref{corollarylocalalmost}
to the action $\mathrm{Conj_{D_{j}}}(\mathcal{R}_{\theta_{*}}^{n_{j}}(\Phi))$
and denote $Z_{j,\ell}=B_{\ell}D_{j},$ then
\begin{equation}\label{jmd0903}
\begin{aligned}
\mathrm{Conj_{Z_{j,\ell}}}(\mathcal{R}_{\theta_{*}}^{n_{j}}(\Phi))
=\widetilde{\Phi}_{1}^{(j,\ell)}:=
((1,Id),(\alpha_{n_{j}},\ R_{\rho_{n_{j}}}\me^{F_{j,\ell}})),
\end{aligned}
\end{equation}
which, together with \eqref{20201225s} and  \eqref{20210102} and
the fact $r_{\ell}\ll rK_{*}^{-2} T_{n_j}^{-1}$, implies
\begin{equation}\label{flzjlestimate}
\begin{aligned}
\|F_{j,\ell}\|_{r_{\ell},1}\leq \varepsilon_{\ell},\ \|Z_{j,\ell}\|_{r_{\ell},T_{n_{j}}}\leq
\|D_{j}\|_{r_{\ell},T_{n_{j}}}\|B_{\ell}\|_{r_{\ell},1}\leq C^{3q_{n_{j}-1}q_{n_{j}}}.
\end{aligned}
\end{equation}

Once we have these, we can set $\widetilde{\Phi}_{2}^{(j,\ell)}=((1,Id),(\alpha_{n_{j}},\ R_{\rho_{n_{j}}}))$
and define $\Phi_{2,j,\ell}$ by
\begin{equation}\label{jmd0904}
\begin{aligned}
\mathrm{Conj_{Z_{j,\ell}}}(\mathcal{R}_{\theta_{*}}^{n_{j}}(\Phi_{2,j,\ell}))
=\widetilde{\Phi}_{2}^{(j,\ell)}.
\end{aligned}
\end{equation}
For given $G\in SL(2,\mathbb{R}),$ $T_{\theta_{*}}$ and $N_{U}$ commute with $\mathrm{Conj_{G}}$ while
$M_{\lambda}\circ\mathrm{Conj_{G}}=\mathrm{Conj_{G(\lambda\cdot)}}\circ M_{\lambda},$ then by \eqref{jmd0903} and \eqref{jmd0904}, we get
\begin{equation}\label{20200711}
\begin{aligned}
\left\{\begin{array}{l l}
\Phi=\mathcal{J}\mathcal{R}_{\theta_{*}}^{n_{j}}
\big(\mathrm{Conj_{Z_{j,\ell}^{-1}}}(\widetilde{\Phi}_{1}^{(j,\ell)})\big)
=\mathrm{Conj_{Z_{j,\ell}^{-1}(\beta_{n_{j}-1}^{-1}\cdot)}}\big(\widetilde{\Phi}_{1,j,\ell}
\big),\\
\Phi_{2,j,\ell}=\mathcal{J}\mathcal{R}_{\theta_{*}}^{n_{j}}
\big(\mathrm{Conj_{Z_{j,\ell}^{-1}}}(\widetilde{\Phi}_{2}^{(j,\ell)})\big)
=\mathrm{Conj_{Z_{(j,\ell)}^{-1}(\beta_{n_{j}-1}^{-1}\cdot)}}\big(\widetilde{\Phi}_{2,j,\ell}
\big),
\end{array}\right.
\end{aligned}
\end{equation}
where
\begin{equation}\label{20200711s}
\begin{aligned}
\left\{\begin{array}{l l}
\widetilde{\Phi}_{1,j,\ell}&=\mathcal{J}\mathcal{R}_{\theta_{*}}^{n_{j}}
(\widetilde{\Phi}_{1}^{(j,\ell)}),\\
\widetilde{\Phi}_{2,j,\ell}&=\mathcal{J}\mathcal{R}_{\theta_{*}}^{n_{j}}
(\widetilde{\Phi}_{2}^{(j,\ell)})
=((1,R_{\rho_{n_{j}}q_{n_{j}-1}}),(\alpha,R_{\rho_{n_{j}}p_{n_{j}-1}})).
\end{array}\right.
\end{aligned}
\end{equation}

Set $\widetilde{Z}_{j}(\theta)=R_{-\rho_{n_{j}}q_{n_{j}-1}\theta},$ then
\begin{equation*}
\begin{aligned}
\mathrm{Conj_{\widetilde{Z}_{j}}}\big(\widetilde{\Phi}_{2,j,\ell}\big)=((1,Id),
(\alpha,R_{\rho_{f}})):=\Phi_{**},
\end{aligned}
\end{equation*}
which, together with \eqref{20200711} and \eqref{20200711s}, implies
\begin{equation}\label{20201226}
\begin{aligned}
\left\{\begin{array}{l l}
\Phi
=\mathrm{Conj_{Z_{j,\ell}^{-1}(\beta_{n_{j}-1}^{-1}\cdot)\widetilde{Z}_{j}^{-1}}}
(\Phi_{*}),\\
\Phi_{2,j,\ell}
=\mathrm{Conj_{Z_{j,\ell}^{-1}(\beta_{n_{j}-1}^{-1}\cdot)\widetilde{Z}_{j}^{-1}}}
(\Phi_{**}),
\end{array}\right.
\end{aligned}
\end{equation}
where $\Phi_{*}=\mathrm{Conj_{\widetilde{Z}_{j}}}\big(\widetilde{\Phi}_{1,j,\ell}\big).$
Moreover, by our selection $\overline{Q}_\ell >\overline{Q}_{\widehat{n}}\geq C^{q_{n_{j}}^{2}}.$
\begin{equation}\label{20201228}
\|\widetilde{Z}_{j}\|_{r_{\ell},1}\leq 2,
\end{equation}
\begin{equation*}
\|F_{j,\ell}\|_{r_{\ell},1}\leq\varepsilon_{\ell}\ll q_{n_{j}}^{-6}.
\end{equation*} In the following, we will give the estimate of the distance of $\Phi$ and
$\Phi_{2,j,\ell}.$

First, we apply Lemma~\ref{renormalizationinverse} with
$\widetilde{\Phi}_{i}^{(j,\ell)}$ in place of  $\Phi_{i},$ and  $n_{j}$ in place of $n$ respectively, then by
\eqref{abdifference}
\begin{equation*}
\begin{aligned}
\|\widetilde{\Phi}_{1,j,\ell}
-\widetilde{\Phi}_{2,j,\ell}\|_{T_{n_{j}}^{-1}r_{\ell},1}
<\|F_{j,\ell}\|_{r_{\ell},1}^{\frac{3}{4}}\leq \varepsilon_{\ell}^{\frac{3}{4}}.
\end{aligned}
\end{equation*}
Finally, by \eqref{flzjlestimate} and \eqref{20200711} and the inequality above we get
\begin{eqnarray}\label{20201012}
\|\Phi-\Phi_{2,j,\ell}\|_{T_{n_{j}}^{-1}r_{\ell},1}
&\leq&\|Z_{j,l}\|_{r_{\ell},T_{n_{j}}}^{2}
\|\widetilde{\Phi}_{1,j,\ell}-\widetilde{\Phi}_{2,j,\ell}
\|_{T_{n_{j}}^{-1}r_{\ell},1}\nonumber\\
&\leq& C^{6q_{n_{j}-1}q_{n_{j}}}\varepsilon_{\ell}^{\frac{3}{4}}.
\end{eqnarray}

Notice that $\Phi_{2,j,\ell}$ may not be normalized, however, by \eqref{20201012} we know that
\begin{eqnarray}\label{20200713}
\|\Phi_{2,j,\ell}(1,0)-Id\|_{T_{n_{j}}^{-1}r_{\ell},1}
&=&\|\Phi_{2,j,\ell}(1,0)-\Phi(1,0)\|_{T_{n_{j}}^{-1}r_{\ell},1}\nonumber\\
&\leq& C^{6q_{n_{j}-1}q_{n_{j}}}\varepsilon_{\ell}^{\frac{3}{4}}.
\end{eqnarray}
Thus by Lemma~\ref{normalizinglemma}, there exists a conjugation $\widetilde{B}_{j,\ell}\in U_{\widetilde{r}_{\ell}}(\R,SL(2,\R))$
 such that
$\overline{\Phi}_{j,\ell}=\mathrm{Conj_{\widetilde{B}_{j,\ell}}}(\Phi_{2,j,\ell})$ is  a normalized action.  Moreover, we have estimate
\begin{eqnarray}\label{bjlestimate}
\|\widetilde{B}_{j,\ell}-Id\|_{\widetilde{r}_{\ell},1}
\leq\|\Phi_{2,j,\ell}(1,0)-Id\|_{T_{n_{j}}^{-1}r_{\ell},1}
\leq C^{6q_{n_{j}}q_{n_{j}-1}}\varepsilon_{\ell}^{\frac{3}{4}}.
\end{eqnarray}

Since $\overline{\Phi}_{j,\ell}(0,1)=\widetilde{B}_{j,\ell}
(\theta+\alpha)\Phi_{2,j,\ell}(0,1)\widetilde{B}_{j,\ell}(\theta)^{-1},$
by \eqref{20200713}, \eqref{bjlestimate} we have
\begin{equation*}
\begin{aligned}
\|\overline{\Phi}_{j,\ell}-\Phi_{2,j,\ell}\|_{\widetilde{r}_{\ell},1}
\leq 2C^{6q_{n_{j}}q_{n_{j}-1}}\varepsilon_{\ell}^{\frac{3}{4}}.
\end{aligned}
\end{equation*}
The inequality above, together with \eqref{20201012}, yields
\begin{equation}\label{jmd0908}
\|\Phi-\overline{\Phi}_{j,\ell}\|_{\widetilde{r}_{\ell},1}
\leq 3C^{6q_{n_{j}}q_{n_{j}-1}}\varepsilon_{\ell}^{\frac{3}{4}}.
\end{equation}

Set $B_{j,\ell}(\cdot)=\widetilde{Z}_{j}(\cdot)Z_{j,\ell}(\beta_{n_{j}-1}^{-1}\cdot)
\widetilde{B}_{j,\ell}^{-1}(\cdot),$ then by \eqref{20201226} we get
\begin{equation}\label{jmd0913}
\begin{aligned}
\mathrm{Conj_{B_{j,\ell}}}(\overline{\Phi}_{j,\ell})
=\mathrm{Conj_{\widetilde{Z}_{j}Z_{j,\ell}(\beta_{n_{j}-1}^{-1}\cdot)}}(\Phi_{2,j,\ell})
=\Phi_{**}.
\end{aligned}
\end{equation}
Thus $B_{j,\ell}$ is $1-$periodic since both $\overline{\Phi}_{j,\ell}$ and $\Phi_{**}$ are normalized.
Moreover, \eqref{flzjlestimate}, \eqref{20201228} and \eqref{bjlestimate} imply
\begin{equation}\label{jmd0914}
\begin{aligned}
\|B_{j,\ell}\|_{\widetilde{r}_{\ell},1}\leq \|
\widetilde{Z}_{j}\|_{\widetilde{r}_{\ell},1}
\|Z_{j,\ell}\|_{\widetilde{r}_{\ell},T_{n_{j}}}
\|\widetilde{B}_{j,\ell}\|_{\widetilde{r}_{\ell},1}\leq  4C^{3q_{n_{j}-1}q_{n_{j}}}.
\end{aligned}
\end{equation}
By \eqref{jmd0908}-\eqref{jmd0914} we get,
\begin{equation*}
\begin{aligned}
\|(0,&B_{j,\ell}(\cdot+\alpha))\circ (\alpha,A)\circ(0,B_{j,\ell})^{-1}-(\alpha, R_{\rho_{f}})\|_{\widetilde{r}_{\ell},1}\\
&\leq\|\mathrm{Conj_{B_{j,\ell}}}
(\Phi)-\mathrm{Conj_{B_{j,\ell}}}
(\overline{\Phi}_{j,\ell})\|_{\widetilde{r}_{\ell},1}
\leq C^{13q_{n_{j}-1}q_{n_{j}}}\varepsilon_{\ell}^{\frac{3}{4}}
\leq\varepsilon_{\ell}^{\frac{1}{2}},
\end{aligned}
\end{equation*}
where the last inequality follows from our selection  $\overline{Q}_\ell >\overline{Q}_{\widehat{n}}\geq C^{q_{n_{j}}^{2}}$ and definition of $\varepsilon_{\ell}$.
Thus, by implicit function theorem, there exists a unique $F_{j,\ell}\in U_{\widetilde{r}_{\ell}}(\T, sl(2,\R)),$
such that
\begin{equation*}
\begin{aligned}
B_{j,\ell}(\theta+\alpha)A(\theta)B_{j,\ell}(\theta)^{-1}=R_{\rho_{f}}
\me^{F_{j,\ell}(\theta)}
\end{aligned}
\end{equation*}
with $\|F_{j,\ell}\|_{\widetilde{r}_{\ell}}\leq\varepsilon_{\ell}^{\frac{1}{2}}.$
\end{proof}

\subsection{Proof of  Proposition~\ref{differenceestimate}}

First we construct the desired sequence.
For the sequence $(q_{\ell})_{\ell\in\mathbb{N}}$ and subsequence $(Q_{\ell})_{\ell\in\mathbb{N}}$ constructed
in Lemma~\ref{bridgeestimate}, first set $n_{1} \geq n_0$ to be the smallest integer such that
\begin{equation}\label{20201002}
\max\{16 C_{M} C^{6q_{n_{j}}^{2}}(2+2\|V\|_{r}) ,  16r^{-1}q_{n_j} K_*^3 \}\leq\overline{Q}_{n_{1}},
\end{equation}
where $n_{0}$ is the one in section~\ref{Inthissection}.  Then we set $n_{*}$ be the smallest integer number such that
\begin{equation}\label{20200630}
\begin{aligned}
q_{n_{*}}> \overline{Q}_{n_{1}+1}^{2\mathbb{A}^{4}\tau^{2}}.
\end{aligned}
\end{equation}
Then, for any fixed $n$ with $n\geq n_{*},$ we set $\ell\in\mathbb{N}$ be the smallest integer number
such that $\overline{Q}_{\ell}^{2\mathbb{A}^{4}\tau^{2}}\geq q_{n}.$ That is
\begin{equation}\label{20200619}
\begin{aligned}
\overline{Q}_{\ell-1}^{2\mathbb{A}^{4}\tau^{2}}
<q_{n}\leq\overline{Q}_{\ell}^{2\mathbb{A}^{4}\tau^{2}}.
\end{aligned}
\end{equation}
Thus by \eqref{20200630} and \eqref{20200619} we get $\ell-1\geq n_{1} \geq \widehat{n}$, where $\widehat{n}$ is the one defined in Lemma \ref{almostreducibility}.

By our construction, for almost every $E\in \Sigma_{ac}(\alpha)$,   $(\alpha, S_{E}^{V})$ is $L^{2}$-conjugated to rotations, and $\rho_{f}=\rho(\alpha, S_{E}^{V})\in\mathcal{P}.$ Then by \eqref{020200600} in Lemma \ref{almostreducibility},   there exist $B_{j,\ell},F_{j,\ell}\in U_{\widetilde{r}_{\ell}}(\mathbb{T}, SL(2,\mathbb{R}))$ such that
\begin{equation}\label{0202006001}
\begin{aligned}
B_{j,\ell}(\theta+\alpha)S_{E}^{V}(\theta)B_{j,\ell}(\theta)^{-1}
=R_{\rho_{f}}\me^{F_{j,\ell}},\,
\end{aligned}
\end{equation}
with estimate
$$ \|F_{j,\ell}\|_{\widetilde{r}_{\ell}}\leq
\varepsilon_{\ell}^{\frac{1}{2}}, \qquad \|B_{j,\ell}\|_{\widetilde{r}_{\ell}}<4C^{3q_{n_{j}-1}q_{n_{j}}}. $$
We shorten $B_{j,\ell}$ and $F_{j,\ell}$  as $B$ and $F,$ respectively.
By \eqref{0202006001} we get
\begin{equation}\label{20200610}
\begin{aligned}
B(\theta+p_{n}/q_{n})S_{E}^{V}(\theta)B(\theta)^{-1}=R_{\rho_{f}}+f(\theta),
\end{aligned}
\end{equation}
where
\begin{eqnarray*}
f(\theta)&=&R_{\rho_{f}}(\me^{F(\theta)}-I)
+\{B(\theta+p_{n}/q_{n})-B(\theta+\alpha)\}S_{E}^{V}(\theta)B(\theta)^{-1}\\
&=& (\textrm{I})+ (\textrm{II}).
\end{eqnarray*}

Note for any $E\in \Sigma_{ac}(\alpha)\subset \Sigma(\alpha)$, we have  $|E|< 2+\|V\|_{r}$, thus by  Cauchy's estimate (Lemma \ref{cauchyestimate}), we get
\begin{eqnarray}\label{f2}
 \|\textrm{II}\|_{\widetilde{r}_{\ell}/2}&\leq&
 |\alpha-\frac{p_n}{q_n}| \|\partial B\|_{\widetilde{r}_{\ell}/2}
\|S_E^V\|_{\widetilde{r}_{\ell}}\|B^{-1}\|_{\widetilde{r}_{\ell}}\nonumber\\
&\leq&
C_{M}\widetilde{r}_{\ell}^{-1}q_{n}^{-2}\|B^{-1}\|_{\widetilde{r}_{\ell}}
\|B\|_{\widetilde{r}_{\ell}} (2+2\|V\|_{r})      \nonumber\\
&\leq&16 C_{M} C^{6q_{n_{j}-1}q_{n_{j}}}  (2+2\|V\|_{r}) \widetilde{r}_{\ell}^{-1}q_{n}^{-2}.
\end{eqnarray}

Since $B$ is 1-periodic, by \eqref{20200610}, we have
\begin{equation*}
\begin{aligned}
B(\theta&+q_{n}p_{n}/q_{n})\Pi_{s=q_{n}-1}^{0}
S_{E}^{V}(\theta+sp_{n}/q_{n})B(\theta)^{-1}\\
&=\Pi_{s=q_{n}-1}^{0}B(\theta+(s+1)p_{n}/q_{n})
S_{E}^{V}(\theta+sp_{n}/q_{n})B(\theta+sp_{n}/q_{n})^{-1}\\
&=\Pi_{s=q_{n}-1}^{0}\{R_{\rho_{f}}+f(\theta+sp_{n}/q_{n})\} .
\end{aligned}
\end{equation*}
As a consequence,
\begin{equation*}
\begin{aligned}
\mathrm{tr}\Pi_{s=q_{n}-1}^{0}S_{E}^{V}(\theta+sp_{n}/q_{n})=\mathrm{tr}\Pi_{s=q_{n}-1}^{0}\{R_{\rho_{f}}+f(\theta+sp_{n}/q_{n})\},
\end{aligned}
\end{equation*}
which, together with \eqref{20200603} and \eqref{20200606}, implies
\begin{equation}\label{20200719}
\begin{aligned}
t_{p_{n}/q_{n}}(E,\theta)=\mathrm{tr}
\Pi_{s=q_{n}-1}^{0}\{R_{\rho_{f}}+f(\theta+sp_{n}/q_{n})\}=\sum_{k\in\mathbb{Z}}a_{q_{n},k}(E)\me^{2\pi\mi k q_{n}\theta}.
\end{aligned}
\end{equation}
The first equality in \eqref{20200719} implies
\begin{equation}\label{20200617}
\begin{aligned}
\|t_{p_{n}/q_{n}}(E,\theta)\|_{\widetilde{r}_{\ell}/2}
\leq2\{1+\|f\|_{\widetilde{r}_{\ell}/2}\}^{q_{n}}.
\end{aligned}
\end{equation}
In the  following we will give the estimate of $\|f\|_{\widetilde{r}_{\ell}/2},$ indeed,  by \eqref{20201002} and $\ell-1\geq n_{1}$, we have
\begin{equation}\label{rll} \widetilde{r}_{\ell}^{-1}=2r^{-1}\beta_{n_{j}-1}^{-1}K_{*}^{3}\overline{Q}_{\ell-1}^{2}
<4^{-1}\overline{Q}_{\ell-1}^{3}.
\end{equation}
Again, by \eqref{parameter}, \eqref{20201002}, \eqref{20200619} and \eqref{f2},  we have
\begin{eqnarray*}
\|f\|_{\widetilde{r}_{\ell}/2}\leq
\|\textrm{I}\|_{\widetilde{r}_{\ell}/2}+\|\textrm{II}\|_{\widetilde{r}_{\ell}/2}
\leq \frac{1}{2}\overline{Q}_{\ell}^{-4\mathbb{A}^{4}\tau^{2}}+ \frac{1}{2}\overline{Q}_{\ell-1}^{4}q_{n}^{-2}
\leq q_{n}^{-2(1-\mathbb{A}^{-4}\tau^{-2})}.
\end{eqnarray*}
Consequently, by \eqref{20200617}
\begin{equation}\label{20200712}
\begin{aligned}
\|t_{p_{n}/q_{n}}(E,\theta)\|_{\widetilde{r}_{\ell}/2}
\leq2\{1+q_{n}^{-2(1-\mathbb{A}^{-4}\tau^{-2})}\}^{q_{n}}<4.
\end{aligned}
\end{equation}

On the other hand, by  \eqref{rll}, we have
\begin{equation*}
\begin{aligned}
q_{n}2^{-1}\widetilde{r}_{\ell}>q_{n}\overline{Q}_{\ell-1}^{-3}> q_{n}\overline{Q}_{\ell-1}^{-4}
> q_{n}^{1-2\mathbb{A}^{-4}\tau^{-2}}>T_{1}.
\end{aligned}
\end{equation*}
Moreover, the second equality in \eqref{20200719} implies
\begin{equation*}
\begin{aligned}
t_{p_{n}/q_{n}}(E,\theta)-a_{q_{n},0}(E)=\mathcal{R}_{q_{n}}t_{p_{n}/q_{n}}(\theta,E).
\end{aligned}
\end{equation*}
Thus by Lemma \ref{projectionestimate} and \eqref{20200712}, we have
\begin{equation*}
\begin{aligned}
\|t_{p_{n}/q_{n}}(E,\theta)-a_{q_{n},0}(E)\|_{C^{0}}
&\leq
\|t_{p_{n}/q_{n}}(E,\theta)\|_{2^{-1}\widetilde{r}_{\ell}}\exp\{-\Lambda(\pi q_{n}2^{-1}\widetilde{r}_{\ell})\}\\
&\leq 4
\exp\{-\Lambda(q_{n}^{1-2\mathbb{A}^{-4}\tau^{-2}})\},
\end{aligned}
\end{equation*}
 the last inequality follows from  the fact that $\Lambda(\cdot)$
is non-decreasing on $\mathbb{R}^{+}.$ \qed

\section{Proof of Theorem~\ref{continuity}}\label{gerc}
In this section we give the proof of Theorem~\ref{continuity} which is based
on the large deviation theorem and avalanche principle. Notice that the
cocycle in Theorem~\ref{continuity} is $\nu$-Gevrey with $1/2<\nu<1$, we will follow the method in \cite{Silviusk05}
to approximate the Gevrey cocycle by its truncated cocycle which is analytic in the certain strip.  For the continuity argument, our scheme is in the spirit of \cite{Bourgainj02} with some modifications.
Compared to the result in \cite{Silviusk05} with $1/2<\nu<1$,
our large deviation theorem also works for more general cocycles (other than Schr\"odinger coycle) and rational frequencies, which is an analogue of Bourgain-Jitomirskaya \cite{Bourgainj02}.
More concretely, if we truncate the cocycle $A(\alpha,\theta)$ to $\widetilde{A}(\alpha,\theta)$, and denote  $\widetilde{A}_{N}(\alpha,\theta)$
to be the transfer matrix, then $\det \widetilde{A}_{N}(\alpha,\theta)$ depends on $\theta$, is not constant anymore (of course not identical to 1), thus we have to prove that the subharmonic
extension of  $N^{-1}\ln\|\widetilde{A}_{N}(\alpha,\theta)\|$ is bounded.
This boundedness will enable us to give an enhanced  version of the large deviation bound shown in \cite{Bourgainj02}. For more results and methods to prove the continuity of the Lyapunov exponents, we refer readers to \cite{AvilajS14,JitomirskayaKS09,JitomirskayaCAMAR12,Jitomirskayamarx11}.

\subsection{Large deviation theorem.}
In this subsection we give a large deviation theorem for the $\nu$-Gevrey cocycle with $1/2<\nu<1.$
Let $A_N(\alpha,\theta)$ and $L_N(\alpha,A)$ be the associated transfer matrix and finite Lyapunov exponent of the cocycle $(\alpha,A)$. Then we have the following:
\begin{proposition}\label{8}
Let $\rho>0$, $\frac{1}{2}<\nu<1, $ $0<\kappa<1.$ Assume that  $A \in G_{\rho}^{\nu}(\mathbb{T}, SL(2,\mathbb{R}))$, and
\begin{equation*}
\begin{aligned}
\big|\alpha-\frac{a}{q}\big|<\frac{1}{q^2},\ \ (a,q)=1.
\end{aligned}
\end{equation*}
Then there exist $c,C_{i}(\kappa)>0,i=1,2$, $\sigma_1>\sigma>1>\gamma>0$ and
$q_0(\kappa,\rho,\nu)\in\N^+$ such that for $q\geq q_0$, $C_1(\kappa)q^\sigma<N<C_2(\kappa)q^{\sigma_1}$,
\begin{equation*}
\begin{aligned}
mes\Big\{\theta:\big|\frac{1}{N}\ln\|A_N(\alpha,\theta)\|-L_N(\alpha,A)\big|>\kappa\Big\}<\me^{-c q^\gamma}.
\end{aligned}
\end{equation*}
\end{proposition}

%

\subsubsection{Averages of shifts of subharmonic functions.}\label{Averagesofshifts}
Let $u=u(\theta)$ be a function on $\T$ having a subharmonic extension on the strip $[|\mathrm{Im}\vartheta|\leq\rho]$, and $\alpha\in\T.$ We prove that the mean of $u$ is close to the
Fej\'{e}r average of $u(\theta)$ for $\theta$ outside a small set (here being `close' or `small' is expressed in terms of
the number of shifts considered).

Consider the Fej\'{e}r kernel of order p:
\begin{equation}\label{fer}
K_R^p(t)=\big(\frac{1}{R}\sum\limits_{j=0}^{R-1}\me^{2\pi \mi jt}\big)^p,
\end{equation}
then we have
\begin{equation*}
\begin{split}
\big|K_R^p(t)\big|=\frac{1}{R^{p}}\big|\frac{1-\me^{2\pi \mi Rt}}{1-\me^{2\pi \mi t}}\big|^p\leq\frac{1}{R^p\|t\|_{\Z}^p}.
\end{split}
\end{equation*}
Notice also $\big|K_R^p(t)\big|\leq 1$, we have
\begin{equation}\label{fejer}
\begin{split}
|K_R^p(t)|\leq\min\big\{1,\frac{1}{R^p\|t\|_\Z^p}\big\}\leq\frac{2}{1+R^p\|t\|_\Z^p}.
\end{split}
\end{equation}
We can rewrite \eqref{fer} as
\begin{equation*}
\begin{split}
K_R^p(t)=\frac{1}{R^p}\sum\limits_{j=0}^{p(R-1)}c_R^p(j) \me^{2\pi \mi jt},
\end{split}
\end{equation*}
where $c^p_R(j)$ are positive integers so that
\begin{equation*}
\begin{split}
\frac{1}{R^p}\sum\limits_{j=0}^{p(R-1)}c^p_R(j)=1.
\end{split}
\end{equation*}
Notice that if $p=1$ then $K_R^1(t)=\frac{1}{R}\sum\limits_{j=0}^{R-1}\me^{2\pi \mi jt}$,  thus $c^1_R(j)=1$ for all $j.$

\begin{proposition}\label{80}
Let  $\rho>0.$ Assume that $u: \T\rightarrow\R$
has a bounded subharmonic extension to the strip $[|\mathrm{Im}\vartheta|\leq\rho]$ and $\|u\|_{C^{0}}\leq S.$
If
\begin{equation*}
\begin{aligned}
\big|\alpha-\frac{a}{q}\big|<\frac{1}{q^2},\ \ (a,q)=1,
\end{aligned}
\end{equation*}
and $\sigma>1,0<\varsigma<\sigma^{-1}, \varsigma(1-\sigma^{-1})^{-1}<p<(\sigma-1)^{-1},$ then
\begin{equation*}
\begin{aligned}
mes\Big\{\theta:\Big|\frac{1}{R^p}\sum\limits_{j=0}^{p(R-1)}c_R^p(j)u(\theta+j\alpha)
-[u(\theta)]_{\theta}\Big|>\varsigma_{2}R^{-\varsigma_{1}}\Big\}
<\frac{R^{2\varsigma_{1}}}{2^{8}\exp\{R^{\varsigma_{3}}\}},
\end{aligned}
\end{equation*}
provided $R=q^\sigma\geq q(\varsigma,p,\sigma)\ (\varsigma_{1}= p(1-\sigma^{-1}), \varsigma_{2}=2^{p+5}S\rho^{-1}, \varsigma_{3}=\frac{1+p}{\sigma}-p).$
\end{proposition}

 \begin{proof}
 The proof is divided into the following 3 steps.\\

{\bf 1. Shift of Fej\'{e}r average.}
Since $u$ is subharmonic in  the strip $[|\mathrm{Im}\vartheta|\leq\rho]$, then from Corollary 4.7 in \cite{Bourgain05}, we get
\begin{equation}\label{subharmonic}
\begin{aligned}
|\widehat{u}(k)|\leq  \frac{S}{\rho|k|}.
\end{aligned}
\end{equation}

Consider the Fej\'{e}r average of $u_N(\theta)$, and notice that
\begin{equation*}
\begin{aligned}
u(\theta+j\alpha)=[u(\theta)]_{\theta}+\sum\limits_{k\neq 0}\widehat{u}(k)\me^{2\pi \mi k(\theta+j\alpha)},
\end{aligned}
\end{equation*}
thus, by shortening $K_R^{p}(\cdot)$ as $K_R(\cdot),$ we get
\begin{eqnarray}\label{star4}
&&\frac{1}{R^p}\sum\limits_{j=0}^{p(R-1)}c_R^p(j)u(\theta+j\alpha)-[u(\theta)]_{\theta}
\nonumber\\
&=&\sum\limits_{k\neq 0}\widehat{u}(k)\Big(\frac{1}{R^p}\sum\limits_{j=0}^{p(R-1)}c_R^p(j)\me^{2\pi \mi jk\alpha}\Big)\me^{2\pi \mi k\theta}\nonumber\\
&=&\sum\limits_{k\neq 0}\widehat{u}(k)\cdot K_R(k\alpha)\me^{2\pi \mi k\theta}:=w(\theta)=\mathcal{T}_Kw(\theta)+\mathcal{R}_Kw(\theta),
\end{eqnarray}
where $K>q$  is a large constant that will be determined later.
\bigskip

\noindent {\bf 2. Estimate of the $w(\theta)$.} In the following, we will give the estimate of $w(\theta)$. Let  $I_{\ell}=[\frac{q}{4}\ell,\frac{q}{4}(\ell+1)),$ then we write
$\mathcal{T}_Kw(\theta)$ as
\begin{equation*}
\begin{split}
\mathcal{T}_Kw(\theta)=\sum\limits_{\ell=0}^{[4Kq^{-1}]+1}\sum\limits_{k\in I_{\ell}}\widehat{u}(k)
    \cdot K_R(k\alpha)\me^{2\pi\mi k\theta}.
\end{split}
\end{equation*}

Note $\big|\alpha-\frac{a}{q}\big|<\frac{1}{q^2},$ it follows that for $|k|\leq \frac{q}{2}$ with $k\neq 0$, we have   $|k\alpha-\frac{ka}{q}|<\frac{1}{2q},$ hence $\|k\alpha\|_{\Z}>\frac{1}{2q}$. Let $\alpha_1,\cdots,\alpha_{q/4}$ be the decreasing rearrangement of $\{\|k\alpha\|_{\Z}^{-1}\}_{0<|k|\leq \frac{q}{4}}$. Then we have $\alpha_i\leq \frac{2q}{i}$. Moreover, for any interval of length $q/4$, same is true for $\{\|k\alpha\|_{\Z}^{-1}\}_{|k|\in I}$ if we exclude at most one value of $k$. By \eqref{fejer} and \eqref{subharmonic}, we have
\begin{equation}\label{star5}
\begin{split}
\sum\limits_{0<|k|<\frac{q}{4}}\big|\widehat{u}(k)K_R(k\alpha)\big|\leq
\sum\limits_{0<|k|<\frac{q}{4}}
\frac{S\|k\alpha\|_{\Z}^{-p}}{|k|\rho R^p}\leq\sum\limits_{1\leq i<\frac{q}{4}}\frac{2S(2q/i)^p}{\rho R^p}
\leq \frac{2^{p+3}S}{\rho}\left(\frac{q}{R}\right)^p,
\end{split}
\end{equation}
and for each  $\ell\geq1,$ we have
\begin{equation*}
\begin{aligned}
\sum\limits_{|k|\in I_{\ell}}\big|\widehat{u}(k)K_R(k\alpha)\big|\leq
\frac{2S}{\frac{q}{4}\rho\ell}\Big(1+\sum\limits_{1\leq i<\frac{q}{4}}\frac{(2q/i)^p}{R^p}\Big)\leq \frac{8S}{\rho q\ell}\big(1+c(q/R)^p\big).
\end{aligned}
\end{equation*}
Thus we have
\begin{align}\label{star6}
\nonumber \Big|\sum\limits_{\ell=1}^{[4Kq^{-1}]+1}\sum\limits_{|k|\in I_{\ell}}\widehat{u}(k)K_R(k\alpha)\me^{2\pi\mi k\theta}\Big|&\leq \sum\limits_{\ell=1}^{[4Kq^{-1}]+1}\frac{8S}{\rho q\ell}\big(1+c(q/R)^p\big)\\ \nonumber
&\leq \frac{8S}{\rho q}\big(1+c(q/R)^p\big)\ln [4Kq^{-1}+1]\\
&\leq \frac{8S}{\rho q}\big(1+c(q/R)^p\big)\ln K.
\end{align}

On the other hand, again by \eqref{subharmonic}, we have
\begin{equation}\label{202000002}
\begin{split}
\|\mathcal{R}_Kw(\theta)\|^2_{\ell^2}\leq \sum\limits_{|k|\geq K}\frac{S^2}{(\rho |k|)^2}\leq \frac{S^2}{\rho^2}K^{-1}.
\end{split}
\end{equation}

\medskip
\noindent {\bf 3.  Choose approximate $K,p$.} %

Now we can finish the proof of  Proposition~\ref{80}.
By \eqref{star4}-\eqref{star6} , we have
\begin{align*}
\big|[u(\theta)]_{\theta}&-\frac{1}{R^p}\sum\limits_{j=0}^{p(R-1)}c_R^p(j)u(\theta+j\alpha)\big|\\
&\leq\frac{2^{p+3}}{\rho}\left(\frac{q}{R}\right)^p+
\frac{8S}{\rho q}\Big(1+c\left(\frac{q}{R}\right)^p\Big)\ln K+|\mathcal{R}_Kw(\theta)|.
\end{align*}

Take $q=R^{\sigma^{-1}}, \sigma>1,$ $K=\exp\{R^{\sigma^{-1}-p(1-\sigma^{-1})}\}$
and  $0<\varsigma<\sigma^{-1}, \varsigma(1-\sigma^{-1})^{-1}<p<(\sigma-1)^{-1}.$
Once we fixed the parameters above, we have,
\begin{equation*}
\begin{aligned}
\Big|\frac{2^{p+3}S}{\rho}\left(\frac{q}{R}\right)^p\Big|=2^{p+3}S\rho^{-1}R^{-\varsigma_{1}},
\end{aligned}
\end{equation*}
\begin{equation*}
\begin{aligned}
\Big|\frac{8S}{\rho q}\Big(1+c\Big(\frac{q}{R}\Big)^p\Big)\ln K\Big|&\leq
16\rho^{-1}S q^{-1}\ln K=16S\rho^{-1}R^{-\varsigma_{1}},
\end{aligned}
\end{equation*}
where $ \varsigma_{1}=p(1-\sigma^{-1})$. By Chebyshev's inequality and \eqref{202000002}, one has
\begin{equation*}
\begin{aligned}
mes\Big\{\theta: |\mathcal{R}_Kw(\theta)|&>2^{4}S\rho^{-1}R^{-\varsigma_{1}}\Big\}\leq(2^{4}S\rho^{-1}R^{-\varsigma_{1}})^{-2}
\|\mathcal{R}_Kw\|^2_{\ell^2}\\
&\leq 2^{-8}R^{2\varsigma_{1}}\exp\{-R^{\varsigma_{3}}\},
\end{aligned}
\end{equation*}
where $\varsigma_{3}=(1+p)\sigma^{-1}-p$. By the above argument, the desired result follows directly.
\end{proof}

\subsubsection{Trigonometric polynomial approximations.}\label{trigonometricpolynomial}
Since $A \in G_{\rho}^{\nu}(\mathbb{T}, SL(2,\mathbb{R}))$, then we can write
$A(\theta)=\sum_{k\in\mathbb{Z}}\widehat{A}(k)\me^{2\pi \mi k\theta}$ with estimate
\begin{equation}\label{fou}
\begin{split}
|\widehat{A}(k)|\leq \|A\|_{\nu,\rho}\me^{-\rho|2\pi k|^{\nu}},\ \ \forall k\in\Z.
\end{split}
\end{equation}
For any $N>0,$  denote $\widetilde{N}=N^{b\nu^{-1}},$ where $b=\delta(\nu^{-1}-1)^{-1}$ and $\delta\in (0,1)$ will be fixed later.
Once we have this, we can consider the truncated cocycle $\widetilde{A}(\theta):=\mathcal{T}_{\widetilde{N}}A(\theta),$ denote by   $\widetilde{A}_N(\alpha,\theta)$ and $\widetilde{L}_N(\alpha,\widetilde{A})$ the associated transfer matrix and finite Lyapunov exponent by substituting $\widetilde{A}(\theta)$ for $A(\theta)$.  Then we have the following lemma.
\begin{lemma}\label{proposition1016}
There exists $N(\rho,\nu,\|A\|_{\nu,\rho})\in\N$, $c=c(\rho,\nu,\|A\|_{\nu,\rho})$ such that if  $N\geq N(\rho,\nu,\|A\|_{\nu,\rho})$, then we have the following estimates:
\begin{equation}\label{error}
\|A(\theta)-\widetilde{A}(\theta)\|\leq \me^{-c\widetilde{N}^{\nu}}= \me^{-cN^b},
\end{equation}
\begin{equation*}
\big|N^{-1}\ln\|A_N(\alpha,\theta)\|-N^{-1}\ln\|\widetilde{A}_N(\alpha,\theta)\|\big|\leq  \me^{-\frac{c}{2}N^b},
\end{equation*}
\begin{equation*}
\big|L_N(\alpha,A)-\widetilde{L}_N(\alpha,\widetilde{A})\big|<\me^{-\frac{c}{2}N^b}.
\end{equation*}
\end{lemma}

\begin{proof}
%

The estimate \eqref{error} follows directly from \eqref{fou}. Moreover,
by telescoping argument, for $N\geq N(\rho,\nu,\|A\|_{\nu,\rho})$ which is large enough, we have
\begin{equation*}
\begin{split}
\big\|A_N(\alpha,\theta)-\widetilde{A}_N(\alpha,\theta)\big\|\leq  (\|A\|_{\nu,\rho}+1)^N \me^{-cN^b}\leq \me^{-\frac{c}{2}N^b}.
\end{split}
\end{equation*}
 It follows that
\begin{equation*}
\begin{split}
\Big|\frac{1}{N}\ln\|A_N(\alpha, \theta)\|-\frac{1}{N}
\ln\|\widetilde{A}_N(\alpha, \theta)\|\Big|\leq\frac{1}{N}\big\|A_N(\theta)-\widetilde{A}_N(\theta)\big\|<\me^{-\frac{c}{2}N^b}.
\end{split}
\end{equation*}
By averaging, one thus has $\big|L_N(\alpha,A)-\widetilde{L}_N(\alpha,\widetilde{A})\big|<\me^{-\frac{c}{2}N^b}.$
\end{proof}

Since  $\widetilde{A}(\theta)$  is  a trigonometric polynomial, then one can analytic  continue $\widetilde{A}(\theta)$  to become an analytic function. Indeed, let
$$\rho_N=\frac{\rho}{4\pi}\widetilde{N}^{\nu-1}=\frac{\rho}{4\pi}N^{-b(\nu^{-1}-1)}:=\frac{\rho}{4\pi}N^{-\delta},$$
and set $\vartheta=\theta+\mi\widetilde{\theta}$, then $\widetilde{A}(\vartheta)$  is analytic  in the strip $|\widetilde{\theta}|\leq \rho_N$:
\begin{equation}\label{20201010}
\begin{split}
\|\widetilde{A}\|_{\rho_N}^{*}=\sum\limits_{|k|<\widetilde{N}}|\widehat{A}(k)|\me^{|2\pi k\rho_N|}
\leq \|A\|_{\nu,\rho}\sum\limits_{|k|<\widetilde{N}}\me^{-\rho|2\pi k|^{\nu}}\me^{|2\pi k|\rho_N}:=\me^{C_{1}}<\infty.
\end{split}
\end{equation}

For $\vartheta=\theta+\mi\widetilde{\theta}$
with $|\widetilde{\theta}|\leq \rho_N$, set
\begin{equation}\label{definetildeu}
\begin{split}
\tilde{u}_N(\vartheta):=\frac{1}{N}\ln\|\widetilde{A}_N(\vartheta)\|.
\end{split}
\end{equation}
In the following lemma, we will prove that $|\tilde{u}_N(\vartheta)|$
is  indeed a bounded subharmonic function in the strip $[|\mathrm{Im}\vartheta|<\rho_N].$
\begin{lemma}\label{estimatetildeu}
We have the estimate
\begin{equation}\label{definetildeus}
\begin{split}
\sup_{\theta\in\T}\sup_{|\widetilde{\theta}|\leq\rho_{N}}|\tilde{u}_N(\vartheta)|\leq \max\{\ln2,C_{1}\},
\end{split}
\end{equation}
where $C_{1}$ is the one in \eqref{20201010}.
\end{lemma}
\begin{proof}
We will prove that the analytic continuation $\widetilde{A}(\vartheta)$ in the strip $[|\mathrm{Im}\vartheta|<\rho_N]$ is not singular, which implies that $|\tilde{u}_N(\vartheta)|$
is  a bounded subharmonic function.

 We first give a estimate about $\|\widetilde{A}(\vartheta)-\widetilde{A}(\theta)\|$ as follows:
\begin{eqnarray*}
\|\widetilde{A}(\vartheta)-\widetilde{A}(\theta)\|&\leq&\sum_{0<|k|<\widetilde{N}}|\widehat{\tilde{A}}(k)|
\sup_{\theta\in\T,|\widetilde{\theta}|\leq\rho_{N}}|\me^{2\pi\mi k\theta}(\me^{-2\pi k\widetilde{\theta}}-1)|\nonumber\\
&=&\sum_{0<|k|\leq N^{\delta/2}}|\widehat{\tilde{A}}(k)|\me^{|2\pi k|\rho_{N}}
(1-\me^{-|2\pi k\rho_{N}|})\nonumber\\
&+&\sum_{N^{\delta/2}<|k|<N^{b\nu^{-1}}}|\widehat{A}(k)|\me^{|2\pi k|^{\nu}\rho}
(\me^{|2\pi k|\rho_{N}}-1)\me^{-|2\pi k|^{\nu}\rho}.
\end{eqnarray*}

To estimate the first term, note  if  $0<|k|\leq N^{\delta/2}$,  then one has $|2\pi k|\rho_{N}\leq \rho N^{-\delta/2}/2\ll1,$
which implies
\begin{equation*}
\begin{split}
1-\me^{-|2\pi k\rho_{N}|}\leq 2|2\pi k\rho_{N}|\leq  \rho N^{-\delta/2}.
\end{split}
\end{equation*}

To estimate the second term, note for all $k$ with $N^{\delta/2}<|k|<N^{b\nu^{-1}},$ one has
\begin{equation*}
\begin{split}
|2\pi k|^{\nu}\rho/2-|2\pi k|\rho_{N}=2^{-1}\rho(|2\pi k|^{\nu}-|k|N^{-\delta})
\geq0,
\end{split}
\end{equation*}
which implies
\begin{equation*}
\begin{split}
(\me^{|2\pi k|\rho_{N}}-1)\me^{-|2\pi k|^{\nu}\rho/2}<2,\ \forall N^{\delta/2}<|k|<N^{b\nu^{-1}}.
\end{split}
\end{equation*}
Consequently, one has
\begin{eqnarray}\label{thatistildeas}
\|\widetilde{A}(\vartheta)-\widetilde{A}(\theta)\| &\leq& \rho N^{-\delta/2}\|\tilde{A}\|_{\rho_{N}}^{*} +2\|A\|_{\nu,\rho}\me^{-|2\pi N^{\delta/2}|^{\nu}\rho/2} \nonumber\\
&<&2\rho N^{-\delta/2}\me^{C_{1}}.
\end{eqnarray}
To see this, one only needs to check that  $f(x)=x^{\nu}-(2\pi)^{-1}xN^{-\delta}>0$ on the interval $[2\pi, 2\pi N^{b\nu^{-1}}].$

Now we give the estimate $|\det{\widetilde{A}(\theta)}|.$ First, the inequality in \eqref{thatistildeas} implies
\begin{equation}\label{dettildeaextimate}
\begin{aligned}
|\det{\widetilde{A}(\vartheta)}-\det{\widetilde{A}(\theta)}|\leq4\|\widetilde{A}(\vartheta)\|
\|\widetilde{A}(\vartheta)-\widetilde{A}(\theta)\|\leq 8\rho \me^{2C_{1}}N^{-\delta/2}\ll1.
\end{aligned}
\end{equation}
Moreover, note $A\in SL(2,\R)$, then by Lemma \ref{proposition1016}, we have
\begin{equation}\label{n1}
\begin{split}
|1-\det{\widetilde{A}(\theta)}|\leq 8\|A(\theta)\|\|A(\theta)-\tilde{A}(\theta)\|\leq C\me^{-cN^b},
\end{split}
\end{equation}
that is $|\det{\widetilde{A}(\theta)}|\geq 1/2,\ \forall\theta\in\T,$ which, together with \eqref{dettildeaextimate}, yields
\begin{equation}\label{dettildeaextimates}
\begin{aligned}
|\det{\widetilde{A}(\vartheta)}|\geq1/4, \forall (\theta,\widetilde{\theta})\in\T\times[-\rho_{N},\rho_{N}].
\end{aligned}
\end{equation}

Once we get the inequality in \eqref{dettildeaextimates}, we are ready to estimate $|\tilde{u}_{N}(\vartheta)|.$
Indeed,
\begin{equation*}
\begin{aligned}
\|\widetilde{A}_{N}(\vartheta)\|^{2}&\geq |\det{\widetilde{A}_N(\vartheta)}|=
|\Pi_{\ell=0}^{N-1}\det{\widetilde{A}(\vartheta+\ell\alpha)}|\\
&=\Pi_{\ell=0}^{N-1}|\det{\widetilde{A}(\vartheta+\ell\alpha)}|
\geq4^{-N},
\end{aligned}
\end{equation*}
which yields $2^{-N}\leq\|\widetilde{A}_{N}(\vartheta)\|\leq\me^{NC_{1}},$
or
\begin{equation*}
\begin{aligned}
\sup_{\theta\in\T}\sup_{|\widetilde{\theta}|\leq\rho_{N}}|\tilde{u}_{N}(\vartheta)|\leq\max\{\ln2,C_{1}\}.
\end{aligned}
\end{equation*}
\end{proof}

\subsubsection{Proof of Proposition~\ref{8}}\label{thirdsubsubsection}
In this section we will give the proof of Proposition~\ref{8} by applying Proposition~\ref{80}.
First by  Lemma \ref{proposition1016}, we have
\begin{align}\label{20201101}
|\frac{1}{N}\ln\|A_N(\theta)\|-L_N(\alpha,A)|&\leq|N^{-1}\ln\|A_N(\theta)\|-\tilde{u}_N(\theta)|\nonumber\\
&+|\tilde{u}_N(\theta)-[\tilde{u}_N(\theta)]_{\theta}|+|[\tilde{u}_N(\theta)]_{\theta}-L_N(\alpha,A)|\nonumber\\
&\leq|\tilde{u}_N(\theta)-[\tilde{u}_N(\theta)]_{\theta}|+2\me^{-\frac{c}{2}N^b}.
\end{align}
 Thus we only need to estimate  $|\tilde{u}_N(\theta)-[\tilde{u}_N(\theta)]_{\theta}|,$ which is controlled by the sum
\begin{equation}\label{20201031}
\begin{split}
|\tilde{u}_N(\theta)-F_{R,p}[\tilde{u}_N](\theta)|+|F_{R,p}[\tilde{u}_N](\theta)-[\tilde{u}_N(\theta)]_{\theta}|,
\end{split}
\end{equation}
where $F_{R,p}[u](\theta)=\frac{1}{R^p}\sum\limits_{j=0}^{p(R-1)}c_R^p(j)\tilde{u}_N(\theta+j\alpha).$

In the following, we will give the estimates of the two terms above.
First we would like to bound $|\tilde{u}_N(\theta)-\tilde{u}_N(\theta+\alpha)|$.
For the function $\tilde{u}_N(\theta)$ defined by \eqref{definetildeu},
the inequality in \eqref{definetildeus} and $C_{1}>\ln 2$ imply
$\tilde{u}_N(\vartheta)$ is a bounded subharmonic function on $[|\mathrm{Im}\vartheta|<\rho_N].$
It follows that
\begin{equation}\label{20201102}
\begin{split}
\|\tilde{u}_N(\vartheta)\|\leq C_1=2^{-1}C_{2},\ C_{2}=2C_{1}.
\end{split}
\end{equation}
That is this function $\tilde{u}_{N}(\theta)$ satisfies the hypotheses in Proposition~\ref{80}.
Consequently,
\begin{equation*}
\begin{aligned}
\Big|\tilde{u}_{N}(\theta)-\tilde{u}_N(\theta+\alpha)\Big|&=\frac{1}{N}
\Big|\ln\|\widetilde{A}_N(\theta)\|-\ln\|\widetilde{A}_N(\theta+\alpha)\|\Big|\\
&=\Big|\frac{1}{N}\ln\frac{\|\widetilde{A}_N(\theta)\|}{\|\widetilde{A}_N(\theta+\alpha)\|}\Big|\\
&=\Big|\frac{1}{N}\ln\frac{\|\widetilde{A}(\theta+(N-1)\alpha)\cdots
\widetilde{A}(\theta+\alpha)\widetilde{A}(\theta)\|}{\|\widetilde{A}(\theta+N\alpha)\widetilde{A}(\theta+(N-1)\alpha)
\cdots\widetilde{A}(\theta+\alpha)\|}\Big|\\
&\leq\frac{1}{N}\ln\big(\|\widetilde{A}(\theta+N\alpha)^{-1}
\|\cdot\|\widetilde{A}(\theta)\|\big).
\end{aligned}
\end{equation*}
Thus we only need to estimate  $\|\widetilde{A}(\theta+N\alpha)^{-1}\|$. Indeed, \eqref{n1} implies
\begin{equation*}
\begin{aligned}
\|\widetilde{A}(\theta)^{-1}\|&\leq \|1/\det{\widetilde{A}(\theta)}\|\|\widetilde{A}(\theta)\|
\\
&\leq \{1+2\|I-\det{\widetilde{A}(\theta)}\|\}\|\widetilde{A}(\theta)\|
\leq2\me^{C_{1}}\leq\me^{2C_{1}}.
\end{aligned}
\end{equation*}
Once we have this we get
\begin{equation}\label{difference}
\Big|\tilde{u}_N(\theta)-\tilde{u}_N(\theta+\alpha)\Big|\leq \frac{3C_2}{2N}.
\end{equation}

For the fixed $\nu$ with $1/2<\nu<1,$ there exists $\delta\in(0,1)$ such that $\nu^{-1}<1+\delta.$
Once we fix $\nu$ and $\delta$ by this way, we set $b=\delta(\nu^{-1}-1)^{-1}>1.$ Then we
 choose $\sigma$ and $\varsigma$ in Proposition~\ref{80} as $1<\sigma<\min\{2,\delta^{-1}\}$, $\varsigma=\delta.$ That is the parameters $\sigma$ and $p>1$ in Proposition~\ref{80} satisfy
\begin{equation*}
\begin{aligned}
\delta=b(1/\nu-1)<\frac{1}{\sigma}<1,\ \  \frac{\delta\sigma}{\sigma-1}<p<\frac{1}{\sigma-1}.
\end{aligned}
\end{equation*}
Take $\gamma=1+p(1-\sigma),\sigma_1=\frac{p}{\delta}(\sigma-1)$.
It is obvious that
\begin{equation*}
\begin{aligned}
1>\gamma=1+p(1-\sigma)>0,\ \
\sigma-\sigma_1=\frac{\delta\sigma-p(\sigma-1)}{\delta}<0.
\end{aligned}
\end{equation*}
Notice that $1<\sigma<\delta^{-1}$ and $\nu^{-1}-1<\delta<1,$ thus $\delta,\sigma\rightarrow1$ as $\nu\rightarrow 1/2,$
which imply $\sigma_{1}\rightarrow\sigma, p\rightarrow\infty$ and $\gamma\rightarrow0$ as $\nu\rightarrow 1/2.$

For $q,R$ with $R=q^{\sigma}$ (as Proposition~\ref{80}),
set $\frac{9pC_2}{\kappa}q^{\sigma}<N<\big(\frac{\kappa\rho}{2^{p+6}\pi C_2}\big)^{\frac{1}{\delta}}q^{\sigma_1}$
and $K$ as the one in Proposition~\ref{80}.
Now we give the estimate of the first term in \eqref{20201031}. More concretely,
\begin{align}\label{star3}
\nonumber \Big|F_{R,p}[\tilde{u}_N](\theta)
-\tilde{u}_N(\theta)\Big|
&\leq\frac{1}{R^p}\sum\limits_{j=0}^{p(R-1)}|\tilde{u}_N(\theta+j\alpha)-\tilde{u}_N(\theta)|c_R^p(j)\\ \nonumber
&\leq\frac{1}{R^p}\sum\limits_{j=0}^{p(R-1)}c_R^p(j)\frac{j3C_2}{2N}<\frac{3p(R-1)C_2}{2N}\nonumber\\
&\leq 3p(R-1)C_2 \frac{\kappa}{18pC_2}q^{-\sigma}\leq \frac{\kappa}{6},
\end{align}
where the third inequality is by \eqref{difference}.

We will apply Proposition~\ref{80} to get the estimate of the second term in \eqref{20201031}. More concretely,
let $\varsigma_{i},i=1,2,3$ be the ones in Proposition~\ref{80} with $2^{-1}C_{2}$ and $\rho_{N}$ in place
of $S$ and $\rho,$ respectively and note $N<\big(\frac{\kappa\rho}{2^{p+6}\pi C_2}\big)^{\frac{1}{\delta}}q^{\sigma_1}$
we get
\begin{equation*}
\begin{aligned}
\varsigma_{2}R^{-\varsigma_{1}}=2^{p+5}2^{-1}C_{2}4\pi\rho^{-1}N^{\delta}q^{-p(\sigma-1)}<\kappa.
\end{aligned}
\end{equation*}
Thus by Proposition~\ref{80} we know that there is a set such that for all $\theta$
outside this set we have
\begin{equation}\label{addnewparameterss}
\begin{aligned}
\Big|F_{R,p}[\tilde{u}_N](\theta)
-[\tilde{u}_N(\theta)]_{\theta}\Big|\leq\varsigma_{2}R^{-\varsigma_{1}}<\kappa.
\end{aligned}
\end{equation}
Moreover, the measure of this set is less than
\begin{equation}\label{addnewparameters}
\begin{aligned}
2^{-8}R^{2\varsigma_{1}}\exp\{-R^{\varsigma_{3}}\}=2^{-8}q^{2p(\sigma-1)}\exp\{-q^{\gamma}\}<\exp\{-2^{-1}q^{\gamma}\}.
\end{aligned}
\end{equation}

Set $C_{1}(\kappa)=\frac{9pC_2}{\kappa},C_{2}(\kappa)=\big(\frac{\kappa\rho}{2^{p+6}\pi C_2}\big)^{\frac{1}{\delta}}$,
$c=\frac{1}{2}$ and $q\geq q_{0},$ with $q_{0}$ depending on $\kappa,\rho,\nu$ (by Lemma~\ref{proposition1016})
and sufficiently large, then by \eqref{20201101}-\eqref{addnewparameters} we finish the proof of Proposition~\ref{8}.

\subsection{Application of avalanche principle.}

\begin{proposition}[\cite{Goldsteins01,Bourgain05}]\label{avalanche}
Let $A_1$, $A_2$, $\cdots$, $A_n$ be a sequence in $SL(2,\R)$ satisfying the conditions
\begin{equation*}
\begin{aligned}
\min\limits_{1\leq j\leq n}&\|A_j\|\geq \mu>n,\\
\max\limits_{1\leq j\leq n}&\left|\ln\|A_j\|+\ln\|A_{j+1}\|-\ln\|A_{j+1}A_j\|\right|<\frac{1}{2}\ln \mu.
\end{aligned}
\end{equation*}
 Then
there exists a constant $C_{A}<\infty$ such that
\begin{equation*}
\begin{aligned}
\Big|\ln\|\prod\limits_{j=1}^nA_j\|+\sum\limits_{j=2}^{n-1}\ln\|A_j\|-\sum\limits_{j=1}^{n-1}\ln\|A_{j+1}A_j\|\Big|<C_A\frac{n}{\mu}.
\end{aligned}
\end{equation*}
\end{proposition}

Following the ideas in \cite{Bourgainj02}, in case of positive Lyapunov exponent, the large deviation theorem provides us a possibility to apply avalanche principal (Proposition \ref{avalanche}) to $A(\theta+jN\alpha)$ for $\theta$ in a set of large measure and therefore pass on to a lager scale.
\begin{lemma}\label{10}
Assume that $|\alpha-\frac{a}{q}|<\frac{1}{q^2}$, $(a,q)=1$.  Let $C_1(\kappa)q^\sigma<N<C_2(\kappa)q^{\sigma_1}$ and $q\geq q_0(\kappa)$ be the same as Proposition \ref{8}. Assume that $L_{N}(\alpha,A)>100\kappa>0$ and $L_{2N}(\alpha,A)>\frac{19}{20}L_N(\alpha,A)$. Then for $N'$ such that $m=N'N^{-1}$ satisfies $\me^{\frac{c}{2}q^{\gamma/4}}<m<\me^{\frac{c}{2}q^{\gamma}}$, we have
\begin{equation*}
\begin{aligned}
\Big|L_{N'}(\alpha,A)+L_N(\alpha,A)-2L_{2N}(\alpha,A)\Big|<C\me^{-\frac{c}{2}q^{\gamma/4}},
\end{aligned}
\end{equation*}
where $c$ is the one from the large deviation bound of Proposition \ref{8}.
\end{lemma}
\begin{proof}
We use the avalanche principal  (Proposition \ref{avalanche}) on $A_j^N(\theta)=A_N(\theta+jN\alpha)$ with $\theta$ being restricted to the set $\Omega\subset\T$, defined by $2m$ conditions:
\begin{equation*}
\begin{aligned}
\Big|\frac{1}{N}\ln\|A_j^N(\theta)\|-L_N(\alpha,A)\Big|\leq\kappa, \ \ 1\leq j\leq m,\\
\Big|\frac{1}{2N}\ln\|A_j^{2N}(\theta)\|-L_{2N}(\alpha,A)\Big|\leq\kappa,\ \  1\leq j\leq m.
\end{aligned}
\end{equation*}
By Proposition \ref{8}, we have for any $j$
\begin{equation*}
\begin{aligned}
mes\Big\{\theta:\big|\frac{1}{N}\ln\|A_j^N(\theta)\|-L_N(\alpha,A)\big|>\kappa\Big\}<\me^{-c q^\gamma},\\
mes\Big\{\theta:\big|\frac{1}{2N}\ln\|A_j^{2N}(\theta)\|-L_{2N}(\alpha,A)\big|>\kappa\Big\}<\me^{-c q^\gamma}.
\end{aligned}
\end{equation*}
Thus we have
\begin{equation}\label{meas}
\begin{aligned}
mes(\T\setminus\Omega)<2m\me^{-cq^{\gamma}}.
\end{aligned}
\end{equation}
For each $A_j^N(\theta)$ with $\theta\in \Omega$,
\begin{equation*}
\begin{aligned}
\me^{N(L_N(\alpha,A)-\kappa)}<\|A_j^N(\theta)\|<\me^{N(L_N(\alpha,A)+\kappa)}.
\end{aligned}
\end{equation*}
Note that since $L_N(\alpha,A)>100\kappa$, then
\begin{equation*}
\begin{aligned}
\|A_j^N(\theta)\|>\me^{\frac{99}{100}NL_N(\alpha,A)}:=\mu.
\end{aligned}
\end{equation*}
For large enough $q$, and hence $N$ by hypothesis, we have $\mu>2m$ (since $\sigma>1>\gamma$). Also for $j<m$, by the fact that
 $A_{j+1}^N(\theta)A_j^N(\theta)=A_j^{2N}(\theta)$, we have
\begin{equation*}
\begin{aligned}
\big|\ln\|A_j^N(\theta)\|&+\ln\|A_{j+1}^N(\theta)\|-\ln\|A_{j+1}^N(\theta)A_j^N(\theta)\|\big|\\
&<4N\kappa+2N|L_N(\alpha,A)-L_{2N}(\alpha,A)|\\
&<\frac{1}{25}NL_N(\alpha,A)+2N(\frac{1}{20}L_N(\alpha,A))<\frac{1}{2} \ln \mu,
\end{aligned}
\end{equation*}
where the second inequality follows by $L_{2N}(\alpha,A)>\frac{19}{20}L_N(\alpha,A)$.
Thus, we can apply the avalanche principal (Proposition \ref{avalanche}) for $\theta\in\Omega$ to obtain
\begin{equation*}
\begin{aligned}
\Big|\ln\|\prod\limits_{j=1}^mA_j^N(\theta)\|&+\sum\limits_{j=2}^{m-1}\ln\|A_j^N(\theta)\|
-\sum\limits_{j=1}^{m-1}\ln\|A_{j+1}^N(\theta)A_j^N(\theta)\|\Big|\\
&<C_A m/\mu<m\me^{-\frac{1}{2}NL_N(\alpha,A)}.
\end{aligned}
\end{equation*}
Integrating on $\Omega$, we get
\begin{equation*}
\begin{aligned}
\Big|\int_\Omega \ln\|A_{N'}(\theta)\|d\theta&+\sum\limits_{j=2}^{m-1}\int_{\Omega}\ln\|A_N(\theta+jN\alpha)\|d\theta\\
&-\sum\limits_{j=1}^{m-1}
\int_\Omega\ln\|A_{2N}(\theta+jN\alpha)\|d\theta\Big|<m\me^{-\frac{1}{2}NL_N(\alpha,A)},
\end{aligned}
\end{equation*}
therefore, recalling \eqref{meas} and the assumption   $N>C_1(\kappa)q^\sigma$, we have
\begin{equation*}
\begin{aligned}
\Big|L_{N'}(\alpha,A)&+\frac{m-2}{m}L_N(\alpha,A)-\frac{2(m-1)}{m}L_{2N}(\alpha,A)\Big|\\
&< mN'^{-1}\me^{-\frac{1}{2}NL_N(\alpha,A)}+C\me^{-\frac{c}{2}q^\gamma}
<C\me^{-\frac{c}{2}q^\gamma}.
\end{aligned}
\end{equation*}
It follows that
\begin{equation*}
\begin{aligned}
|L_{N'}(\alpha,A)&+L_N(\alpha,A)-2L_{2N}(\alpha,A)|\\
&<C\me^{-\frac{c}{2}q^\gamma}+2m^{-1}|L_N(\alpha,A)-L_{2N}(\alpha,A)|\\
&<C\me^{-\frac{c}{2}q^\gamma}
+L_N(\alpha,A)(10m)^{-1}<C\me^{-\frac{c}{2}q^{\gamma/4}},
\end{aligned}
\end{equation*}
where the second inequality is by $L_{2N}(\alpha,A)>\frac{19}{20}L_N(\alpha,A)$ and the last inequality is by $m>\me^{\frac{c}{2}q^{\gamma/4}}$.
\end{proof}
Actually, the condition ``$L_{2N}(\alpha,A)>\frac{19}{20}L_N(\alpha,A)$" is not necessary if $q$ is sufficiently large
and $L(\alpha,A)>0$.

\begin{lemma}\label{useful1}
Assume that $L(\alpha,A)>100\kappa>0,$ then there exists $N_0\in \N$ with
$C_1(\kappa)q_0^{\sigma}<N_0<C_2(\kappa)q_0^{\sigma_{1}}$, $q_0$ is the one defined in Proposition \ref{8} such that
\begin{equation}\label{20201011}
\begin{aligned}
L_{2N_0}(\alpha,A)>\frac{99}{100}L_{N_0}(\alpha,A).
\end{aligned}
\end{equation}
\end{lemma}
\begin{proof}
Note that for any $n$, by subadditivity, we have
\begin{equation*}
\begin{aligned}
100\kappa<L(\alpha,A)=\inf L_n(\alpha,A)\leq L_{2n}(\alpha,A)\leq L_n(\alpha,A)\leq C_{1},
\end{aligned}
\end{equation*}
where $C_{1}$ is the one in \eqref{20201010}.

Set $j_0=\big[(\ln(100/99))^{-1}\ln(C_{1}/100\kappa)\big],$ that is
\begin{equation}\label{20200926}
\begin{aligned}
(99/100)^{j_{0}+1}C_{1}<100\kappa\leq(99/100)^{j_{0}}C_{1}.
\end{aligned}
\end{equation}

Consider the sequence $\{L_{2^jN}(\alpha,A)\}$ where $N= [C_{1}(\kappa)q_{0}^{\sigma}]+1,$ $j\in \N$.   If
\begin{equation*}
\begin{aligned}
L_{2^{j+1}N}(\alpha,A)\leq (99/ 100)L_{2^jN}(\alpha,A)
\end{aligned}
\end{equation*}
hold for all $0\leq j\leq j_0,$ then
\begin{equation*}
\begin{aligned}
100\kappa <L_{2^{j_{0}+1}N}(\alpha,A)&\leq(99/ 100)^{j_{0}+1}L_{N}(\alpha,A)\leq(99/100)^{j_{0}+1}C_{1}<100\kappa,
\end{aligned}
\end{equation*}
where the last inequality is by first inequality in \eqref{20200926}. Thus
there exists $j_{*}'s$ with $0\leq j_{*}\leq j_0$ such that
\begin{equation*}
\begin{aligned}
L_{2^{j_{*}+1}N}(\alpha,A)>(99/100)L_{2^{j_{*}}N}(\alpha,A).
\end{aligned}
\end{equation*}

Moreover, since $j_{0}$ is fixed, we can set $q_{0}$ large enough such that
\begin{equation*}
\begin{aligned}
2^{j_{0}}<2^{-1}C_{1}(\kappa)^{-1}C_{2}(\kappa)q_{0}^{\sigma_{1}-\sigma}.
\end{aligned}
\end{equation*}
Set $N_{0}=2^{j_{*}}N.$  Thus we have the estimates
\begin{equation*}
\begin{aligned}
C_{1}(\kappa)q_{0}^{\sigma}
\leq N\leq N_{0}\leq2^{j_{0}}N\leq C_{2}(\kappa)q_{0}^{\sigma_{1}}
\end{aligned}
\end{equation*}
with
\begin{equation*}
\begin{aligned}
L_{2N_0}(\alpha,A)>(99/100)L_{N_0}(\alpha,A).
\end{aligned}
\end{equation*}

\end{proof}

\subsection{Inductive argument.}
Once one has Lemma \ref{10}, one can follow the induction arguments developed in \cite{Bourgainj02}. However, in our case there is an upper bound  of $N$ in the large deviation theorem. Thus we can only deal with Diophantine frequencies and their continued fraction expansions. Moreover, we need to deal with the Diophantine frequencies and their continued fraction expansions seperately. Let $p_n/q_n$ be the continued fraction expansion of $\alpha$. To apply Lemma \ref{10} inductively, we first fix $\alpha \in DC(v,\tau),$ and inductively choose the following sequences:
\begin{equation}\label{20201103}
\begin{aligned}
q_0=\tilde{q}_0<N_0<\tilde{q}_1<N_1<\cdots<N_s<\tilde{q}_{s+1}<N_{s+1}<\cdots,
\end{aligned}
\end{equation}
where $\tilde{p}_i/\tilde{q}_i$ is a subsequence of the continued fraction expansion of $\alpha$ with
\begin{equation}\label{12}
\text{$\tilde{q}_{s+1}$ is the smallest $q_{j}$ such that $\tilde{q}_{s+1}>\me^{\tilde{q}_{s}^{\gamma/2}}$},\ s\geq0,
\end{equation}
\begin{equation}\label{13}
\begin{aligned}
C_1(\kappa)\tilde{q}_{s}^\sigma<N_{s}<C_2(\kappa)\tilde{q}_{s}^{\sigma_1},\ \ \tilde{q}_{s}|N_{s},\ s\geq0,
\end{aligned}
\end{equation}
\begin{equation}\label{14}
\begin{aligned}
N_{s+1}=m_{s+1}N_{s},\ \ \me^{\frac{c}{2}\tilde{q}_{s}^{\gamma/4}}<m_{s+1}<2m_{s+1}<\me^{\frac{c}{2}\tilde{q}_{s}^{\gamma}},\ s\ge0.
\end{aligned}
\end{equation}

Actually, we can inductively select a sequence $\{N_{s}\}$ such that \eqref{12}-\eqref{14} hold, indeed the starting case $s=0$ follows from Lemma \ref{useful1}.
 First by the selection of $\tilde{q}_{s}$ and the Diophantine condition of $\alpha$, one can check that $\me^{\tilde{q}_{s}^{\gamma/2}}<\tilde{q}_{s+1}<\me^{2\tau\tilde{q}_{s}^{\gamma/2}}$.
Take $N_{s+1}=N_{s}m_{s+1}$ with $m_{s+1}:=\tilde{q}_{s+1}([\tilde{q}_{s+1}^{\sigma-1}]+1).$ It's easy to check that
\begin{equation*}
\begin{aligned}
C_1(\kappa)\tilde{q}_{s+1}^\sigma<C_1(\kappa)\tilde{q}_{s}^\sigma\tilde{q}_{s+1}^\sigma< N_{s+1}<2C_2(\kappa)\tilde{q}_{s}^{\sigma_1}\tilde{q}_{s+1}^\sigma<C_2(\kappa)\tilde{q}_{s+1}^{\sigma_1},
\end{aligned}
\end{equation*}
\begin{equation*}
\begin{aligned}
\me^{\frac{c}{2}\tilde{q}_{s}^{\gamma/4}}<\me^{\sigma\tilde{q}_{s}^{\gamma/2}}<\tilde{q}_{s+1}^\sigma<m_{s+1}<2 m_{s+1}\leq 4\tilde{q}_{s+1}^\sigma<4\me^{2\sigma\tau\tilde{q}_{s}^{\gamma/2}}<\me^{\frac{c}{2}\tilde{q}_{s}^{\gamma}}.
\end{aligned}
\end{equation*}
Thus such a choice of $N_s$ satisfies all estimates in \eqref{12}-\eqref{14} if $\tilde{q}_0$ is sufficiently large. With the help of such sequence, we can prove the following:

\begin{lemma}\label{key1}
Assume that $\alpha \in DC(v,\tau)$ and $L(\alpha,A)>100\kappa>0$. There exist $c''>0$ and
$C_1(\kappa)\tilde{q}_{0}^{\sigma}<N_0<C_2(\kappa)\tilde{q}_{0}^{\sigma_{1}}$ such that
\begin{equation*}
\begin{aligned}
|L(\alpha,A)+L_{N_0}(\alpha,A)-2L_{2N_0}(\alpha,A)|<\me^{-c''\tilde{q}_0^{\gamma/4}}.
\end{aligned}
\end{equation*}
\end{lemma}
\begin{proof}
Let $c/100<c_3<c_2<c_1<c/2$, $2C_{1}<C<\infty$, and $\tilde{q}_{-1}=0$. We use induction to show that the
 sequences $\{N_{s}\}$ and $\{\tilde{q}_{s}\},$ defined by \eqref{20201103}, additionally satisfy, for $s\geq 0$,
\begin{equation}\label{17}
\begin{aligned}
|L_{N_{s+1}}(\alpha,A)+L_{N_s}(\alpha,A)-2L_{2N_s}(\alpha,A)|<C\me^{-c_1\tilde{q}_s^{\gamma/4}},
\end{aligned}
\end{equation}
\begin{equation}\label{18}
\begin{aligned}
|L_{2N_{s+1}}(\alpha,A)-L_{N_{s+1}}(\alpha,A)|<C\me^{-c_2\tilde{q}_{s}^{\gamma/4}},
\end{aligned}
\end{equation}
\begin{equation}\label{19}
\begin{aligned}
|L_{N{s+1}}(\alpha,A)-L_{N_s}(\alpha,A)|<C\me^{-c_3\tilde{q}_{s-1}^{\gamma/4}}.
\end{aligned}
\end{equation}

We first check the case $s=0$. Fix $N_1$ satisfying \eqref{14}. We will show
\begin{equation*}
\begin{aligned}
|L_{N_1}(\alpha,A)+L_{N_0}(\alpha,A)-2L_{2N_0}(\alpha,A)|<C\me^{-c_1\tilde{q}_0^{\gamma/4}},\\
|L_{2N_1}(\alpha,A)-L_{N_1}(\alpha,A)|<C\me^{-c_2\tilde{q}_{0}^{\gamma/4}},\\
|L_{N_{1}}(\alpha,A)-L_{N_0}(\alpha,A)|<C\me^{-c_3\tilde{q}_{-1}^{\gamma/4}}=C.
\end{aligned}
\end{equation*}
In this case,  the last inequality holds automatically since  $\tilde{q}_{-1}=0$, one only needs to check the  first two inequalities.
By \eqref{20201011} and \eqref{13},\eqref{14} with $s=0$ we know that
the conditions in Lemma \ref{10} are all satisfied with $N'=N_1$, $N=N_0$ and $q=\tilde{q}_0$. Therefore, by Lemma \ref{10}, we have
\begin{align*}
|L_{N_1}(\alpha,A)+L_{N_0}(\alpha,A)-2L_{2N_0}(\alpha,A)|<C\me^{-\frac{c}{2}\tilde{q}_0^{\gamma/4}}<C\me^{-c_1\tilde{q}_0^{\gamma/4}}.
\end{align*}
On the other hand, \eqref{14} ensures one can also apply Lemma \ref{10} to $N'=2N_1$, thus we have
\begin{equation*}
\begin{aligned}
|L_{2N_{1}}(\alpha,A)+L_{N_0}(\alpha,A)-2L_{2N_0}(\alpha,A)|<C\me^{-c_1\tilde{q}_0^{\gamma/4}}.
\end{aligned}
\end{equation*}
It follows that
\begin{equation*}
\begin{aligned}
|L_{2N_{1}}(\alpha,A)-L_{N_{1}}(\alpha,A)|<2C\me^{-c_1\tilde{q}_0^{\gamma/4}}<C\me^{-c_2\tilde{q}_0^{\gamma/4}},
\end{aligned}
\end{equation*}
and we have completed the initial case $s=0.$

%

For $j\geq1,$
assume that \eqref{17}-\eqref{19} hold for all $s$ with $s\leq j-1.$
Now we consider the case $s=j.$
Fix $N_{j+1}$ satisfying \eqref{14}. By induction we have
\begin{align*}
|L_{2N_j}(\alpha,A)-L_{N_j}(\alpha,A)|<C\me^{-c_2\tilde{q}_{j-1}^{\gamma/4}}\leq C\me^{-c_2\tilde{q}_{0}^{\gamma/4}}.
\end{align*}
This implies $L_{2N_j}(\alpha,A)>(19/20)L_{N_j}(\alpha,A),$
which together with \eqref{13}, implies $N_{j}$ satisfies the two conditions of $N$ in
Lemma \ref{10} with $\tilde{q}_{j}$ in place of $q.$ Moreover, by \eqref{14}, $m_{j+1}=N_{j+1}N_{j}^{-1}$ satisfies the
estimate of $m$ in Lemma \ref{10} with $\tilde{q}_{j}$ in place of $q.$
Thus by Lemma \ref{10}, with $N'=N_{j+1},N=N_j$ and $q=\tilde{q}_{j}$ we get
\begin{align*}
|L_{N_{j+1}}(\alpha,A)+L_{N_j}(\alpha,A)-2L_{2N_j}(\alpha,A)|<C\me^{-\frac{c}{2}\tilde{q}_j^{\gamma/4}}<C\me^{-c_1\tilde{q}_j^{\gamma/4}}.
\end{align*}
Similarly, \eqref{14} ensures one can also apply Lemma \ref{10} to $N'=2N_{j+1}$, and we have
\begin{equation*}
\begin{aligned}
|L_{2N_{j+1}}(\alpha,A)+L_{N_j}(\alpha,A)-2L_{2N_j}(\alpha,A)|<C\me^{-c_1\tilde{q}_j^{\gamma/4}}.
\end{aligned}
\end{equation*}
Thus
\begin{equation*}
\begin{aligned}
|L_{2N_{j+1}}(\alpha,A)-L_{N_{j+1}}(\alpha,A)|<2C\me^{-c_1\tilde{q}_j^{\gamma/4}}<C\me^{-c_2\tilde{q}_j^{\gamma/4}},
\end{aligned}
\end{equation*}
\begin{align*}
|L_{N_{j+1}}(\alpha,A)-L_{N_j}(\alpha,A)|\leq&|L_{N_{j+1}}(\alpha,A)+L_{N_j}(\alpha,A)-2L_{2N_j}(\alpha,A)|\\
&+|2L_{2N_j}(\alpha,A)-2L_{N_j}(\alpha,A)|\\
<&C\me^{-c_1\tilde{q}_j^{\gamma/4}}+2C\me^{-c_2\tilde{q}_{j-1}^{\gamma/4}}<C\me^{-c_3\tilde{q}_{j-1}^{\gamma/4}}.
\end{align*}
That is the estimates in \eqref{17}-\eqref{19} hold for all $s\in\N.$
As a consequence of \eqref{17} with $s=0$ and \eqref{19}
\begin{equation*}
\begin{aligned}
|L(\alpha,A)&+L_{N_0}(\alpha,A)-2L_{2N_0}(\alpha,A)|\\
&\leq|L_{N_{1}}(\alpha,A)+L_{N_0}(\alpha,A)-2L_{2N_0}(\alpha,A)|\\
&\ +\sum_{s\geq1}|L_{N_{s+1}}(\alpha,A)-L_{N_s}(\alpha,A)|<\me^{-c''\tilde{q}_0^{\gamma/4}}.
\end{aligned}
\end{equation*}
\end{proof}
For the rational frequency case, we will first estimate
 the difference between $L(p_{j}/q_{j},A_{j})$ and $L_{n}(p_{j}/q_{j},A_{j})$  with $n$ much larger than $q_{j}$ (Lemma \ref{keyys}), and then
use avalanche principle estimate to get estimate of  $L_{n}(p_{j}/q_{j},A_{j})$ (Lemma \ref{keyy}).

\begin{lemma}\label{keyys}
Consider the cocycle $(p/q, A)\in \Q\times C^{0}(\T,SL(2,\R))$ with $p,q\in\N, (p,q)=1,$ and $\|A\|\leq\me^{C_{1}}.$
Set $n=mq+r, m\in\N,0\leq r<q,$ then
\begin{eqnarray*}
L_{n}(p/q,A)\leq L(p/q,A)+2n^{-1}(\ln m+qC_{1}).
\end{eqnarray*}
\end{lemma}

\begin{proof}
Set $
A_{q}:=A_{q}(\theta)=A(\theta+(q-1)p/q)A(\theta+(q-2)p/q)\cdots A(\theta).
$
For the matrix $A_{q}$
there exists a unitary $U$ such that $A_{q}=U\begin{pmatrix}
\lambda&\psi\\
0&\lambda^{-1}
\end{pmatrix}U^{-1}.$
Then for $m\in\N,$ we have $A_{q}^{m}=U\begin{pmatrix}
\lambda^{m}&R_{m}(\lambda,\psi)\\
0&\lambda^{-m}
\end{pmatrix}U^{-1}, m\geq 2$ with
\begin{equation*}
R_{m}(\lambda,\psi)=\left\{
\begin{aligned}
&\sum_{l=1}^{k}\{\lambda^{2l-1}+\lambda^{-(2l-1)}\}\psi&m=2k, k\geq1,\\
&\psi+\sum_{l=1}^{k}\{\lambda^{2l}+\lambda^{-2l}\}\psi&m=2k+1, k\geq1.
 \end{aligned}
 \right.
\end{equation*}
Thus
\begin{equation*}
\|A_{q}^{m}\|\leq \|\lambda^{m}\|+\|R(\lambda,\psi)\|\leq\|\lambda\|^{m}(1+m\|\psi\|)\leq\rho(A_{q})^{m}(1+m\exp\{qC_{1}\}),
\end{equation*}
where $\rho(A)$ stands for the spectrum radius. Note, for $n=mq+r,$
$A_{n}(\theta)=A_{r}(\theta)A_{q}^{m}(\theta),$
then by the inequality above we get
\begin{eqnarray*}
L_{n}(p/q,A)&=&\frac{1}{n}\int_{\T}\ln \|A_{n}(\theta)\|d\theta\leq \frac{1}{n}\int_{\T}\ln \|A_{q}^{m}(\theta)\|d\theta+\frac{rC_1}{n}
\nonumber\\
&\leq&\frac{1}{n}\Big\{\int_{\T}\ln \rho(A_{q})^{m}d\theta+\int_{\T}\ln (1+m\exp\{qC_{1}\})d\theta\Big\}+\frac{qC_1}{n}
\nonumber\\
&\leq& L(p/q,A)+2n^{-1}(\ln m+qC_{1}).
\end{eqnarray*}

\end{proof}

%

\begin{lemma}\label{keyy}
Assume that $\alpha\in DC(v,\tau)$ and  $\{p_{j}/q_{j}\}$ is the sequence of continued fraction expansion of
$\alpha,$ and $A_{j}\rightarrow A$ under the topology derived by $\|\cdot\|_{\nu,\rho}$-norm.
Then there exists a $j_{1}$ such that for $j\geq j_{1},$ we have
\begin{equation*}
\begin{aligned}
|L(p_{j}/q_{j},A_{j})+L_{N_0}(p_{j}/q_{j},A_{j})-2L_{2N_0}(p_{j}/q_{j},A_{j})|
<2\me^{-c''\tilde{q}_{0}^{\gamma/4}}.
\end{aligned}
\end{equation*}
\end{lemma}
\begin{remark}
Here $N_0$ and $c''$ are the ones in Lemma \ref{key1}.
\end{remark}
\begin{proof}
For the fixed $N_{0},$ note $L_{2N_0}(\cdot_{1},\cdot_{2})$ and $L_{N_0}(\cdot_{1},\cdot_{2})$ are continuous in both variables
and $L_{2N_0}(\alpha,A)>(99/100)L_{N_0}(\alpha,A)$,  $L_{N_0}(\alpha,A)>100\kappa>0$,
then there exists $j_1\in\Z^+$, such that if $j>j_1$, we have
\begin{eqnarray}
\label{rational1}  L_{N_0}(p_j/q_j,A_{j}) &>& 99 \kappa   \\
\label{rational2}  L_{2N_0}(p_j/q_j,A_{j})&>&(49/50)L_{N_0}(p_j/q_j,A_{j}).
\end{eqnarray}

For the fixed $p_{j}/q_{j}$ and the sequence $\{\tilde{q}_{\ell}\}$ defined by \eqref{20201103}, there exists $s_{j}\in \N$ such that
$\tilde{q}_{s_{j}}\leq q_j< \tilde{q}_{s_{j}+1}.$ Then
we define the same sequences $\{\tilde{q}_{\ell}\}_{\ell=0}^{s_{j}}$ and  $\{N_{\ell}\}_{\ell=0}^{s_{j}+1}$ for $p_j/q_j$ as $\alpha$ such that \eqref{12}-\eqref{14} hold.
Following Lemma \ref{key1}, we will inductively show that
\begin{equation}\label{17sss}
\begin{aligned}
|L_{N_{\ell+1}}(p_{j}/q_{j},A_{j})+L_{N_{\ell}}(p_{j}/q_{j},A_{j})-2L_{2N_{\ell}}(p_{j}/q_{j},A_{j})|<
C\me^{-c_{1}\tilde{q}_{\ell}^{\gamma/4}},
\end{aligned}
\end{equation}
\begin{equation}\label{18sss}
\begin{aligned}
|L_{2N_{\ell+1}}(p_{j}/q_{j},A_{j})-L_{N_{\ell+1}}(p_{j}/q_{j},A_{j})|<C\me^{-c_2\tilde{q}_{\ell}^{\gamma/4}},
\end{aligned}
\end{equation}
\begin{equation}\label{19sss}
\begin{aligned}
|L_{N_{\ell+1}}(p_{j}/q_{j},A_{j})-L_{N_{\ell}}(p_{j}/q_{j},A_{j})|\leq C\me^{-c_{3}\tilde{q}_{s_{\ell-1}}^{\gamma/4}}.
\end{aligned}
\end{equation}

To give the induction, the key is to apply Lemma \ref{10}, and verify that
\begin{eqnarray}
\label{rational3}     |p_{j}/q_{j}-\tilde{p}_{\ell}/\tilde{q}_{\ell}|&<&\tilde{q}_{\ell}^{-2},\\
\label{rational5}  L_{2N_\ell}(p_j/q_j,A_{j})&>&(19/20)L_{N_\ell}(p_j/q_j,A_{j}),\\
\label{rational4}  L_{N_\ell}(p_j/q_j,A_{j}) &>& 90 \kappa.
\end{eqnarray}
Indeed, by the property of continued fraction expansion, \eqref{rational3} holds for any  $0\leq \ell\leq s_{j}$, and \eqref{rational5} follows from \eqref{rational2} and \eqref{18sss}. On the other hand, if $\ell=0$, by \eqref{rational1}, \eqref{rational2} and \eqref{17sss}, we have
\begin{equation*}
\begin{aligned}
|L_{N_{1}}(&p_{j}/q_{j},A_{j})-L_{N_{0}}(p_{j}/q_{j},A_{j})|\nonumber\\
&\leq\Big|L_{N_{1}}(p_{j}/q_{j},A_{j})+L_{N_{0}}(p_{j}/q_{j},A_{j})-2L_{2N_{0}}(p_{j}/q_{j},A_{j})\Big|\nonumber\\
&\ +2|L_{2N_{0}}(p_{j}/q_{j},A_{j})-L_{N_{0}}(p_{j}/q_{j},A_{j})|\nonumber\\
&\leq C\me^{-c_{1}\tilde{q}_{0}^{\gamma/4}}
+2L_{N_{0}}(p_{j}/q_{j},A_{j})/50,
\end{aligned}
\end{equation*}
which implies that
\begin{eqnarray}\label{20s}
L_{N_{1}}(p_{j}/q_{j},A_{j})\geq 48L_{N_{0}}(p_{j}/q_{j},A_{j})/50 -
C\me^{-c_{1}\tilde{q}_{0}^{\gamma/4}}>95\kappa.
\end{eqnarray}
As for \eqref{rational4}, in case $\ell\geq 1$, by \eqref{20s} and \eqref{19sss}, one has
\begin{eqnarray*}
L_{N_{\ell+1}}(p_{j}/q_{j},A_{j}) &\geq& L_{N_1}(p_{j}/q_{j},A_{j})- \sum_{k=1}^{\ell}|L_{N_{k+1}}(p_{j}/q_{j},A_{j})-L_{N_{k}}(p_{j}/q_{j},A_{j})|\\
&>&95\kappa-\sum_{k=0}^{\ell-1}
C\me^{-c_{3}\tilde{q}_{i}^{\gamma/4}}
>90\kappa.
\end{eqnarray*}

Therefore, the iteration can be conducted $s_j$ times, and we obtain
\begin{equation}\label{20200000}
\begin{aligned}
|L_{N_{s_{j}+1}}(p_j/q_j,A_j)+L_{N_0}(p_j/q_j,A_j)-2L_{2N_0}(p_j/q_j,A_j)|
<\me^{-c''\tilde{q}_{0}^{\gamma/4}}.
\end{aligned}
\end{equation}
Moreover, by Lemma \ref{keyys}, we get
\begin{equation*}
\begin{aligned}
L_{N_{s_{j}+1}}(p_j/q_j,A_j)\leq L(p_j/q_j,A_j)+5C_{1}C_{1}(\kappa)^{-1} \tilde{q}_{s_{j}}^{-(\sigma-1)}.
\end{aligned}
\end{equation*}
The inequality above, together with \eqref{20200000} yields
\begin{equation*}
\begin{aligned}
|L(p_j/q_j,A_j)+L_{N_0}(p_j/q_j,A_j)-2L_{2N_0}(p_j/q_j,A_j)|
<2\me^{-c''\tilde{q}_{0}^{\gamma/4}}.
\end{aligned}
\end{equation*}

\end{proof}

\subsection{Proof of Theorem \ref{continuity}.}
Assume $\alpha\in DC(v,\tau)$ and $p_n/q_n$ be the continued fraction expansion of $\alpha$. Notice that since for each $N$,
$L_N(\alpha,A)$ is a continuous function in both variables, thus, $L(\alpha,A)=\inf L_N(\alpha,A)$ is upper semi-continuous, then in the
case $L(\alpha,A)=0$ we get
\begin{equation*}
\begin{aligned}
0\leq\liminf_{n\rightarrow\infty}L(p_{n}/q_{n},A_{n}) \leq\limsup_{n\rightarrow\infty}L(p_{n}/q_{n},A_{n})\leq L(\alpha,A)=0,
\end{aligned}
\end{equation*}
that is $\lim_{n\rightarrow\infty}L(p_{n}/q_{n},A_{n})=0.$
Therefore we may assume $L(\alpha,A)>100\kappa>0$.

Take $j>j_1$ and $C_1(\kappa)q_{0}^{\sigma}<N_0<C_2(\kappa)q_{0}^{\sigma_{1}}$, by Lemma \ref{key1} and Lemma \ref{keyy}, we have
\begin{equation*}
\begin{aligned}
|L(\alpha,A)+L_{N_0}(\alpha,A)-2L_{2N_0}(\alpha,A)|<\me^{-c''\tilde{q}_{0}^{\gamma/4}},
\end{aligned}
\end{equation*}
and
\begin{equation*}
\begin{aligned}
|L(p_{j}/q_{j},A_j)+L_{N_0}(p_{j}/q_{j},A_j)-2L_{2N_0}(p_{j}/q_{j},A_j)|<2\me^{-c''\tilde{q}_0^{\gamma/4}}.
\end{aligned}
\end{equation*}
Hence, one can estimate
\begin{equation*}
\begin{aligned}
|L(\alpha,A)&-L(p_{j}/q_{j},A_j)|\leq |L_{N_0}(p_{j}/q_{j},A_j)-L_{N_0}(\alpha,A)|\\
&\ +2|L_{2N_0}(p_{j}/q_{j},A_j)-L_{2N_0}(\alpha,A)|
+3\me^{-c''\tilde{q}_{0}^{\gamma/4}}\\
&\leq C(\kappa)^{N_{0}}\{|p_{j}/q_{j}-\alpha|+\|A-A_{j}\|_{\nu,\rho}\}
+3\me^{-c''\tilde{q}_{0}^{\gamma/4}},
\end{aligned}
\end{equation*}
it follows that
\begin{equation*}
\begin{aligned}
\limsup\limits_{j\rightarrow\infty}|L(\alpha,A)-L(p_{j}/q_{j},A_j|\leq 4\me^{-c''\tilde{q}_{0}^{\gamma/4}},
\end{aligned}
\end{equation*}
let $\tilde{q}_0\rightarrow \infty$, we get the result.

\section*{Appendix: Proof of Lemma~\ref{caiyouzhou19}}

Define $B_{r}(\delta)=\{Y\in \mathcal{B}_{r}^{(nre)}:\|Y\|_{r}\leq \delta\}$ and
set $\varepsilon=8^{-2}\gamma^{2}Q_{n+1}^{-2\tau^{2}}.$ Then we define the nonlinear functional
\begin{equation*}
\begin{split}
\mathcal{F}:B_{r}(\varepsilon^{1/2}) \rightarrow\mathcal{B}_{r}^{(nre)}
\end{split}
\end{equation*}
by
\begin{equation}\label{constructoperator}
\begin{split}
\mathcal{F}(Y)=\mathbb{P}_{nre}\ln (\me^{A^{-1} Y(\theta+\alpha)A}\me^{g}\me^{-Y}),
\end{split}
\end{equation}
where $\mathbb{P}_{nre}$ denotes the standard projections from $\mathcal{B}_{r}$ to $\mathcal{B}_{r}^{(nre)}.$

We will find a solution of functional equation
\begin{equation}\label{constructoperators}
\begin{split}
\mathcal{F}(Y_{t})=(1-t)\mathcal{F}(Y_{0}),\ Y_{0}=0.
\end{split}
\end{equation}

The derivative of $\mathcal{F}$ at $Y\in B_{r}(\varepsilon^{1/2})$ along $Y'\in \mathcal{B}_{r}^{(nre)}$ is given by
\begin{equation}\label{operatorderivative}
\begin{split}
D\mathcal{F}(Y)Y'=\mathbb{P}_{nre}\big\{A^{-1} Y'(\theta+\alpha)A-Y'
+O(\|A\|^{2}g)Y'+P[A,Y,Y',g](\theta)\big\},
\end{split}
\end{equation}
where
\begin{equation*}
\begin{split}
P[A,Y,Y',g](\theta)&=
O(A^{-1}Y(\theta+\alpha)A)A^{-1} Y'(\theta+\alpha)A+2^{-1}[Y''',F+H]+\cdots\\
&-O(Y)Y'+2^{-1}[F+H',-Y'']+\cdots-O(\|A\|^{2}g)Y',\\
O(\|A\|^{2}g)Y'&=O(g)A^{-1} Y'(\theta+\alpha)A+O(g)Y',
\end{split}
\end{equation*}
with $P[A,Y,Y',0](\theta)=0,$
\begin{equation*}
\begin{split}
Y'''(\theta+\alpha)&=A^{-1}Y'(\theta+\alpha)A+O(A^{-1}Y(\theta+\alpha)A)
A^{-1}Y'(\theta+\alpha)A,\\
Y''(\theta)&=Y'(\theta)+O(Y(\theta))Y'(\theta),\\
F(\theta)&=A^{-1}Y(\theta+\alpha)A+g(\theta)-Y(\theta),
\end{split}
\end{equation*}
and $H, H'$ being sums of terms at least 2 orders in
$A^{-1}Y(\theta+\alpha)A,g(\theta),-Y(\theta).$
Moreover, the first $``\cdots"$ denotes the sum of terms which are at least 2 orders in $F+H$ but only 1 order in $Y'''.$ The second $``\cdots"$ denotes the sum of terms which are at least 2 orders in $F+H'$  but only 1 order in $Y''.$

We give a estimate about the operator $D\mathcal{F}(Y)^{-1}$.
\begin{proposition}
For the fixed $Y\in\mathcal{B}_{r}(\varepsilon^{1/2}),$ $D\mathcal{F}(Y)$ (defined by \eqref{operatorderivative})
is a linear map from $\mathcal{B}_{r}^{(nre)}$ to $\mathcal{B}_{r}^{(nre)}$ with estimate
\begin{equation}\label{20201006}
\begin{split}
\|D\mathcal{F}(Y)^{-1}\|\leq 2^{-1}\varepsilon^{-1/2}.
\end{split}
\end{equation}
\end{proposition}
\begin{proof}
For the fixed $Y\in\mathcal{B}_{r}(\varepsilon^{1/2}),$ obviously, the operator $D\mathcal{F}(Y)$ defined by \eqref{operatorderivative}
is a linear map from $\mathcal{B}_{r}^{(nre)}$ to $\mathcal{B}_{r}^{(nre)}$. In the following we prove the estimate in \eqref{20201006}. To this end, we consider the operator $D\mathcal{F}(0)$ given by
\begin{equation*}
\begin{split}
D\mathcal{F}(0)Y'=A^{-1} Y'(\theta+\alpha)A-Y'+\mathbb{P}_{nre}O(\|A\|^{2}g)Y', Y'\in \mathcal{B}_{r}^{(nre)}.
\end{split}
\end{equation*}
Note the operator $D\mathcal{F}(0)$ is a linear map mapping $\mathcal{B}_{r}^{(nre)}$ to $\mathcal{B}_{r}^{(nre)}$.
Next we give the estimate about $D\mathcal{F}(0)^{-1}.$

Note $\overline{Q}_{n+1}\geq T>(2\gamma^{-1})^{2\tau},n\geq0$ (\eqref{qnestimate}),
then by \eqref{diophantine} in Lemma~\ref{bettersmalldivisor} we get
\begin{equation*}
\begin{split}
\|k\alpha\pm2\rho_{f}\|_{\mathbb{Z}}\geq \gamma Q_{n+1}^{-\tau^{2}}= 8\varepsilon^{\frac{1}{2}}, |k|< \overline{Q}_{n+1}^{\frac{1}{2}}.
\end{split}
\end{equation*}
By the inequality above, one can check, for $Y'\in\mathcal{B}_{r}^{(nre)},$
\begin{equation*}
\begin{split}
A^{-1}&Y'(\cdot+\alpha)A-Y'\in\mathcal{B}_{r}^{(nre)},\\
\|A^{-1}&Y'(\cdot+\alpha)A-Y'\|_{r}
\geq
8\|A\|^{2}\varepsilon^{\frac{1}{2}}\|Y'\|_{r}.
\end{split}
\end{equation*}
Moreover, by Lemma~\ref{banachalgebra} (Banach algebra property) we get $\|O(\|A\|^{2}g)\|_{r}\leq 2\|A\|^{2}\varepsilon.$
Note $(8\varepsilon^{1/2})^{-1} 2\varepsilon<1,$ then
\begin{equation}\label{operatorestimate}
\begin{split}
\|D\mathcal{F}(0)^{-1}\|_{r}\leq2(8\varepsilon^{1/2})^{-1}=4^{-1}\varepsilon^{-1/2}.
\end{split}
\end{equation}

Once we get \eqref{operatorestimate}, we will turn to $\|D\mathcal{F}(Y)^{-1}\|.$
The calculations below also depends on Lemma~\ref{banachalgebra}, we omit the reference about this lemma.

Note $\{D\mathcal{F}(Y)-D\mathcal{F}(0)\}Y'=\mathbb{P}_{nre}\big\{P(A,Y,Y'g)-P(A,0,Y'g)\},$
we get
\begin{equation}\label{operatorestimates}
\begin{split}
\sup_{\|Y\|_{r}\leq\varepsilon^{\frac{1}{2}},\|g\|_{r}\leq\varepsilon
}\|D\mathcal{F}(Y)-D\mathcal{F}(0)\|\leq 2\varepsilon^{\frac{1}{2}}.
\end{split}
\end{equation}
\eqref{operatorestimate} and \eqref{operatorestimates} yield
\begin{equation*}
\begin{split}
\|D\mathcal{F}(0)^{-1}
(D\mathcal{F}(Y)-D\mathcal{F}(0))\|\leq4^{-1}\varepsilon^{-1/2}2\varepsilon^{\frac{1}{2}}
=2^{-1}.
\end{split}
\end{equation*}
Finally, note
\begin{equation*}
\begin{split}
D\mathcal{F}(Y)^{-1}=\{1+D\mathcal{F}(0)^{-1}(D\mathcal{F}(Y)-D\mathcal{F}(0))\}^{-1}D\mathcal{F}(0)^{-1},
\end{split}
\end{equation*}
then by the inequality above we know that $D\mathcal{F}(Y)$ is invertible with
\begin{equation*}
\begin{split}
\|D\mathcal{F}(Y)^{-1}\|\leq 2\|D\mathcal{F}(0)^{-1}\|\leq 2^{-1}\varepsilon^{-1/2}.
\end{split}
\end{equation*}
\end{proof}

Now we turn to functional equation \eqref{constructoperators}. Formally, we get
\begin{equation}\label{20201007}
\begin{split}
Y_{t}=-\int_{0}^{t}D\mathcal{F}(Y_{s})^{-1}\mathcal{F}(Y_{0})ds
=-\int_{0}^{t}D\mathcal{F}(Y_{s})^{-1}\mathbb{P}_{nre} gds, 0\leq t\leq1.
\end{split}
\end{equation}
Moreover, by \eqref{20201006}
\begin{equation*}
\begin{split}
\|Y_{t}\|_{r}\leq\sup_{Y\in \mathcal{B}_{r}(\varepsilon^{1/2}),\|g\|_{r}\leq\varepsilon}\|D\mathcal{F}(Y)^{-1}\|\|g\|_{r}
\leq2^{-1}\varepsilon^{-1/2}\varepsilon<\varepsilon^{1/2},0\leq t\leq1.
\end{split}
\end{equation*}
Therefore, the solution of \eqref{constructoperators} exists in $\mathcal{B}_{r}(\varepsilon^{1/2})$ and is given by \eqref{20201007}.

For $Y_{t},\ 0\leq t\leq1,$ given above,
we know that $\mathcal{F}(Y_{1})=0,$ that is (by \eqref{constructoperator})
\begin{equation*}
\begin{split}
\mathbb{P}_{nre}\ln (\me^{A^{-1} Y_{1}(\theta+\alpha)A}\me^{g}\me^{-Y_{1}})=0,
\end{split}
\end{equation*}
which implies that there exists a matrix $g^{(re)}\in\mathcal{B}_{r}^{(re)}$ such that
\begin{equation*}
\begin{split}
\ln (\me^{A^{-1} Y_{1}(\theta+\alpha)A}\me^{g}\me^{-Y_{1}})=g^{(re)}.
\end{split}
\end{equation*}
That is
\begin{equation*}
\begin{split}
\me^{ Y_{1}(\theta+\alpha)}A\me^{g}\me^{-Y_{1}}=A\me^{g^{(re)}}.
\end{split}
\end{equation*}
By standard calculations we get the estimate $\|g^{(re)}\|\leq2\varepsilon.$
This $Y_{1},$ with the estimate $\|Y_{1}\|_{r}<\varepsilon^{\frac{1}{2}},$ is the one we want.

\section*{Acknowledgments}
J. You and Q. Zhou were supported by
 National Key R\&D Program of China (2020YFA0713300) and Nankai Zhide Foundation.
H. Cheng was supported by NSFC grant (12001294). She would like to  thank Y. Pan for useful discussion. L. Ge  was partially supported by NSF DMS-1901462 and AMS-Simons Travel Grant 2020–2022. J. You was also partially supported by NSFC grant (11871286).
 Q. Zhou was also partially supported by NSFC grant (12071232), The Science Fund for Distinguished Young Scholars of Tianjin (No. 19JCJQJC61300).

\bibliographystyle{alpha}
\bibliography{2021-ultradifferentiable}
\end{document}